\documentclass[a4paper, 10pt]{amsart}
\usepackage{geometry}
\usepackage{hyperref}
\usepackage{mathrsfs}                    
\usepackage{mathabx}              

\usepackage{graphicx}
\usepackage{amssymb}
\usepackage{epstopdf}
\usepackage{fullpage}
\usepackage[all]{xy}
\usepackage{xcolor}
\newcommand{\Mod}[1]{\mathrm{Mod}(#1)}
\renewcommand{\mod}[1]{\mathrm{mod}(#1)}
\newcommand{\Qcoh}[1]{\mathrm{Qcoh}(#1)}

\newcommand{\D}[1]{\mathrm{D}(#1)}
\newcommand{\Db}[1]{\mathrm{D}^b(\mod{#1})}
\newcommand{\Hom}{\mathrm{Hom}}
\newcommand{\End}{\mathrm{End}}
\newcommand{\Ext}{\mathrm{Ext}}
\renewcommand{\t}{\mathbf{t}}
\newcommand{\inc}{\mathsf{inc}}
\newcommand{\fp}{\mathrm{fp}}
\newcommand{\id}{\mathrm{id}}

\newcommand{\cone}[1]{\mathrm{cone}(#1)}
\renewcommand{\ker}[1]{\mathrm{Ker}(#1)}
\newcommand{\coker}[1]{\mathrm{Coker}(#1)}
\newcommand{\im}[1]{\mathrm{Im}(#1)}
\newcommand{\essim}[1]{\mathsf{Im}(#1)}

\newcommand{\gen}[1]{\mathrm{gen}(#1)}
\newcommand{\Gen}[1]{\mathrm{Gen}(#1)}
\newcommand{\Cogen}[1]{\mathrm{Cogen}(#1)}

\newcommand{\Prod}[1]{\mathrm{Prod}(#1)}
\newcommand{\Add}[1]{\mathrm{Add}(#1)}
\newcommand{\add}[1]{\mathrm{add}(#1)}
\newcommand{\Inj}[1]{\mathrm{Inj}\left(#1\right)}
\newcommand{\Proj}[1]{\mathrm{Proj}\left(#1\right)}
\newcommand{\filt}[1]{\mathrm{filt}(#1)}
\newcommand{\Filt}[1]{\mathrm{Filt}(#1)}

\newcommand{\Spec}[1]{\mathrm{Spec}(#1)}
\newcommand{\Ind}[1]{\mathrm{Ind}(#1)}

\newcommand{\tors}[1]{\mathrm{tors}(#1)}
\newcommand{\Cosilt}[1]{\mathrm{Cosilt}(#1)}
\newcommand{\CosiltP}[1]{\mathrm{Cosilt}_*(#1)}

\newcommand{\tube}{\ensuremath{\mathbf{t}}}

\newcommand{\Acal}{\mathcal{A}}
\newcommand{\Bcal}{\mathcal{B}}
\newcommand{\Ccal}{\mathcal{C}}
\newcommand{\Dcal}{\mathcal{D}}
\newcommand{\Ecal}{\mathcal{E}}
\newcommand{\Fcal}{\mathcal{F}}
\newcommand{\Gcal}{\mathcal{G}}
\newcommand{\Hcal}{\mathcal{H}}
\newcommand{\Ical}{\mathcal{I}}

\newcommand{\Lcal}{\mathcal{L}}
\newcommand{\Mcal}{\mathcal{M}}
\newcommand{\Ncal}{\mathcal{N}}
\newcommand{\Pcal}{\mathcal{P}}
\newcommand{\Qcal}{\mathcal{Q}}
\newcommand{\Rcal}{\mathcal{R}}
\newcommand{\Scal}{\mathcal{S}}
\newcommand{\Tcal}{\mathcal{T}}
\newcommand{\Ucal}{\mathcal{U}}
\newcommand{\Vcal}{\mathcal{V}}
\newcommand{\Wcal}{\mathcal{W}}
\newcommand{\Xcal}{\mathcal{X}}
\newcommand{\Ycal}{\mathcal{Y}}
\newcommand{\Zcal}{\mathcal{Z}}

\newcommand{\Dscr}{\mathscr{D}}
\newcommand{\Escr}{\mathscr{E}}

\newcommand{\Pscr}{\mathscr{P}}

\newtheorem{theorem}{Theorem}[section]
\newtheorem{lemma}[theorem]{Lemma}
\newtheorem{corollary}[theorem]{Corollary}
\newtheorem{proposition}[theorem]{Proposition}

\newtheorem*{theorem*}{Theorem}
\newtheorem*{theoremA}{Theorem A}
\newtheorem*{theoremB}{Theorem B}
\newtheorem*{theoremC}{Theorem C}

\theoremstyle{definition}
\newtheorem{definition}[theorem]{Definition}
\newtheorem{setup}[theorem]{Setup}

\newtheorem{example}[theorem]{Example}

\theoremstyle{remark}
\newtheorem{remark}[theorem]{Remark}

\title{Mutation and torsion pairs}
\author{Lidia Angeleri H\"ugel, Rosanna Laking, Jan {\v S}{\v{t}}ov{\'{\i}}{\v{c}}ek,  Jorge Vit\'oria}
\address{Lidia Angeleri H\"ugel, Dipartimento di Informatica - Settore di Matematica, Universit\`a degli Studi di Verona, Strada le Grazie 15 - Ca' Vignal, I-37134 Verona, Italy} 
\email{lidia.angeleri@univr.it}
\address{Rosanna Laking, Dipartimento di Informatica - Settore di Matematica, Universit\`a degli Studi di Verona, Strada le Grazie 15 - Ca' Vignal, I-37134 Verona, Italy} 
\email{rosanna.laking@univr.it}
\address{Jan {\v S}{\v{t}}ov{\'{\i}}{\v{c}}ek, Charles University, Faculty of Mathematics and Physics, Department of Algebra, Sokolovsk\'a 83, 186 75 Praha, Czech Republic}
\email{stovicek@karlin.mff.cuni.cz}
\address{Jorge Vit\'oria, Dipartimento di Matematica e Informatica, Universit\'a degli Studi di Cagliari, Palazzo delle Scienze, Via Ospedale, 72, 09124, Cagliari, Italy}
\email{jorge.vitoria@unica.it}
\thanks{Acknowledgments:
The second-named author was supported by the European Union's Horizon 2020 research and innovation programme under the Marie Sklodowska-Curie Grant Agreement No. 797281.
The third-named author was supported by the Czech Science Foundation grant number 20-13778S.
The authors also acknowledge support from the project \textit{REDCOM: Reducing complexity in algebra, logic, combinatorics}, financed by the programme  \textit{Ricerca Scientifica di Eccellenza 2018} of the Fondazione Cariverona.}
\date{\today}                                           
\begin{document}
\begin{abstract}
Mutation of compact silting objects is a fundamental operation in the representation theory of finite-dimensional algebras due to its connections to cluster theory and to the lattice of torsion pairs in module or derived categories. In this paper  we develop a  theory of mutation in the broader framework of silting or
cosilting t-structures in triangulated categories. We show that mutation of pure-injective cosilting objects encompasses the classical concept of mutation for compact silting complexes. As an application we prove that any minimal inclusion
of torsion classes in the category of finitely generated modules over
an artinian ring corresponds to an irreducible mutation. This generalises a well-known result for functorially finite torsion classes.
\end{abstract} 
\maketitle

\tableofcontents

\section{Introduction}

The operation of mutation has a long history in representation theory and algebraic geometry. The aim of a mutation is to create a new object from an old one, changing a designated part of it and keeping the other part. This operation has been around for at least thirty years, for example in the study of exceptional collections of sheaves (\cite{GR}) or in the combinatorial study of tilting modules (\cite{RS,HU}). Mutation also plays a central role in the foundations of cluster theory which were set up by  Fomin and Zelevinsky from the early 2000's. Their work, together with the categorification results of \cite{BMRRT}, reinforced the importance of mutation in contemporary representation theory. 

As shown by Aihara and Iyama  in \cite{AI}, the right framework to study mutation in derived categories of finite-dimensional algebras is provided by silting complexes, a significant generalisation of tilting modules which was introduced in \cite{KV}. Indeed, the categorification of cluster algebras allows us to interpret clusters as silting complexes and cluster mutation as an operation that produces a new silting complex from a given one by exchanging a summand. It turns out that mutation of silting objects also provides a deep insight into the structure of the lattice $\tors A$ of torsion pairs in the category $\mod A$ of finite-dimensional modules over a finite-dimensional algebra $A$. In \cite{AIR} it is shown that minimal inclusions of functorially finite torsion classes correspond, in a suitable way, to a mutation operation on associated silting complexes. The following question was the initial motivation for this paper.\\

\noindent\textbf{Question:} Do minimal inclusions of arbitrary torsion classes correspond to an operation of mutation?\\

In order to answer this question, we are forced to leave the realm of finite-dimensional modules, but fortunately we can benefit from well-developed model theoretic methods~\cite{Prest,L}.
We approach the above question via cosilting theory, where small silting complexes are replaced by pure-injective cosilting objects. More recently, this concept has also appeared in the literature under the name derived injective (or just injective) object and turned out to be indispensable in various other contexts: in spectral algebraic geometry~\cite[Appendix C.5.7]{LurSAG}, deformation theory of dg categories~\cite{GLVdB} or derived commutative algebra~\cite{Sha}.

Our starting point here is an observation based on \cite{CBlfp,BZ,WZ}: every ``small'' torsion pair in $\mod A$ corresponds bijectively to a ``large'' torsion pair in the category $\Mod A$ of all $A$-modules, and the latter is determined by a large (i.e.~not necessarily compact) pure-injective two-term cosilting complex $\sigma$. In other words, while two-term silting complexes in the sense of \cite{AIR} only detect functorially finite torsion pairs in $\mod A$, two-term cosilting complexes detect all torsion pairs in $\mod A$. In fact, pure-injective cosilting complexes give rise to a class of t-structures in the derived category that encompasses those associated to compact silting complexes.\\

\noindent \textbf{Summary of main results.} We first develop a general framework to study mutation of silting and cosilting objects in triangulated categories, without any compactness or pure-injectivity assumptions. Our Definitions~\ref{def:mut} and~\ref{def:siltmut} extend the concept of mutation in \cite{AI} to the non-compact case, adopting the notion of (co)silting from \cite{PV,NSZ}. 

Every cosilting object $\sigma$ gives rise to a t-structure $\mathbb{T}_\sigma$ and to an abelian category $\Hcal_\sigma$, the heart of $\mathbb{T}_\sigma$. Moreover, every subset $\Escr$ of $\Prod{\sigma}$ induces a set of injectives ${H^0_\sigma(\mathscr{E})}$  in $\Hcal_\sigma$ and thus a hereditary torsion pair $(\Scal, \Rcal)=({}^{\perp_0}H^0_\sigma(\mathscr{E}),\,\Cogen{H^0_\sigma(\mathscr{E})})$. It turns out that such  torsion pairs control the process of mutation via HRS-tilting.

\begin{theoremA}[Theorem~\ref{thm:mutation as HRS-tilt}]
Let $\Dcal$ be a triangulated category with products, and let $\sigma$ and $\sigma'$ be two cosilting objects in $\Dcal$. Denote $\Escr=\Prod{\sigma}\cap\Prod{\sigma'}$, and 
let $(\Scal, \Rcal)$ be the torsion pair in $\Hcal_\sigma$ cogenerated by~$H^0_\sigma(\Escr)$.
\begin{enumerate}
\item $\sigma'$ is a right mutation of $\sigma$ if and only if $\sigma$ admits an $\Escr$-precover and
$\mathbb{T}_{\sigma'}$ is the right HRS-tilt of $\mathbb{T}_{\sigma}$ at the torsion pair $(\Scal,\Rcal)$ in $\Hcal_\sigma$. 
\item $\sigma'$ is a left mutation of $\sigma$  if and only if the  class $\Scal$ is closed under products (hence a TTF-class) in $\Hcal_\sigma$ and $\mathbb{T}_{\sigma'}$ is the left HRS-tilt of $\mathbb{T}_{\sigma}$ at the torsion pair $(\Tcal, \Scal)$ in $\Hcal_\sigma$. 
\end{enumerate}
\end{theoremA}

An important difference with respect to the theory developed in \cite{AI} is that, even for finite-dimensional algebras, mutation of large silting or cosilting objects is not always possible (Examples~\ref{expl:krone} and~\ref{expl:krone3}). Nevertheless, we show that cosilting mutation is, in itself, a generalisation of compact silting mutation. Indeed, any mutation of a compact silting object corresponds to a mutation of a pure-injective cosilting object  (see Theorem~\ref{thm:compact silting mutation}). The class of pure-injective cosilting objects, which includes bounded cosilting complexes in the derived category of a ring (see Example \ref{Ex: bdd pure-inj}), will play a distinguished role. The big advantage of working with a pure-injective cosilting object $\sigma$ lies in the fact that the heart $\Hcal_\sigma$ is a Grothendieck category (\cite{AMV3}) --- and in this case we can characterise the \textbf{existence of mutations} as follows.

\begin{theoremB}[Theorem~\ref{existence right mutation}]
Let $\sigma$ be a pure-injective cosilting object in a compactly generated triangulated category $\Dcal$, and $\Escr=\Prod\Escr$ a  subcategory of $\Prod{\sigma}$. Let $(\Scal,\Rcal)$ be the torsion pair in $\Hcal_\sigma$ cogenerated by~$H^0_\sigma(\Escr)$. The following statements are equivalent.
\begin{enumerate}
\item $\sigma$ admits a  right mutation $\sigma'$  with respect to $\Escr$.
\item The torsion-free class $\Rcal=\Cogen{H^0_\sigma(\Escr)}$ in $\Hcal_\sigma$ is closed under direct limits.
\item The cosilting object $\sigma$ admits an $\Escr$-cover.
\end{enumerate}

\smallskip \noindent
Dually, the following statements are equivalent.
\begin{enumerate}
\item $\sigma$ admits a   left mutation $\sigma'$  with respect to $\Escr$.
\item The torsion class $\Scal={}^{\perp_0}H^0_\sigma(\Escr)$ in $\Hcal_\sigma$ is closed under products (that is, it is a TTF class).
\item The object $\varepsilon_0\oplus \varepsilon_1$  arising from an $\Escr$-envelope 
 $\sigma\to \varepsilon_0$ and its cone $\varepsilon_1$ is a cosilting object.
\end{enumerate}
In both cases, if the equivalent conditions are satisfied, any  mutation $\sigma'$ as in \emph{(1)} is pure-injective.
\end{theoremB}

We also provide an interpretation of mutation of pure-injective cosilting objects in terms of localisation theory (see Section~\ref{sec:local}), which essentially states that the operation of mutation can be understood as a three-step process: first restrict the associated t-structures to certain subcategories; then shift one of the restricted t-structures; finally glue them back together.

The whole machinery of cosilting mutation in triangulated categories leads to an \textbf{answer to our motivating question} concerning the study of the lattice of torsion pairs. In \cite{DIRRT} and \cite{BCZ}, it was shown that minimal inclusions of torsion classes (not necessarily functorially-finite) in $\mod A$, for a finite-dimensional algebra $A$, are parametrised by bricks. These bricks turn out to correspond to certain indecomposable summands of the associated (two-term) cosilting objects. We show that minimal inclusions of torsion classes then correspond to swapping precisely this indecomposable summand. We call this irreducible mutation. This result generalises the phenomenon that is well-understood for minimal inclusions of functorially-finite torsion classes. While we prove the theorem below in a more general setting (see Setup \ref{setup: length} and Corollary \ref{arrows}), here we state it in the setting of artinian rings. 

\begin{theoremC}[Corollary \ref{arrows}]
Let $A$ be an artinian ring. Consider two cosilting torsion pairs $\mathfrak{t}=(\Tcal,\Fcal)$ and $\mathfrak{u}=(\Ucal,\Vcal)$ in $\mod{A}$ such that $\Ucal\subseteq \Tcal$. Let  $\sigma_{\mathfrak t}$ and $\sigma_{\mathfrak u}$ be the two-term cosilting complexes associated to $\mathfrak{t}$ and $\mathfrak{u}$, respectively. Then the following statements are equivalent.
\begin{enumerate}
\item $\sigma_{\mathfrak t}$ is an {irreducible right mutation} of $\sigma_\mathfrak{u}$.
\item $\sigma_\mathfrak{u}$ is an irreducible left mutation of $\sigma_\mathfrak{t}$.
\item The class $\Scal=\Tcal\cap\Vcal$ coincides with $\filt{M}$ for a brick $M$ in $\mod{A}$.
\item The inclusion $\Ucal\subseteq \Tcal$ is a minimal inclusion of torsion classes.
\end{enumerate}
\end{theoremC}

\smallskip

In forthcoming work \cite{AHL, ALS}, we will specialise to the case of two-term cosilting complexes in the derived category of a finite-dimensional algebra $A$ in order to obtain a more explicit description of the operation of mutation involving the approximation theory and the Ziegler spectrum of $\Mod{A}$. 

\medskip

\noindent \textbf{Structure of the paper.} The paper is organised as follows.

\begin{itemize}
\item In Section~\ref{sec:prelim} we collect the necessary background on t-structures, HRS-tilts, silting and cosilting objects.

\item In Section~\ref{sec:mutation} we define cosilting mutation and interpret these operations in terms of HRS-tilts. In particular, we prove Theorem~A above (Theorem~\ref{thm:mutation as HRS-tilt}).

\item In Section~\ref{sec:mutation and purity} we specialize to pure-injective cosilting objects in compactly generated triangulated categories. In this setting the situation becomes more tractable, and we are able to provide necessary and sufficient conditions of the existence of cosilting mutation, proving Theorem~B (Theorem~\ref{existence right mutation}). 

\item Section~\ref{sec:silting mutation} is devoted to the dual situation of silting objects and we prove there that cosilting mutation generalises the notion of mutation for compact silting objects, as set up in \cite{AI} (see Theorem~\ref{thm:compact silting mutation}).

\item In Section~\ref{sec:local}, we interpret mutation from the point of view of localisation theory of triangulated categories and we observe that the operation of mutation can be broken into 3 parts: restriction, shifting and gluing (see Theorem~\ref{existence right mutation2}). 

\item In Section~\ref{sec:mutation and torsion pairs} we consider general mutation of cosilting objects associated with torsion pairs in the heart of a pure-injective cosilting t-structure and we characterise  this situation in terms of wide subcategories of the heart (Theorem~\ref{thm: mutation general}).

\item Section~\ref{Sec Mut in small} clarifies the bijection between torsion pairs in ``small'' and ``large'' triangulated categories and describes mutation of ``small'' torsion pairs (Theorem~\ref{Thm: mutations are wide}).

\item In the final Section~\ref{sec:mutation and simples}, we focus on mutation of torsion pairs in abelian length categories. It is in this framework that we prove, in particular, Theorem~C (Corollary~\ref{arrows}).
\end{itemize}

\section{Preliminaries} \label{sec:prelim}
\subsection{Notation}
All subcategories considered are strict and full. Given an object  $X$ in an additive category $\Acal$, we denote by $\Add{X}$ the subcategory whose objects are summands of existing coproducts of copies of $X$. Dually, we write $\Prod{X}$ for the subcategory whose objects are summands of existing products of copies of $X$.  We denote the isomorphism classes of indecomposable objects in $\Acal$ by $\Ind\Acal$.

 Let $\Xcal$ be a class of objects in a complete and cocomplete abelian category $\Acal$. We denote by  $\Gen{\Xcal}$ (respectively, $\mathrm{gen}(\Xcal)$)  the subcategory formed by all epimorphic images of coproducts (respectively, of finite coproducts) of objects in $\Xcal$, and by  $\Cogen{\Xcal}$ the subcategory formed by all subobjects of products of objects in $\Xcal$. We  further write $\varinjlim\Xcal$ for the subcategory of $\Acal$ formed by direct limits of objects in $\Xcal$ and {$\fp\Xcal$} for the collection of finitely presented objects in $\Xcal$.  
Furthermore, we denote by $\Filt\Xcal$ the class of all  objects $M$ which admit an ascending chain $(M_\lambda, \lambda \leq \mu)$ of subobjects indexed over an ordinal number $\mu$ where $M_0 = 0$,  all consecutive factors $M_{\lambda+1}/M_\lambda$ with $\lambda<\mu$ belong to $\mathcal X$, and  $M=\bigcup _{\lambda \le \mu} M_\lambda$. The class of objects with a finite filtration of this form is denoted by $\filt{\Xcal}$.

For a pair of full subcategories $\Mcal$ and $\Ncal$ of an abelian (respectively, triangulated) category $\Ccal$, we use the notation $\Mcal \star \Ncal$ for the full subcategory of $\Ccal$ consisting of objects $X$ such that there exists a short exact sequence $0\to M \to X \to N \to 0$
(respectively, a triangle $ M \to X \to N \to M[1]$) with $M\in \Mcal$ and $N\in \Ncal$.  

Let $\Mcal$ be a class of objects in a triangulated category $\Dcal$. Given a set of integers $I$ (which is often expressed by symbols such as $>n$,  $\neq n$, or just $n$),
 we write
${}^{\perp_I}\Mcal$ for the orthogonal class given by the objects $X$ satisfying $\Hom_\Dcal(X,M[i])=0$ for all $M\in\Mcal$ and $i\in I$,
while  ${\Mcal}^{\perp_I}$ consists of the objects $X$ such that $\Hom_\Dcal(M,X[i])=0$ for all $M\in\Mcal$ and  $i\in I$.
If $\Mcal$ is a class of objects in an abelian category $\Acal$ and $I$ is a set of natural numbers, we similarly denote by ${}^{\perp_I}\Mcal$ the class given by the objects $X\in\Acal$ such that $\Ext^i_\Acal(X,M)=0$ and by ${\Mcal}^{\perp_I}$ the class of objects $X$ satisfying $\Ext^i_\Acal(M,X)=0$ for all $M\in\Mcal$ and  $i\in I$.

Finally,
when $R$ is a ring, $\Mod{R}$ denotes the category of all left $R$-modules and $\D{R}$ the unbounded derived category of $\Mod{R}$. If $R$ is left coherent, then $\mod{R}$ denotes the abelian subcategory of finitely presented left $R$-modules and $\mathrm{D}^b(\mod{R})$ its bounded derived category. 

\subsection{Torsion pairs, t-structures and HRS-tilts}

Recall that a \textbf{torsion pair} in an abelian (respectively, triangulated) category $\Ccal$ is a pair of idempotent-complete subcategories $\mathfrak{t}:=(\Tcal,\Fcal)$ such that $\Hom_\Acal(T,F)=0$ for all $T$ in $\Tcal$ and $F$ in $\Fcal$ and, furthermore, with the property that $\Ccal=\Tcal\star\Fcal$. The subcategory $\Tcal$ is often called a \textbf{torsion class}  while the subcategory $\Fcal$ is often referred to as a \textbf{torsion-free class}. If a torsion-free class $\Fcal$ is again a torsion class with respect to another torsion pair $(\Fcal,\Gcal)$, then we say that $\Fcal$ is a \textbf{torsion torsion-free class} (or TTF class, for short).

A torsion pair $(\Tcal,\Fcal)$ in an abelian or triangulated category $\Ccal$ is said to be:
\begin{itemize}
\item \textbf{cogenerated by a subcategory $\Scal$} if $\Tcal={}^{\perp_0}\Scal$; and
\item \textbf{generated by a subcategory $\Scal$} if $\Fcal=\Scal^{\perp_0}$.
\end{itemize}
If $\Ccal$ is abelian, we say that $(\Tcal,\Fcal)$ is
\begin{itemize}
\item \textbf{hereditary} if $\Tcal$ is closed under subobjects; and
\item \textbf{cohereditary} if $\Fcal$ is closed under quotient objects.
\end{itemize}
If $\Ccal$ is furthermore AB5 (i.e., a cocomplete abelian category with exact direct limits), then   $(\Tcal,\Fcal)$  is  \begin{itemize}\item \textbf{of finite type} if  $\Fcal$ is closed under {direct limits in $\Ccal$}.\end{itemize}
A subcategory $\Mcal$ of an abelian (respectively, triangulated category) will be said to be \textbf{extension-closed} if $\Mcal\star \Mcal\subseteq \Mcal$. An extension-closed subcategory $\Xcal$ of a triangulated category $\Dcal$ is said to be \textbf{suspended} (respectively, \textbf{cosuspended}) if $\Xcal[1]\subseteq \Xcal$ (respectively, if $\Xcal[-1]\subseteq \Xcal$).

A torsion pair $\mathbb{T}:=(\Xcal,\Ycal)$ for which $\Xcal$ is suspended is called a \textbf{t-structure}. {Then $\Xcal$ is called the \textbf{aisle} and $\Ycal$ the \textbf{coaisle}  of $\mathbb{T}$.} Such torsion pairs give rise to an abelian subcategory of $\Dcal$, the  \textbf{heart} of the t-structure, which can be obtained as $\Hcal_\mathbb{T}:=\Xcal[-1]\cap\Ycal.$ Furthermore, there is a cohomological functor associated to $\mathbb{T}$, i.e.~a functor $H^0_\mathbb{T}\colon \Dcal\longrightarrow \Hcal_\mathbb{T}$ that sends triangles to long exact sequences, and we denote by  $H^i_\mathbb{T}:\Dcal\longrightarrow \Hcal_\mathbb{T}$  the functor given by  $H^i_\mathbb{T}(X)=H^0_\mathbb{T}(X[i])$, for any $i$ in $\mathbb{Z}$. The following useful lemma is a direct consequence of the construction of this cohomological functor. 

\begin{lemma}\label{lem:cohomol1}
Let $\mathbb{T}=(\Xcal,\Ycal)$ be a  t-structure in a triangulated category $\Dcal$. Let $W$ be an object in the heart $\Hcal$, and let $Y$ be in $\Ycal$ and $Z$ in $\Xcal[-1]$. Then we have that  $\Hom_\Dcal(W,Y)\cong \Hom_\Hcal(W, H^0_\mathbb{T}(Y))$  and $\Hom_\Dcal(Z,W)\cong \Hom_\Hcal(H^0_\mathbb{T}(Z), W).$
\end{lemma}

{A t-structure $\mathbb{T}=(\Xcal,\Ycal)$ is said to be
\begin{itemize}
\item \textbf{nondegenerate} if $\bigcap_{n\in\mathbb{Z}}\Xcal[n]=0=\bigcap_{n\in\mathbb{Z}}\Ycal[n]$.
\end{itemize}
It is easy to check that if $\mathbb{T}$ is nondegenerate, then the aisle $\Xcal$ consists of the objects $X$ with ${H}^k_{\mathbb{T}}(X)=0$ for all $k\ge0$, and the coaisle  $\Ycal$ of those with ${H}^k_{\mathbb{T}}(X)=0$ for all $k< 0$.

If $\Dcal$ is a  triangulated category with coproducts (respectively, products), then we say that $\mathbb{T}$ is
\begin{itemize}\item \textbf{smashing}  if the coaisle $\Ycal$ is closed under coproducts; and
\item  \textbf{cosmashing} if the aisle $\Xcal$ is closed under products.
\end{itemize}
A nondegenerate t-structure
 $\mathbb{T}$ is smashing (respectively, cosmashing) if and only if
  the $H^0_\mathbb{T}$ preserves coproducts (respectively, products), see \cite[Lemma 3.3]{AMV3}.}
For more details on t-structures we refer to \cite{BBD}.

For a torsion pair $\mathfrak{t}:=(\Tcal,\Fcal)$ in the heart $\Hcal_\mathbb{T}$ of a t-structure $\mathbb{T}=(\Xcal,\Ycal)$ we can build a new t-structure {according to \cite{HRS}. It is called} the \textbf{{left} HRS-tilt} at $\mathfrak{t}$, and it is defined as follows:
$$\mathbb{T}_{\mathfrak{t}^+}:=(\Xcal_{\mathfrak{t}^+}:=\Xcal[1]\star\Tcal[1],\,\Ycal_{\mathfrak{t}^+}:=\Fcal[1]\star\Ycal).$$
The corresponding heart is then given by {$\Hcal_{\mathfrak{t}^+}=\Fcal[1]\star\Tcal$, it is equipped with a torsion pair $(\Fcal[1],\Tcal)$ and the {cohomological functor}  $H^0_{\mathfrak{t}^+}\colon\Dcal\longrightarrow \Hcal_{\mathfrak{t}^+}$}.
Dually, we denote the \textbf{{right} HRS-tilt} at $\mathfrak{t}$ by
$$\mathbb{T}_{\mathfrak{t}^-}:=(\Xcal_{\mathfrak{t}^-}:=\Xcal\star\Tcal,\,\Ycal_{\mathfrak{t}^-}:=\Fcal\star(\Ycal[-1])).$$
The corresponding heart is then given by {$\Hcal_{\mathfrak{t}^-}=\Fcal\star\Tcal[-1]$, it is equipped with a torsion pair $(\Fcal,\Tcal[-1])$ and the {cohomological functor}  $H^0_{\mathfrak{t}^-}\colon\Dcal\longrightarrow \Hcal_{\mathfrak{t}^-}$}.

\begin{remark}\label{rem:inversetilt}
Note that $\mathbb{T}_{\mathfrak{t}^+}=\mathbb{T}_{t^-}[1]${, and $\mathbb{T}=(\mathbb{T}_{\mathfrak{t}^-})_{\mathfrak{s}^+}$ for the torsion pair $\mathfrak{s}=(\Fcal,\Tcal[-1])$ in $\Hcal_{\t^-}$}. Furthermore, if $\mathbb{T}$ is nondegenerate, then the right HRS-tilt is given by
$$\Xcal_{\mathfrak{t}^-}=\{X\in\Dcal\,\mid\, H^{0}_\mathbb{T}(X)\in\Tcal\text{ and }   H^k_\mathbb{T}(X)=0 \text{ for all } k> 0\},$$
$$\Ycal_{\mathfrak{t}^-}=\{X\in\Dcal\,\mid\, H^{0}_\mathbb{T}(X)\in\Fcal\text{ and }   H^k_\mathbb{T}(X)=0 \text{ for all } k<0\},$$
and the corresponding statement holds true for the left HRS-tilt.
\end{remark}

By construction, we have that $\Ycal\subseteq \Ycal_{\mathfrak{t}^+}\subseteq \Ycal[1]$ and $\Ycal[-1]\subseteq \Ycal_{\mathfrak{t}^-}\subseteq \Ycal$. These properties completely characterise $\mathbb{T}_{\mathfrak{t}^+}$ as a left HRS-tilt and  $\mathbb{T}_{\mathfrak{t}^-}$ as a right HRS-tilt.

\begin{proposition}\cite[Lemma 1.1.2]{Pol}\cite[Proposition 2.1]{Woolf}\label{prop:intermediate}
Let $\mathbb{T}=(\Xcal,\Ycal)$ be a t-structure in a triangulated category $\Dcal$ with heart $\Hcal$.
\begin{enumerate}

\item The assignment $\mathfrak{t}\mapsto \mathbb{T}_{\mathfrak{t}^+}$ defines a bijection between torsion pairs in $\Hcal$ and t-structures $\mathbb{T'}=(\Xcal',\Ycal')$ in $\Dcal$ with $\Ycal\subseteq \Ycal'\subseteq \Ycal[1]$. The inverse assignment
takes a t-structure $\mathbb{T'}$  with heart $\Hcal'$
 to the torsion pair $\mathfrak{t}:=(\Tcal,\Fcal)$ in  $\Hcal$ given by $\Tcal=\Hcal\cap\Hcal'$ and $\Fcal=\Hcal\cap\Hcal'[-1]$.

\item The assignment $\mathfrak{t}\mapsto \mathbb{T}_{\mathfrak{t}^-}$ defines a bijection between torsion pairs in $\Hcal$ and t-structures $\mathbb{T'}=(\Xcal',\Ycal')$ in $\Dcal$ with $\Ycal[-1]\subseteq \Ycal'\subseteq \Ycal$. The inverse assignment
takes a t-structure $\mathbb{T'}$  with heart $\Hcal'$ to the torsion pair $\mathfrak{t}:=(\Tcal,\Fcal)$ in  $\Hcal$ given by $\Tcal=\Hcal\cap\Hcal'[1]$ and $\Fcal=\Hcal\cap\Hcal'$.
\end{enumerate}
\end{proposition}


\subsection{Silting and cosilting t-structures}\label{sec: silting and cosilting}
Recall that an object $\sigma$ of a triangulated category $\Dcal$ with (set-indexed) coproducts is said to be \textbf{silting} if the pair $({\sigma}^{\perp_{\geq 0}},{\sigma}^{\perp_{<0}})$ is a t-structure in $\Dcal$. Dually, if $\Dcal$ has (set-indexed) products, an object $\sigma$ of $\Dcal$ is said to be \textbf{cosilting} if $({}^{\perp_{\leq 0}}\sigma, {}^{\perp_{>0}}\sigma)$ is a t-structure in $\Dcal$. In both cases we denote the associated (silting/cosilting) t-structure by $\mathbb{T}_\sigma$, its heart by $\Hcal_\sigma$ and its associated cohomological functor by $H^0_\sigma\colon\Dcal\longrightarrow \Hcal_\sigma$.
Two (co)silting objects are said to be \textbf{equivalent} if they give rise to the same t-structure.
 
\begin{remark}
Note that it is our convention that the heart of a t-structure is contained in the coaisle, not {in} the aisle. This justifies the slight adaptation (by a shift) of the definition of the t-structure associated to a silting {object} presented above (compare with \cite{PV}).
\end{remark}

We say that a subcategory $\Mcal$ \textbf{generates} $\Dcal$ if ${\Mcal}^{\perp_{\mathbb Z}}=0$ and $\Mcal$ \textbf{cogenerates} $\Dcal$ if ${}^{\perp_{\mathbb Z}}\Mcal=0$. It follows from the definition that if $\sigma$ is a silting object, then $\Add{\sigma}$ generates $\Dcal$, and if $\sigma$ is a cosilting objects, then $\Prod{\sigma}$ cogenerates $\Dcal$. We recall the following properties of (co)silting t-structures.

\begin{proposition}\cite[Proposition 4.3 and Lemma 4.5]{PV},\cite[Lemma 2.8, Theorem 3.5, and Corollary 3.8]{AMV3}\label{prop:summary cosilting}
Let $\sigma$ be an object in a triangulated category $\Dcal$.
\begin{enumerate}
\item If $\Dcal$ has coproducts and $\sigma$ is a silting object, then
 the associated heart $\Hcal_\sigma=\sigma^{\perp_{\neq 0}}$ is an abelian category with enough projectives, and the functor $H^0_\sigma$ induces an equivalence of categories between $\Add{\sigma}$ and $\Proj{\Hcal_\sigma}$ and a natural {isomorphism} $\Hom_\Dcal(\sigma,-)\cong \Hom_\Dcal(H^0_\sigma(\sigma),H^0_\sigma(-)).$ Furthermore, if $\Xcal={\sigma}^{\perp_{> 0}}$, then
 $\Add{\sigma}={}^{\perp_1}\Xcal\cap\Xcal$.

\item If $\Dcal$ has products and $\sigma$ is a cosilting object, then
 the associated heart $\Hcal_\sigma={}^{\perp_{\neq 0}}\sigma$ is an abelian category with enough injectives, and
 the functor $H^0_\sigma$ induces an equivalence of categories between $\Prod{\sigma}$ and $\Inj{\Hcal_\sigma}$ and a natural {isomorphism} $\Hom_\Dcal(-,\sigma)\cong \Hom_\Dcal(H^0_\sigma(-),H^0_\sigma(\sigma)).$ Furthermore, if $\Ycal={}^{\perp_{>0}}\sigma$, then  $\Prod{\sigma}=\Ycal\cap\Ycal^{\perp_1}$.
\end{enumerate}
\end{proposition}

\begin{remark}\label{rem:torsion torsion-free characterisation}
If $\Dcal$ is a triangulated category with coproducts, then every heart is known to be cocomplete (\cite[Proposition 3.2]{PS}). Moreover, if a given t-structure is associated to a silting object, then the corresponding heart $\Hcal$ has a projective generator by the proposition above. Then Freyd's adjoint functor theorem implies that $\Hcal$ is also complete (see \cite[Proposition 6.4]{Fa}) and products are necessarily exact. Moreover, $\Hcal$ has the property of being \textbf{well-powered}, i.e.~every object has a set of subobjects (see \cite[Proposition IV.6.6]{Stenstrom}).
Dually, if $\Dcal$ is a triangulated category with products, then the heart $\Hcal$ of any t-structure associated to a cosilting object  is complete, cocomplete, with exact coproducts and well-powered.

It follows in both cases from \cite{Dickson} that torsion classes are precisely those closed under coproducts, extensions and quotients, and torsion-free classes are precisely those closed under products, extensions and subobjects. Furthermore, in any complete, cocomplete and well-powered abelian category $\Acal$, we then have that:
\begin{itemize}
\item if $\Escr$ is a family of injective objects in $\Acal$, then $({}^{\perp_0}\Escr,\Cogen{\Escr})$ is a hereditary torsion pair in $\Acal$; and
\item if $\Pscr$ is a family of projective objects in $\Acal$, then $(\Gen{\Pscr},\Pscr^{\perp_0})$ is a cohereditary torsion pair in $\Acal$.
\end{itemize}
\end{remark}

We close the section with a brief review of the notion of purity in triangulated categories.
Assume now that $\Dcal$ admits (set-indexed) coproducts. Recall that an object $X$ in $\Dcal$ is said to be \textbf{compact} if the functor $\Hom_\Dcal(X,-)$ commutes with coproducts. If the subcategory $\Dcal^c$ of compact objects is skeletally small and  generates $\Dcal$, then $\Dcal$ is said to be \textbf{compactly generated}. It is well-known that $\Dcal$ then also admits products.

When $\Dcal$ is a compactly generated triangulated category, the category of additive (contravariant) functors $(\Dcal^c)^\mathrm{op}\longrightarrow \Mod{\mathbb{Z}}$, denoted by $\Mod{\Dcal^c}$, is a locally coherent Grothendieck category with enough projectives. Recall that a Grothendieck category $\Gcal$ is said to be \textbf{locally coherent} if its subcategory of finitely presented objects $\fp\Gcal$ is an abelian subcategory and generates $\Gcal$. 

Moreover, the functor $\mathbf{y}\colon \Dcal\longrightarrow \Mod{\Dcal^c}$, defined by $\mathbf{y}X:=\Hom_{\Dcal}(-,X)_{|\Dcal^c}$, sends triangles to long exact sequences. A \textbf{pure triangle} is a triangle in $\Dcal$ of the form
$$\Delta\colon \ \ \ \xymatrix{X\ar[r]^f&Y\ar[r]^g&Z\ar[r]^h&X[1]}$$
such that $\mathbf{y}\Delta$ is a short exact sequence, i.e.\ $\mathbf{y}f$ is a monomorphism (and, hence, $f$ is called a \textbf{pure monomorphism}), $\mathbf{y}g$ is an epimorphism (and, hence, $g$ is called a \textbf{pure epimorphism}) and $\mathbf{y}h=0$ (and, hence, $h$ is called a \textbf{phantom map}). We say that an object $X$ of $\Dcal$ is \textbf{pure-injective} (respectively, \textbf{pure-projective}) if $\mathbf{y}X$ is an injective (respectively, projective) object in $\Mod{\Dcal^c}$. Equivalently, an object $X$ of $\Dcal$ is pure-injective (respectively, pure-projective) if and only if every pure triangle starting in $X$ (respectively, ending in $X$) splits. It is well-known that the pure-projective objects coincide precisely with $\Add{\Dcal^c}$ (\cite[Lemma 8.1]{Bel}). The following theorem explains the benefits of considering pure-injective/pure-projective cosilting/silting objects.

\begin{theorem}\label{thm:purity and silting}\cite{AMV3,NSZ}
Let $\Dcal$ be a compactly generated triangulated category.
\begin{enumerate}
\item There is a bijection between equivalence classes of cosilting objects and smashing nondegenerate t-structures whose heart has an injective cogenerator.  Furthermore, a cosilting object is pure-injective if and only if the heart of the associated t-structure is a Grothendieck category.

\item There is a bijection between equivalence classes of silting objects and cosmashing nondegenerate t-structures whose heart has a projective generator. Furthermore, a silting object is pure-projective if and only if the associated t-structure is smashing and its heart is a Grothendieck category with a projective generator.
\end{enumerate}
\end{theorem}

\begin{example}\cite{WZ,MV}\label{Ex: bdd pure-inj} An object  $\sigma$ in the category $\mathsf{K}^b(\mathrm{Inj}(R))$ of bounded complexes of injective   $R$-modules is  cosilting  in $\D R$ if and only if $\Hom_{\D R}(\sigma^I,\sigma[i])=0$ for all sets $I$ and $i>0$, and  $\mathsf{K}^b(\mathrm{Inj}(R))$  is the smallest triangulated subcategory of $\D R$ containing $\Prod\sigma$. Such cosilting complexes are pure injective objects of $\D R$  and thus give rise to  t-structures with Grothendieck heart.
\end{example}

\section{The concept of mutation}\label{sec:mutation}

Silting  mutation was introduced by Aihara and Iyama in \cite{AI}. In this section, we define mutation for large silting or  cosilting objects, and in Theorem~\ref{thm:compact silting mutation} we will show that it extends Aihara-Iyama's silting mutation.   We prove   that a (co)silting object $\sigma'$ is a mutation (left or right) of another (co)silting object $\sigma$ if and only if the associated (co)silting t-structures are related by a suitable (left or right) HRS-tilt. The torsion pair at which the HRS-tilt is performed  is determined by the intersection $\Ecal=\Prod{\sigma}\cap\Prod{\sigma'}$ in the cosilting case, and by $\Ecal=\Add{\sigma}\cap\Add{\sigma'}$ in the silting case.

We begin with a useful lemma. Recall that, given an object $C$ in an additive category $\Ccal$, we say that  a morphism $g\colon M\to C$ with $M$ in $\Mcal$  is an \textbf{$\Mcal$-precover} (or a \textbf{right $\Mcal$-approximation}) if any other morphism $M'\to C$ with $M'$ in $\Mcal$ factors through $g$. If, in addition, any endomorphism $h\colon M\to M$ with $gh=g$ is an isomorphism, then $g$ is called an \textbf{$\Mcal$-cover}. Dually, one defines \textbf{$\Mcal$-preenvelopes} (or \textbf{left $\Mcal$-approximations}) and \textbf{$\Mcal$-envelopes}.

 \begin{lemma}\label{lem:cosilting enveloping}\label{lem:(co)generate}
Let $\sigma$ be an object in a triangulated category $\Dcal$.
\begin{enumerate}
\item If $\Dcal$ has products, $\sigma$ is cosilting and {$\Escr=\Prod\Escr$ is a subcategory} of $\Prod{\sigma}$, then:
\begin{enumerate}
\item $\Escr$ is preenveloping in $\Prod{\sigma}$; and 
\item if $\Escr$ cogenerates $\Dcal$, then ${\Escr}=\Prod{\sigma}$.
\end{enumerate}
\item If $\Dcal$ has coproducts, $\sigma$ is silting and {$\Pscr=\Add\Pscr$ is a subcategory} of $\Add{\sigma}$, then:
\begin{enumerate}
\item $\Pscr$ is precovering in $\Add{\sigma}$; and
\item if $\Pscr$ generates $\Dcal$, then $\Pscr=\Add{\sigma}$.
\end{enumerate}
\end{enumerate}
\end{lemma}
\begin{proof}
We prove the cosilting case; the silting case is analogous. (1)(a): Consider the torsion pair $\mathfrak{s}:=({}^{\perp_0}H^0_\sigma(\Escr),\Cogen{H^0_\sigma(\Escr)})$ and take $a\colon  H^0_\sigma(\sigma)\longrightarrow A$ to be the epimorphism to a torsion-free object $A$ with a torsion kernel (with respect to $\mathfrak{s}$). Consider then a monomorphism $\pi\colon A\longrightarrow H^0_\sigma(E)$, for some $E$ in $\Escr$. It is easy to observe that $\pi\circ a$ is an $H^0_\sigma(\Escr)$-preenvelope of $H^0_\sigma(\sigma)$, and by  Proposition~\ref{prop:summary cosilting} the map $\Psi\colon \sigma\longrightarrow E$ such that $H^0_\sigma(\Psi)=\pi\circ a$ is a $\Escr$-preenvelope of $\sigma$.

(1)(b): Let $\Psi\colon \sigma\longrightarrow E$ be the preenvelope obtained above. If $\Escr$ cogenerates $\Dcal$, then $H^0_\sigma(\Escr)$ is a {cogenerating} class of injective objects in $\Hcal_\sigma$. In particular, $H^0_\sigma(\Psi)$ must be a monomorphism and, thus, a split map. Therefore also $\Psi$ splits and we conclude that $\Escr=\Prod{\sigma}$.
\end{proof}

We will first discuss mutation of cosilting objects. Later we will restrict our attention to pure-injective cosilting objects, for which we can use some approximation-theoretic tools to simplify the definition below.

\begin{definition}\label{def:mut} Let $\Dcal$ be a triangulated category with products. Let
$\sigma$ and $\sigma'$ be two cosilting objects in $\Dcal$,  and let $\Escr=\Prod{\sigma}\cap\Prod{\sigma'}$. We say that
\begin{enumerate}
\item  $\sigma'$ is a \textbf{left mutation} of $\sigma$
if there is a triangle $\xymatrix{\sigma\ar[r]^\Phi& \varepsilon_0\ar[r]& \varepsilon_1\ar[r]& \sigma[1]}$  such that:
 \begin{itemize}
 \item $\Phi$ is an $\Escr$-preenvelope  of $\sigma$ in $\Dcal$; and
 \item $\varepsilon_0\oplus \varepsilon_1$ is a cosilting object equivalent to $\sigma'$.
 \end{itemize}
 \item $\sigma'$ is a \textbf{right mutation} of $\sigma$ if there is a triangle $\xymatrix{\sigma[-1]\ar[r]& \gamma_1\ar[r]& \gamma_0\ar[r]^\Phi&\sigma}$ such that:
 \begin{itemize}
 \item $\Phi$ is an $\Escr$-precover  of $\sigma$ in $\Dcal$; and  
 \item $\gamma_0\oplus \gamma_1$ is a cosilting object equivalent to $\sigma'$.
 \end{itemize}
 \end{enumerate}
 We will also say that $\sigma'$ is a \textbf{left (or right) mutation} of $\sigma$ \textbf{with respect to $\Escr$}.
\end{definition}

Our first aim is to clarify the relation between mutations and HRS-tilts. This will allow us to see, in particular, that mutation is defined up to equivalence of cosilting objects. For that purpose, we will make use of the following lemma.

\begin{lemma}\label{lem:approx prop}
Let $\sigma$ be a  cosilting object in a  triangulated category $\Dcal$ with products, and let {$\Escr=\Prod\Escr$} be a  subcategory of $\Prod{\sigma}$. Denote by $$\mathfrak{t}=(\Scal, \Rcal)=({}^{\perp_0}H^0_\sigma(\Escr),\,\Cogen{H^0_\sigma(\Escr)})$$ the torsion pair in $\Hcal_{\sigma}$ cogenerated by the set of injective objects $H^0_\sigma(\Escr)$, and let
$\mathbb{T}=(\Xcal',\Ycal')$ be the right HRS-tilt of $\mathbb{T}_{\sigma}$ at $\mathfrak{t}$. The following statements are equivalent for a map $\Phi\colon E\longrightarrow \sigma$ with $E$ in $\Escr$. 
\begin{enumerate}
\item $\Phi$ is an $\Escr$-precover of $\sigma$ in $\Dcal$.
\item $\phi:=H^0_\sigma(\Phi)$ is an $\Rcal$-precover of $H_\sigma^0(\sigma)$ in $\Hcal_\sigma$.
\item $\Phi$ is a $\Ycal'$-precover of $\sigma$ in $\Dcal$.
\end{enumerate}
\end{lemma}
\begin{proof}
(1)$\Rightarrow$(2): The map $\phi$ is an $H^0_\sigma(\Escr)$-precover
because $H_\sigma^0$ induces an equivalence of categories between $\Prod{\sigma}$ and $\Inj{\Hcal_\sigma}$.
 Let  $g\colon X\longrightarrow H^0_\sigma(\sigma)$ be a map in $\Hcal_\sigma$ with $X$ in $\Rcal=\Cogen{H^0_\sigma(\Escr)}$. Since $H^0_\sigma(\sigma)$ is injective, $g$ extends along any monomorphism  $h\colon X\longrightarrow H^0_\sigma(E')$ with  $E'$ in $\Escr$, that is, there is $t\colon H^0_\sigma(E')\longrightarrow H_\sigma^0(\sigma)$ such that $t\circ h=g$. But then $t$ must factor through $\phi$, so there is $\psi\colon H^0_\sigma(E')\longrightarrow  H^0_\sigma(E)$ such that $\phi\circ \psi=t$. This shows that $g=\phi\circ\psi\circ h$, as wanted.

(2)$\Rightarrow$(3): Suppose that $\phi$ is an {$\Rcal$-pre}cover of $H_\sigma^0(\sigma)$ in $\Hcal_\sigma$ and let $f\colon X\longrightarrow \sigma$ be a map with $X$ in $\Ycal'$. Applying $H_\sigma^0$ to the triangle induced by {$f$} we get an exact sequence
$$\xymatrix{0\ar[r]&H^0_\sigma(K)\ar[r]&H^0_\sigma(X)\ar[r]^{H_\sigma^0(f)}&H^0_\sigma(\sigma)}.$$
{Now $H^0_\sigma(X)$ lies in $\Rcal$ by Remark~\ref{rem:inversetilt}, so the map} $H^0_\sigma(f)$ factors through $\phi$, that is, there is a map $\psi\colon H^0_\sigma(X)\to H^0_\sigma(E)$ such that $H^0_\sigma(f)=\phi\circ\psi$. Since $E$ lies in $\Prod{\sigma}$, we know from Proposition~\ref{prop:summary cosilting}(2) that the functor $H^0_\sigma$  induces an isomorphism $\Hom(X,E)\cong \Hom(H^0_\sigma(X),H^0_\sigma(E))$. Hence there is a unique map $\Psi\colon X\to E$ such that $H^0_\sigma(\Psi)=\psi$. Clearly, it follows that $f=\Phi\Psi$ as wanted.

(3)$\Rightarrow$(1): This is clear from the fact that $\Escr$ is contained in $\Ycal'$ and $E$, by assumption, lies in $\Escr$.
\end{proof}

\begin{remark}
It is clear that the  class $\Rcal$ in the torsion pair $\mathfrak{t}=(\Scal, \Rcal)$ above  is enveloping (as every torsion-free class is!), but without further assumptions on the cosilting object  it may be hard to say whether it is precovering or not. We will discuss the approximation properties of such torsion pairs in Section~\ref{sec:mutation and purity} (compare also with the silting case in Remark~\ref{rem:silting is preenveloping}).
\end{remark}

We are now ready for our first theorem.

\begin{theorem}\label{thm:mutation as HRS-tilt}
Let $\Dcal$ be a triangulated category with products. Let
$\sigma$ and $\sigma'$ be two cosilting objects in $\Dcal$,  and let
$\Escr=\Prod{\sigma}\cap\Prod{\sigma'}$.
\begin{enumerate}
\item  $\sigma'$ is a {left mutation} of $\sigma$  if and only if
${}^{\perp_0}H^0_\sigma(\mathscr{E})$ is closed under products in $\Hcal_\sigma$ and
$\mathbb{T}_{\sigma'}$ is the left HRS-tilt of $\mathbb{T}_{\sigma}$ at the torsion pair $\mathfrak{t}=(\Tcal, {}^{\perp_0}H^0_\sigma(\mathscr{E}))$ in $\Hcal_\sigma$. 
\item $\sigma'$ is a {right mutation} of $\sigma$ if and only if $\sigma$ admits an $\Escr$-precover and
$\mathbb{T}_{\sigma'}$ is the right HRS-tilt of $\mathbb{T}_{\sigma}$ at the torsion pair $\mathfrak{t}=({}^{\perp_0}H^0_\sigma(\mathscr{E}),\,\Cogen{H^0_\sigma(\mathscr{E})})$ in $\Hcal_\sigma$. 
\end{enumerate}
In both cases, the torsion pairs involved do not depend on the triangle in Definition~\ref{def:mut}.
\end{theorem}

\begin{proof} Let $\mathbb{T}=\mathbb{T}_{\sigma}=(\Xcal, \Ycal)$ be the cosilting t-structure associated to $\sigma$ with heart $\Hcal=\Hcal_\sigma$, and $\mathbb{T'}=\mathbb{T}_{\sigma'}=(\Xcal', \Ycal')$  the cosilting t-structure associated to $\sigma'$ with heart $\Hcal'=\Hcal_{\sigma'}$.

(1): Suppose first that $\sigma'$ is a {left mutation} of $\sigma$. Let $\Phi$ be an $\Escr$-preenvelope  of $\sigma$ in $\Dcal$ and consider the   cosilting object $\tilde{\sigma}=\varepsilon_0\oplus \varepsilon_1$, equivalent to $\sigma'$, where $\varepsilon_1$ is defined via the triangle
\begin{equation}\tag{$\Delta_1$}\label{eq:triangle preenvelope}
\xymatrix{\sigma\ar[r]^\Phi& \varepsilon_0\ar[r]& \varepsilon_1\ar[r]& \sigma[1]}
\end{equation}
Then $\mathbb{T'}=({}^{\perp_{\le 0}}\tilde{\sigma},{}^{\perp_{>0}}\tilde{\sigma})$, and it is easy to see that
 $\Ycal\subseteq \Ycal'\subseteq \Ycal[1]$. From Proposition \ref{prop:intermediate}(1) we infer that
$\mathbb{T'}= \mathbb{T}_{\mathfrak{t}^+}$ is the left HRS-tilt of $\mathbb{T}$ at the torsion pair $\mathfrak{t}:=(\Tcal,\Fcal)$ in  $\Hcal$ given by  $\Fcal=\Hcal\cap\Hcal'[-1]=\Hcal\cap\Ycal'[-1]$.
Note that $\mathfrak{t}$ only depends on $\sigma$ and $\sigma'$. It remains to verify  that $\Fcal={}^{\perp_0}H^0_\sigma(\mathscr{E})$.

In fact, we first observe that $\Fcal={}^{\perp_0}H^0_\sigma(\varepsilon_0)$. An object $X$ of $\Hcal$ lies in $\Fcal$ if and only if $X[1]$ lies in $\Ycal'$. In particular, $\Hom_\Dcal(X[1],\tilde{\sigma}[1])=0$ and, thus, $\Hom_\Dcal(X,\tilde{\sigma})=0$. By Lemma~\ref{lem:cohomol1} we conclude that $\Hom_\Hcal(X,H^0_\sigma(\varepsilon_0))\cong\Hom_\Dcal(X,\varepsilon_0)=0$. Conversely, if $X$ belongs to ${}^{\perp_0}H^0_\sigma(\varepsilon_0)$, that is, $\Hom_\Dcal(X,\varepsilon_0)=0$, then applying $\Hom_\Dcal(X[1],-)$ to the triangle above and keeping in mind that $X$ lies in $\Ycal={}^{\perp_{>0}}{\sigma}$ and $\varepsilon_0$ lies in $\Prod\sigma$, we see that $X[1]$ lies indeed in ${}^{\perp_{>0}}\tilde{\sigma}=\Ycal'$, which amounts to $X$ lying in $\Fcal$.

Finally, we show that ${}^{\perp_0}H^0_\sigma(\varepsilon_0)={}^{\perp_0}H^0_\sigma(\mathscr{E})$, thus proving our claim. Note that, being left orthogonal to (classes of) injective objects in $\Hcal$, both these classes are torsion (see Remark~\ref{rem:torsion torsion-free characterisation}), and it is clear that ${}^{\perp_0}H^0_\sigma(\mathscr{E})\subseteq {}^{\perp_0}H^0_\sigma(\varepsilon_0)$. So, in order to prove the desired equality, it is enough to show that the corresponding torsion-free classes satisfy $({}^{\perp_0}H^0_\sigma(\mathscr{E}))^{\perp_0}\subseteq ({}^{\perp_0}H^0_\sigma(\varepsilon_0))^{\perp_0}$, and for that it suffices to show that $H^0_\sigma(\mathscr{E})$ lies in $({}^{\perp_0}H^0_\sigma(\varepsilon_0))^{\perp_0}=\Fcal^{\perp_0}$. Recall that  $\Ycal'=\Fcal[1]\star \Ycal$,  and that  $\Prod\sigma=\Ycal\cap(\Ycal[-1])^{\perp_0}$ by Proposition~\ref{prop:summary cosilting}. Therefore, we have  
$$\Prod\sigma\cap \Fcal^{\perp_0}=\Ycal\cap(\Ycal[-1])^{\perp_0}\cap\Fcal^{\perp_0}=\Ycal\cap(\Fcal\star\Ycal[-1])^{\perp_0}=\Ycal\cap(\Ycal'[-1])^{\perp_0}$$
But the latter class coincides with $\Escr$. Indeed, since $\Ycal\subseteq\Ycal'$ (and, thus, $(\Ycal[-1])^{\perp_0}\supseteq(\Ycal'[-1])^{\perp_0}$) we have
$$\Ecal= \Prod\sigma\cap\Prod{\sigma'}=(\Ycal\cap(\Ycal[-1])^{\perp_0})\cap(\Ycal'\cap(\Ycal'[-1])^{\perp_0})=\Ycal\cap(\Ycal'[-1])^{\perp_0}.$$
So we have shown that, in fact, $\Inj{\Hcal}\cap \Fcal^{\perp_0}=H^0_\sigma(\mathscr{E})$.

Conversely, assume now that $\mathbb{T}'$ is the left HRS-tilt of $\mathbb{T}$ at the torsion pair  $\mathfrak{t}=(\Tcal, \Fcal)$ in $\Hcal_\sigma$ where $\Fcal={}^{\perp_0}H^0_\sigma(\mathscr{E})$. Note here that since the hereditary torsion class ${}^{\perp_0}H^0_\sigma(\mathscr{E})$ is assumed to be closed under products, it is also a torsion-free class (see Remark~\ref{rem:torsion torsion-free characterisation}), so the torsion pair $\mathfrak{t}$ is well-defined and the condition makes sense.
Consider an $\Escr$-preenvelope $\Phi\colon \sigma\longrightarrow\varepsilon_0$ of $\sigma$ in $\Dcal$ (see Lemma~\ref{lem:cosilting enveloping} for the existence of such a map). We complete $\Phi$ to a triangle of the form \eqref{eq:triangle preenvelope} and set  $\tilde{\sigma}=\varepsilon_0\oplus \varepsilon_1$. We will show that $\Prod{\sigma'}=\Prod{\tilde{\sigma}}$. Note that since $\sigma$ cogenerates $\Dcal$, so does $\tilde{\sigma}$. Hence, if we show that $\Prod{\tilde{\sigma}}$ is contained in $\Prod{\sigma'}$, Proposition~\ref{lem:(co)generate}(1)(b) guarantees the equality. For this purpose we show that $\tilde{\sigma}$ lies in $\Ycal'\cap\Ycal'[-1]^{\perp_0}$, which coincides with $\Prod{\sigma'}$ by Proposition~\ref{prop:summary cosilting}. Recall from Remark~\ref{rem:inversetilt} that
$$\Ycal'=\Ycal_{\mathfrak{t}^+}= \Fcal[1]\star\Ycal= \{X\in\Dcal\,\mid\, H_\sigma^{-1}(X)\in\Fcal, \textrm{ and } H_\sigma^k(X)=0 \text{ for all } k<-1\}.$$
It is clear that $\varepsilon_0$ lies in $\Ycal\subset\Ycal'$. We claim that also $\varepsilon_1$ lies in $\Ycal'$. Indeed, we infer from the triangle above that
$H_\sigma^k(\varepsilon_1)=0$ for all $k<-1$, and moreover, since $H^0_\sigma(\Phi)$ is an $H^0_\sigma(\Escr)$-preenvelope of $H^0_\sigma(\sigma)$, we also have that $H_\sigma^{-1}(\varepsilon_1)$ lies in ${}^{\perp_0}H^0_\sigma(\Escr)=\Fcal$. Hence, $\tilde{\sigma}$ belongs to $\Ycal'$.
Next, we pick an object $Y$ in $\Ycal'$. Then $H_\sigma^{-1}(Y)$ lies in $\Fcal$ and, thus, $\Hom_\Dcal(Y,\varepsilon_0[1])\cong \Hom_\Dcal(H_\sigma^{-1}(Y),H_\sigma^{0}(\varepsilon_0))=0$. Moreover, as $\Ycal'\subset \Ycal[1]$, we also have $\Hom_\Dcal(Y,\sigma[2])=0$. Thus, we infer from a rotation of the triangle above that $\Hom_\Dcal(Y,\varepsilon_1[1])=0$. This shows that $\tilde{\sigma}$ belongs to $\Ycal'^{\perp_1}$ and
completes the proof.

(2): Suppose that $\sigma'$ is a {right mutation} of $\sigma$. By definition, there is an $\Escr$-precover $\Phi\colon\gamma_0\longrightarrow\sigma$ in $\Dcal$ and $\tilde{\sigma}=\gamma_0\oplus \gamma_1$ is a cosilting object equivalent to $\sigma'$, where $\gamma_1$ is defined by the triangle
\begin{equation} \tag{$\Delta_2$}\label{eq:triangle precover}
\xymatrix{\sigma[-1]\ar[r]& \gamma_1\ar[r]& \gamma_0\ar[r]^\Phi&\sigma.}
\end{equation}
As in the previous part, one verifies that $\Ycal[-1]\subseteq\Ycal'\subseteq\Ycal$ and $\mathbb{T'}= \mathbb{T}_{\mathfrak{t}^-}$ is the right HRS-tilt of $\mathbb{T}$ at a torsion pair $\mathfrak{t}=(\Tcal,\Fcal)$ in $\Hcal$ where $\Tcal={}^{\perp_0}H^0_\sigma(\gamma_0)$. Consequently, $\Fcal=\Cogen {H^0_\sigma(\gamma_0)}$ by Remark~\ref{rem:torsion torsion-free characterisation}. It remains to check that $\Fcal=\Cogen {H^0_\sigma(\Escr)}$. However, we know from Lemma~\ref{lem:approx prop} that $H^0_\sigma(\Phi)$ is a $\Cogen {H^0_\sigma(\Escr)}$-precover of $H^0_\sigma(\sigma)$.  Since $H^0_\sigma(\sigma)$ is an injective cogenerator of $\Hcal$, it follows that indeed $\Fcal=\Cogen {H^0_\sigma(\Escr)}$.

Conversely, suppose that $\sigma$ admits an $\mathscr{E}$-precover  $\Phi\colon \gamma_0\longrightarrow \sigma$ and that $\mathbb{T}_{\sigma'}$ is the right HRS-tilt of $\mathbb{T}_{\sigma}$ at the torsion pair $\mathfrak{t}$ in $\Hcal_\sigma$ given by $\mathfrak{t}=({}^{\perp_0}H^0_\sigma(\Escr),\Cogen{H^0_\sigma(\Escr)})$. Complete $\Phi$ to a triangle of the form \eqref{eq:triangle precover} and set  $\tilde{\sigma}=\gamma_0\oplus \gamma_1$. We have to show that $\Prod{\sigma'}=\Prod{\tilde{\sigma}}$. As in (1),  it will suffice to show that $\tilde{\sigma}$ belongs to $\Ycal'\cap(\Ycal'[-1])^{\perp_0}$. Since $\gamma_0$ lies in $\Ycal$, there is a truncation triangle of the form
$$\xymatrix{H^0_\sigma(\gamma_0)\ar[r]&\gamma_0\ar[r]& Y[-1] \ar[r]& H^0_\sigma(\gamma_0)[1]}$$
with  $Y$ in $\Ycal$ and, thus, $\gamma_0$ lies in $\Fcal\star \Ycal[-1]=\Ycal_{\mathfrak{t}^-}=\Ycal'$. Since $\gamma_1$ is an extension of $\sigma[-1]$ (which lies in $\Ycal[-1]\subseteq \Ycal'$) and $\gamma_0$, we conclude that also $\gamma_1$ (and, hence, $\tilde{\sigma}$) belongs to $\Ycal'$. Next, we pick an object $Y$ in $\Ycal'$ and  apply  $\Hom_\Dcal(Y,-)$ to the following rotation of the triangle \eqref{eq:triangle precover}  
$$\xymatrix{\gamma_0\ar[r]^\Phi&\sigma\ar[r]& \gamma_1[1]\ar[r]&\gamma_0[1].}$$
Since, from Lemma~\ref{lem:approx prop}, $\Phi$ is an $\Ycal'$-precover, and since $\Ycal'\subseteq\Ycal={}^{\perp_{>0}}\sigma$ and $\gamma_0$ lies in $\Prod\sigma$, we conclude that $\tilde{\sigma}$ belongs to $(\Ycal'[-1])^{\perp_0}$, completing the proof.
\end{proof}

\begin{remark}\label{rem:different kinds of torsion} Note that the HRS-tilts discussed above involve torsion pairs in $\Hcal_\sigma$ of different flavours:
 in the case of {left mutation}, $\mathbb{T}_{\sigma'}$ is the HRS-tilt {of} $\mathbb{T}_\sigma$ at a torsion pair $(\Tcal,\Fcal)$ for which $\Fcal$ is a TTF class;
in the case of {right mutation}, $\mathbb{T}_{\sigma'}$ is the HRS-tilt {of} $\mathbb{T}_\sigma$ at a hereditary torsion pair $(\Scal,\Rcal)$. Moreover, if there exists both a left and a right mutation with respect to $\mathscr{E}$, then $(\Tcal,\Fcal)$ is left adjacent to $(\Scal,\Rcal)$, i.e. $\Scal=\Fcal(={}^{\perp_0}H_\sigma^0(\mathscr{E}))$.
\end{remark}

\begin{corollary}\label{cor:inverse mutations}
Let $\Dcal$ be a triangulated category with products. Then both left and right mutation of cosilting objects in $\Dcal$ are well-defined up to equivalence. Moreover, the induced operations on equivalence classes of cosilting objects are inverse of each other.
\end{corollary}
\begin{proof}
Theorem~\ref{thm:mutation as HRS-tilt} characterises the fact that two cosilting objects $\sigma$ and $\sigma'$ are mutations of each other in terms of properties of the associated cosilting t-structures and of the class $\Escr=\Prod{\sigma}\cap\Prod{\sigma'}$. These t-structures and $\Escr$ depend exclusively on the equivalence class of the cosilting objects considered. Moreover, using the fact that products of $\Escr$-precovers are $\Escr$-precovers, we note that if $\sigma$ admits an $\Escr$-precover, then so does any cosilting object equivalent to $\sigma$.

For the second statement, we observe that $\sigma'$ is a left mutation of $\sigma$ if and only if $\sigma$ is a right mutation of $\sigma'$.  Indeed, if  $\Escr=\Prod{\sigma}\cap\Prod{\sigma'}$, consider the triangle induced by an $\Escr$-preenvelope $\Phi$ as follows
$$\xymatrix{\sigma\ar[r]^\Phi& \varepsilon_0\ar[r]^\Psi&\varepsilon_1\ar[r]& \sigma[1],}$$
and the cosilting object $\tilde{\sigma}=\varepsilon_0\oplus\varepsilon_1$ which, by assumption, is equivalent to $\sigma'$. It then follows that $\Psi$ is an $\Escr$-precover of $\varepsilon_1$ (since $\Hom_\Dcal(\mathscr{E},\sigma[1])=0$), and $\theta=\Psi\oplus 1_{\varepsilon_0}$ is an $\Escr$-precover of $\tilde{\sigma}$. Therefore, the induced triangle
$$\xymatrix{\sigma\ar[r]&\varepsilon_0\oplus \varepsilon_0\ar[r]^{\ \ \ \theta}& \tilde{\sigma}\ar[r]& \sigma[1]}$$
shows that $\sigma$ is a right mutation of $\tilde\sigma$. The other implication is obtained by dual arguments.
\end{proof}

The torsion pairs appearing the theorem above are determined by the cosilting objects involved in the mutation process. In general it is difficult to characterise when a torsion pair arises in this way, although we will see later that this is possible in the setting of Section~\ref{sec:mutation and purity}. Nevertheless, the following definition is helpful in characterising when mutation is possible.

\begin{definition}\label{def: Cosilt}
Let $\Dcal$ be a triangulated category with products and $\mathbb{T}$ a t-structure with heart $\Hcal$. A torsion pair $\mathfrak{t}:=(\Tcal,\Fcal)$ in $\Hcal$ is said to be a \textbf{cosilting torsion pair} if and only if there is a cosilting object $\sigma$ in $\Dcal$ such that $\mathbb{T}_{\mathfrak{t}^-}=\mathbb{T}_\sigma$. We denote the collection of cosilting torsion pairs in $\Hcal$ by $\Cosilt\Hcal$.  If $\Dcal$ is compactly generated, then we denote the set of cosilting torsion pairs in $\Hcal$ arising from pure-injective cosilting objects by $\CosiltP\Hcal$.
\end{definition}

\begin{example}\label{expl:cosiltR}
Let $R$ be a ring. The modules  $C=H^0(\sigma)$ arising as zero cohomologies of a cosilting complex $\sigma\colon I^0\to I^1$ of length two, concentrated in degrees $0$ and $1$, are precisely the \textbf{cosilting modules} introduced in \cite{BP}. The t-structure $\mathbb{T}_\sigma$ then coincides with the right HRS-tilt of the standard t-structure $\mathbb{T}=(\mathrm{D}^{\le -1},\mathrm{D}^{\ge 0})$ of $\D R$ at the torsion pair $\mathfrak{t}:=({}^{\perp_0}C,\,\Cogen C)$ in $\Mod{R}$  cogenerated by $C$.
In other words, $\mathfrak{t}$ is contained in $\Cosilt{R} := \Cosilt{\Mod R}$.

Furthermore, any cosilting torsion pair in $\Mod R$ is actually of this form. Indeed, let $\sigma\in\D R$ be  an associated cosilting complex  and $\mathbb{T}_\sigma=(\Xcal,\Ycal)$ the cosilting t-structure. Then $\sigma\in\Ycal\subseteq\mathrm{D}^{\ge0}$
is isomorphic to a complex of injectives concentrated in non-negative cohomological degrees. As on the other hand $\mathrm{D}^{\ge1}\subseteq\Ycal$, we have
$$ \sigma\in\Ycal^{\perp_1}\subseteq(\mathrm{D}^{\ge1})^{\perp_1}=(\mathrm{D}^{\ge2})^{\perp_0}, $$
which implies that such a complex of injectives is homotopy equivalent to a 2-term complex concentrated in degrees $0$ and $1$ (see e.g.\ \cite[Lemma 4.12]{PSV}).

Recall further that all cosilting complexes of length two in $\D R$ are pure-injective (Example~\ref{Ex: bdd pure-inj}). Thus $\Cosilt{R}= \CosiltP{\Mod R}$. In fact, it follows from \cite{BZ,WZ} that a torsion pair in $\Mod{R}$ is cosilting  if and only if it is of finite type. We are going to see in Proposition~\ref{prop:cover vs finite type} that this is a special instance of a general phenomenon.
\end{example}

The following proposition explains the relation between cosilting torsion pairs and mutation, providing a criterion for the existence of mutation.

\begin{proposition}\label{prop:general existence of mutations}
Let {$\Dcal$} be a triangulated category with products.
  Let $\sigma$ be a cosilting object and $\Escr=\Prod\Escr$  a subcategory of $\Prod{\sigma}$. Then we have:
\begin{enumerate}
\item $\sigma$ admits a   left mutation $\sigma'$  with respect to $\Escr$ if and only if there is  a cosilting torsion pair of the form $\mathfrak{t}:=(\Tcal,{}^{\perp_0}H^0_\sigma(\Escr))$ in $\Hcal_\sigma$; and
\item $\sigma$ admits a   right mutation $\sigma'$  with respect to $\Escr$ if and only if $\sigma$ admits an $\mathscr{E}$-precover and the  torsion pair $\mathfrak{t}=(\Scal,\Rcal)=({}^{\perp_0}H^0_\sigma(\Escr),\,\Cogen{H^0_\sigma(\Escr)})$ in $\Hcal_\sigma$  is a cosilting torsion pair.
\end{enumerate}
\end{proposition}
\begin{proof}
(1): By Theorem~\ref{thm:mutation as HRS-tilt}, we only need to prove the if-part. Suppose  that there is a cosilting torsion pair  $\mathfrak{t}=(\Tcal, \Fcal)$  in $\Hcal_\sigma$ where $\Fcal={}^{\perp_0}H^0_\sigma(\Escr)$. By definition, there is a cosilting object $\sigma'$ such that $\mathbb{T}_{\sigma'}$ coincides with the left HRS-tilt $\mathbb{T}_{\mathfrak{t}^+}=(\Xcal',\Ycal')$ of the t-structure $\mathbb{T}=\mathbb{T}_\sigma=(\Xcal,\Ycal)$ associated to $\sigma$.
As in the proof of Theorem~\ref{thm:mutation as HRS-tilt}, we see that $\Prod{\sigma'}$ coincides with $\Prod{\tilde{\sigma}}$ where  
$\tilde{\sigma}=\varepsilon_0\oplus \varepsilon_1$ is obtained from a triangle
$$\xymatrix{\sigma\ar[r]^\Phi& \varepsilon_0\ar[r]& \varepsilon_1\ar[r]& \sigma[1]}$$
where $\Phi$ is an $\Escr$-preenvelope. By Theorem~\ref{thm:mutation as HRS-tilt}, it thus only remains to show that $\Escr=\Prod{\sigma}\cap\Prod{\sigma'}$. For the inclusion $\Escr\subseteq \Prod{\sigma}\cap\Prod{\sigma'}$ we have to show that $\Escr\subseteq\Prod{\sigma'}=\Ycal'\cap\Ycal'^{\perp_1}$. Of course, $\Escr$ is contained in $\Ycal\subseteq\Ycal'$. Now pick $E$ in $\Escr$ and $X$ in $\Ycal'$, and consider a decomposition triangle with respect to the t-structure $\mathbb{T}$
 $$\xymatrix{A\ar[r]&X\ar[r]&B\ar[r]&A[1]}$$
where $A$ in $\Xcal$ and $B$ in $\Ycal\subseteq \Ycal'$. In particular, we have that  $A$ lies in $\Xcal\cap\Ycal'=\Fcal[1]$,  because $\Ycal'=\Ycal_{\mathfrak{t}^+}=\Fcal[1]\star\Ycal$. Applying $\Hom_\Dcal(-,E[1])$ to the triangle, it follows that
 $$\Hom_\Dcal(A,E[1])\cong \Hom_{\Hcal_\sigma}(A[-1],H^0_\sigma(E))=0.$$
 Further, since $E$ lies in $\Prod{\sigma}$, we have $\Hom_\Dcal(B,E[1])=0$ and, thus, $\Hom_\Dcal(X,E[1])=0$. To prove the other inclusion, let $X$ be an object in $\Prod{\sigma}\cap\Prod{\sigma'}$. Clearly, $H^0_\sigma(X)$ is injective in $\Hcal_\sigma$, because $X$ lies in $\Prod{\sigma}$. Furthermore, if $F$ is an object in $\Fcal$, since $F[1]$ lies in $\Ycal'={}^{\perp_{>0}}\sigma'$ and $X$ lies in $\Prod{\sigma'}$, we have
 $$\Hom_\Dcal(F,H^0_\sigma(X))\cong \Hom_\Dcal(F,X)\cong \Hom_\Dcal(F[1],X[1])=0.$$  
 It follows that $H^0_\sigma(X)$ lies in $\Fcal^{\perp_0}=\Cogen{H^0_\sigma(\Escr)}$. We conclude that $H^0_\sigma(X)$ lies in $H^0_\sigma(\Escr)$ and, thus, $X$ lies in $\Escr$.

(2): By Theorem~\ref{thm:mutation as HRS-tilt}, we only need to prove the if-part. Suppose  that $\mathfrak{t}=({}^{\perp_0}H^0_\sigma(\Escr),\,\Cogen{H^0(\Escr))}$ is  a cosilting torsion pair  in $\Hcal_\sigma$. Then there is a cosilting object $\sigma'$ such that $\mathbb{T}_{\sigma'}$ coincides with the right HRS-tilt $\mathbb{T}_{\mathfrak{t}^-}=(\Xcal',\Ycal')$ of   the t-structure $\mathbb{T}=(\Xcal,\Ycal)$ associated to $\sigma$.
By Theorem~\ref{thm:mutation as HRS-tilt}, it only remains to show that $\Escr=\Prod{\sigma}\cap\Prod{\sigma'}$. Now $\Escr$ is  contained in $\Prod{\sigma}$, and also in $\Ycal'$, which consists of those objects $X$ in $\Ycal$ for which $H^0_\sigma(X)$ lies in $\Cogen{H^0_\sigma(\Escr)}$ by Remark~\ref{rem:inversetilt}.
Moreover, given  $E$ in $\Escr$ and  $X$ in $\Ycal'$, we have
$\Hom_\Dcal(X,E[1])=0$ since $X$ lies in $\Ycal={}^{\perp_{>0}}\sigma$ and $E$ lies in $\Prod\sigma$.
We conclude that $E$ is an object of $\Ycal'\cap\Ycal'^{\perp_1}=\Prod{\sigma'}$.
Conversely, take $X$  in $\Prod{\sigma}\cap\Prod{\sigma'}$. Then $H^0_\sigma(X)$ is injective in $\Hcal_\sigma$, because $X$ lies in $\Prod{\sigma}$. Furthermore, it lies in $\Cogen{H^0_\sigma(\Escr)}$, because $X$ lies in $\Prod{\sigma'}\subseteq\Ycal'$. Hence $H^0_\sigma(X)$ lies in $H^0_\sigma(\Escr)$ and, thus, $X$ lies in $\Escr$.
\end{proof}


\section{Mutation and purity}\label{sec:mutation and purity}
As recalled in Theorem~\ref{thm:purity and silting}, pure-injective cosilting objects 
or pure-projective silting objects 
give rise to Grothendieck hearts. In order to investigate mutations of these objects, we need to understand when HRS-tilts are Grothendieck categories. Recall that a cocomplete abelian category is said to be \textbf{AB4} if coproducts are exact and it is said to be \textbf{AB5} if all direct limits are exact.

\subsection{HRS-tilts with Grothendieck heart}
{Our {aim in this subsection} is to generalise~\cite[Theorem 1.2]{PS-add}, which characterises when the heart of an HRS-tilt is a Grothendieck category.
This requires some preparation, however. Suppose that $\Hcal$ is an AB4 abelian category, $I$ a directed poset and $\Xcal=(X_i, f_{ij}\colon X_i\to X_j \mid i,j\in I, i\ge j)$ a direct system in $\Hcal$. This system gives rise to a complex
$$
B\Xcal_\mathsf{aug}\colon\quad
\xymatrix{
\cdots \ar[r] &
\coprod\limits_{i_0<i_1<i_2} X_{i_0} \ar[r]^-{d_2} &
\coprod\limits_{i_0<i_1} X_{i_0} \ar[r]^-{d_1} &
\coprod\limits_{i_0\in I} X_{i_0} \ar[r]^-{d_0} &
\varinjlim \Xcal \ar[r] &
0.
}
$$
in $\Hcal$, where $d_0$ is the canonical morphism and $d_n\colon \coprod_{i_0<\dots<i_n}X_{i_0} \to \coprod_{j_0<\dots<j_{n-1}}X_{j_0}$ for $n\ge1$ are described as follows: if $(j_0<\dots<j_{n-1})$ is obtained from $(i_0<\dots<i_n)$ by removing $i_k$ for some $0\le k\le n$, then the component of $d_n$ between the corresponding summands of the coproducts is equal to $(-1)^k f_{i_0j_0}\colon X_{i_0} \to X_{j_0}$ (we stress that $f_{i_0j_0}=\id_{X_{i_0}}$ if $k>0$ as then $i_0=j_0$). All the other components of $d_n$ vanish by definition.

We denote by $B\Xcal$ the complex obtained from $B\Xcal_\mathsf{aug}$ by deleting the last term with the direct limit. Hence $B\Xcal$ is a complex over $\Hcal$ which consists just of coproducts of objects in $\Xcal$ and we place the term $\coprod_{i_0\in I} X_{i_0}$ in cohomological degree 0. Now we define for $n\ge 0$ functors
$$
\varinjlim\nolimits_n\colon \Hcal^I\longrightarrow \Hcal
$$
by putting $\varinjlim_n\Xcal := H^{-n}(B\Xcal)$. We summarise some well-known properties of these functors.

\begin{lemma}\label{lem:derived colim basic}
Given a short exact sequence $0\to\Xcal\to\Ycal\to\Zcal\to 0$ of $I$-shaped direct systems in an AB4 abelian category $\Hcal$, there is a natural long exact sequence
$$
\cdots \to {\varinjlim}_2\Xcal \to {\varinjlim}_2\Ycal \to {\varinjlim}_2\Zcal \to {\varinjlim}_1\Xcal \to {\varinjlim}_1\Ycal \to {\varinjlim}_1\Zcal \to \varinjlim\Xcal \to \varinjlim\Ycal \to \varinjlim\Zcal \to 0.
$$
Moreover, the following statements are equivalent:
\begin{enumerate}
\item $\Hcal$ is AB5,
\item $\varinjlim_1\Xcal=0$ for each directed poset $I$ and an $I$-shaped direct system $\Xcal$,
\item $\varinjlim_n\Xcal=0$ for each directed poset $I$, an $I$-shaped direct system $\Xcal$ and $n>0$.
\end{enumerate}
\end{lemma}

\begin{proof}
The short exact sequence $0\to\Xcal\to\Ycal\to\Zcal\to 0$ yields a short exact sequence of complexes $0\to B\Xcal\to B\Ycal\to B\Zcal\to 0$ since we assume $\Hcal$ to be AB4.
The moreover part is standard: The implications $(3)\Rightarrow(2)\Rightarrow(1)$ are straightforward. Regarding $(1)\Rightarrow(3)$, one obeserves that $B\Xcal_\mathsf{aug}$ is a direct limit of contractible complexes, hence exact assuming~$(1)$. Indeed, $I$ is easily seen to be a direct union of finite directed subsets $F\subseteq I$, and each such $F$ has a unique maximal element  $\max(F)\in F$ (which is an upper bound of all the elements in $F$). The corresponding complex for the restricted direct system $\Xcal|_F$ is canonically identified with a subcomplex of $B\Xcal_\mathsf{aug}$ of the form
$$
B(\Xcal|_F)_\mathsf{aug}\colon\quad
\xymatrix{
\cdots \ar[r] &
\coprod\limits_{i_0<i_1<i_2} X_{i_0} \ar[r]^-{d_2} &
\coprod\limits_{i_0<i_1} X_{i_0} \ar[r]^-{d_1} &
\coprod\limits_{i_0\in F} X_{i_0} \ar[r]^-{d_0} &
X_{\max(F)} \ar[r] &
0
}
$$
and is contractible. A particular nullhomotopy is given by $(s_n)_{n\ge0}$, where
\[ s_n\colon \coprod_{j_0<\dots<j_{n-1}}X_{j_0} \longrightarrow \coprod_{i_0<\dots<i_n}X_{i_0} \]
with the components $(-1)^n\id_{X_{j_0}}$ if
$(i_0<\dots<i_{n-1}<i_n)=(j_0<\dots<j_{n-1}<\max(F))$, and zero maps otherwise.
The special case of $s_0\colon X_{\max(F)}\to\coprod_{i_0\in F} X_{i_0}$ is the obvious split inclusion.
Finally, one observes that $B\Xcal_\mathsf{aug}=\varinjlim_F B(\Xcal|_F)_\mathsf{aug}$, where $F$ runs over all finite directed subsets of $I$, ordered by inclusion.
\end{proof}

Now we focus on how the $\varinjlim_n$-functors interact with HRS tilting.

\begin{lemma}\label{lem:HRS and derived colim}
Let $\Dcal$ be a triangulated category with coproducts, $\mathbb{T}$ be a smashing t-structure with an AB5 heart $\Hcal$, and let $\mathfrak{t}=(\Tcal,\Fcal)$ be a torsion pair {in $\Hcal$}.
We recall that the heart $\Hcal_{\mathfrak{t}^-}$ of the right HRS-tilt has a torsion pair $(\Fcal,\Tcal[-1])$.
Then the following hold:
\begin{enumerate}
\item If $\Xcal$ is a direct system in $\Tcal$ and $T=\varinjlim \Xcal$ in $\Hcal$, then $T[-1]$ is the direct limit of $\Xcal[-1]$ in $\Hcal_{\t^-}$ and $\varinjlim_n\Xcal[-1]=0$ in $\Hcal_{\t^-}$ for each $n>0$.
\item Assume moreover that $(\Tcal,\Fcal)$ is of finite type in $\Hcal$. If $\Xcal$ is a direct system in $\Fcal$ and $F=\varinjlim \Xcal$ in $\Hcal$, then $F$ is also the direct limit of $\Xcal$ in $\Hcal_{\t^-}$ and $\varinjlim_n\Xcal=0$ in $\Hcal_{\t^-}$ for each $n>0$. The 
\end{enumerate}
\end{lemma}

\begin{proof}
We prove (2) only as the proof of (1) is completely analogous. Since the complex $B\Xcal_\mathsf{aug}$ is acyclic in $\Hcal$ by Lemma~\ref{lem:derived colim basic}, we have induced triangles in $\Dcal$,
$$
\xymatrix{
Z_{n+1} \ar[r] & \coprod_{i_0<\dots<i_n}X_{i_0} \ar[r] & Z_n \ar[r] & Z_{n+1}[1]
}
$$
for all $n\ge 0$. Here, $Z_n=\im{d_n}$ are the images in $\Hcal$ and the coproducts can be taken equally well in $\Hcal$ and in $\Dcal$ as $\mathbb{T}$ is smashing. Since $\Fcal$ is closed under coproducts ($\Hcal$ is AB5) and subobjects, we have $Z_n\in\Fcal$ for each $n>1$. We also have $Z_0\in\Fcal$ since we assume $\t=(\Tcal,\Fcal)$ to be of finite type. Consequently, the triangles above induce short exact sequences in $\Hcal_{\t^-}$.
Note further that $\mathbb{T}_{\t^-}$ is a smashing t-structure and $\Hcal_{\t^-}$ is closed under coproducts in $\Dcal$. It follows that $\Hcal_{\t^-}$ is AB4, that the coproducts in the triangles above are also coproducts in $\Hcal_{\t^-}$. Thus, the complexes $B\Xcal_\mathsf{aug}$ taken in $\Hcal$ and $\Hcal_{\t^-}$ coincide and are both acyclic in the corresponding hearts.
\end{proof}

Now we can give the promised generalisation of~\cite[Theorem 1.2]{PS-add}.

\begin{theorem}\label{thm:HRS-tilt AB5}
Let $\Dcal$ be a triangulated category with coproducts, $\mathbb{T}$ be a smashing t-structure with an AB5 heart $\Hcal$, and let $\mathfrak{t}=(\Tcal,\Fcal)$ be a torsion pair {in $\Hcal$}. Then the torsion pair $(\Fcal,\Tcal[-1])$ is of finite type in the right HRS-tilt $\Hcal_{\t^-}$ and the following statements are equivalent.
\begin{enumerate}
\item $\Hcal_{\t^-}$ is AB5.
\item The torsion pair $\mathfrak{t}=(\Tcal,\Fcal)$ is of finite type in $\Hcal$.
\end{enumerate}
{If $\Dcal$ is compactly generated, the conditions above are further equivalent to
\begin{enumerate}
\item[(3)]  $\Hcal_{\t^-}$ is a Grothendieck category.
\end{enumerate}}
\end{theorem}

\begin{proof}
The torsion pair $(\Fcal,\Tcal[-1])$ is of finite type in $\Hcal_{\t^-}$ thanks to Lemma~\ref{lem:HRS and derived colim}(1). If $(\Tcal,\Fcal)$ is of finite type in $\Hcal$, then $\varinjlim_1$ vanishes in $\Hcal_{\t^-}$ on all direct systems in $\Fcal$ or $\Tcal[-1]$ by Lemma~\ref{lem:HRS and derived colim}. If $\Xcal$ is any direct system in $\Hcal_{\t^-}$, we obtain a short exact sequence of direct systems as follows
$$\xymatrix{0 \ar[r]& \Xcal_\Fcal \ar[r]& \Xcal \ar[r]& \Xcal_{\Tcal[-1]} \ar[r]& 0,}$$
where $\Xcal_\Fcal$ and $\Xcal_{\Tcal[-1]}$ are the direct system of torsion and torsion-free parts of $\Xcal$ with respect to $(\Fcal,\Tcal[-1])$, respectively. If we take the direct limit, by Lemma~\ref{lem:derived colim basic}, we obtain an exact sequence
$$
0 = {\varinjlim}_1\Xcal_\Fcal \to {\varinjlim}_1\Xcal \to {\varinjlim}_1\Xcal_{\Tcal[-1]}=0.
$$
and the same lemma tells us (since, by \cite[Proposition 3.3]{PS}, $\Hcal_{\t^-}$ is AB4), that $\Hcal_{\t^-}$ is also AB5.

Suppose conversely that $\Hcal_{\t^-}$ is AB5. Then we can apply the previous arguments to $\Hcal_{\t^-}$ with torsion pair $(\Fcal,\Tcal[-1])$. As the HRS-tilt is equivalent to $\Hcal$ with $(\Tcal,\Fcal)$, it follows that the latter torsion pair is of finite type.

Now assume that $\Dcal$ is compactly generated, with $\Dcal^c$ denoting its subcategory of compact objects. Suppose that (1) holds true. In order to prove that  $\Hcal_{\t^-}$ is a Grothendieck category, it remains to exhibit a generator. We claim that $\coprod H^0_{\t^-}(C)$, where the coproduct runs over all isoclasses of compact objects $C$ in $\Dcal$, is a generator of $\Hcal_{\t^-}$. Indeed, given any $X$ in $\Hcal_{\t^-}$, the canonical map
$$\Phi\colon\coprod C^{(\Hom_\Dcal(C,X))}\longrightarrow X$$
is a pure epimorphism. Note that $H^0_{\t^-}\colon\Dcal\longrightarrow\Hcal_{\t^-}$ sends pure triangles in $\Dcal$ to short exact sequences by~\cite[Corollary 2.5]{Krause-TC}. Therefore, $H^0_{\t^-}(\Phi)$ is an epimorphism, which proves the claim.
\end{proof}
}

\subsection{Mutation of pure-injective cosilting objects} We are now ready to examine mutation of pure-injective cosilting objects. We will see that this setting is somewhat nicer than the general setting explored in Section 3.

In view of Proposition~\ref{prop:general existence of mutations}, we begin by investigating cosilting torsion pairs in Grothendieck hearts.
As indicated in Example~\ref{expl:cosiltR}, when $\mathfrak{t}=(\Tcal,\Fcal)$ is a cosilting torsion pair in the heart $\Hcal$ of a t-structure $\mathbb{T}$ in $\Dcal$, and $\sigma$ is a cosilting object such that $\mathbb{T}_{\mathfrak{t}^-}=\mathbb{T}_\sigma$, it is interesting to study properties of the object $C=H^0_\mathbb{T}(\sigma)$. As the following lemma says, this object in fact determines the torsion pair $\mathfrak{t}$ and hence the cosilting object $\sigma$ up to equivalence. To this end, recall that an object $C$ in a complete abelian category $\Hcal$ is called \textbf{quasi-cotilting} if $\Cogen{C}=\gen{\Cogen{C}}\cap {^{\perp_1}C}$ (one can also define quasi-tilting objects and prove a dual analogous version of the following lemma, but it is not needed here).

\begin{lemma} \label{lem:quasi-cotilting}
Let $\Dcal$ be a triangulated category with products and $\mathbb{T}=(\Xcal,\Ycal)$ be a t-structure with heart $\Hcal$.
Suppose that $\Hcal$ is complete and that $\mathfrak{t}=(\Tcal,\Fcal)$ is a cosilting torsion pair in $\Hcal$, and $\sigma$ is a cosilting object such that $\mathbb{T}_{\mathfrak{t}^-}=\mathbb{T}_\sigma$. Then the object $C=H^0_\mathbb{T}(\sigma)$ is quasi-cotilting in $\Hcal$ and $\Fcal=\Cogen C$.
\end{lemma}

\begin{proof}
We will adapt the argument for~\cite[Theorem 4.1(3)$\Rightarrow$(2)]{PSV}.
Recall that there is a torsion pair $(\Fcal,\Tcal[-1])$ in the heart $\Hcal_\sigma$ of $\mathbb{T}_{\mathfrak{t}^-}=(\Xcal_{\mathfrak{t}^-},\Ycal_{\mathfrak{t}^-})$, and that the torsion part of $X\in\Hcal_\sigma$ with respect to this torsion pair equals $H^0_\mathbb{T}(X)$ \cite[Corollary 2.2]{HRS}. Hence the functor $H^0_\mathbb{T}|_{\Hcal_\sigma}\colon \Hcal_\sigma\longrightarrow\Hcal$, being a composition of the right adjoints $\Hcal_\sigma\longrightarrow\mathcal\Fcal$ and $\mathrm{inc}\colon\Fcal\longrightarrow\Hcal$, is itself a right adjoint.
Since $E:=H^0_\sigma(\sigma)$ is an injective cogenerator of $\Hcal_\sigma$ by Proposition~\ref{prop:summary cosilting}, any object $F\in\Fcal$ admits a monomorphism $F\longrightarrow E^I$ for some set $I$, which induces a monomorphism $F=H^0_\mathbb{T}(F)\longrightarrow H^0_\mathbb{T}(E)^I$ in $\Hcal$. Moreover, since $\sigma\in\Ycal_{\mathfrak{t}^-}\subseteq\Ycal$, we observe that $E$ is an $\Xcal_{\mathfrak{t}^-}[-1]$-coreflection of $\sigma$ and $C$ is an $\Xcal[-1]$-coreflection of $\sigma$, which is the same as an $\Xcal[-1]$-coreflection of $E\in\Ycal_{\mathfrak{t}^-}$. In particular, $C=H^0_\mathbb{T}(E)$, the object $F$ embeds in $C^I$ and, hence, $\Fcal\subseteq\Cogen{C}$. Since clearly $C\in\Fcal$ and $\Fcal$ is closed under products and subobjects, we obtain the equality $\Fcal=\Cogen{C}$.

In order to prove that $C$ is quasi-cotilting, we first claim that $\Ext^1_\Hcal(F,C)=\Hom_\Dcal(F[-1],C)=0$ for each $F$ in $\Fcal=\Xcal[-1]\cap\Ycal_{\mathfrak{t}^-}$. To see this,  recall from the discussion above that there is a decomposition triangle 
 $$\xymatrix{Y[-2]\ar[r]&C\ar[r]&\sigma\ar[r]&Y[-1],}$$
where $Y\in\Ycal$ (and $C\in\Xcal[-1]$). Notice now that $\Hom_\Dcal(F[-1],Y[-2])=0$ since $F[-1]\in\Xcal[-2]$, and that $\Hom_\Dcal(F[-1],\sigma)=0$  as $F[-1]\in\Ycal_{\mathfrak{t}^-}[-1]$ and $\sigma\in(\Ycal_{\mathfrak{t}^-})^{\perp_1}$. It follows from the decomposition triangle that $\Hom_\Dcal(F[-1],C)=0$, proving the claim.

It remains to check that every object $X\in\gen{\Cogen{C}}\cap {^{\perp_1}C}$ lies in $\Fcal$.
Each such $X$ is  part of a short exact sequence in $\Hcal$ of the form
 $$\xymatrix{0\ar[r]&F'\ar[r]&F\ar[r]&X\ar[r]&0}$$
with $F',F\in\Fcal$.
 First note that since $C$ is an $\Xcal[-1]$-coreflection of $\sigma$, we have a natural isomorphism $\Hom_\Hcal(-,C)\cong\Hom_\Dcal(-,\sigma)|_\Hcal$. Applying $\Hom_\Hcal(-,C)$ to the above exact sequence and using this natural isomorphism, we obtain a commutative diagram with exact rows,
$$
\xymatrix{
\Hom_\Hcal(F,C)\ar[r]\ar[d]_-\cong&\Hom_\Hcal(F',C)\ar[r]\ar[d]^-\cong&\Ext^1_\Hcal(X,C)=0 \\
\Hom_\Dcal(F,\sigma)\ar[r]&\Hom_\Dcal(F',\sigma)\ar[r]&\Hom_\Dcal(X[-1],\sigma)\ar[r]&\Hom_\Dcal(F[-1],\sigma)=0.
}
$$
It follows that $\Hom_{\Hcal_\sigma}(H_\sigma^{-1}(X),E)\cong\Hom_\Dcal(X[-1],\sigma)=0$, so $H_\sigma^{-1}(X)=0$ and $X\in\Hcal\cap\Hcal_\sigma=\Fcal$.
\end{proof}

We now characterise cosilting torsion pairs in Grothendieck hearts coming from pure-injective cosilting objects. A similar result is proved in \cite[Theorem A]{PSV} (see also~\cite[Proposition 5.19]{PSV-cotilting}) when the ambient category $\Dcal$ is the derived category of $\Hcal$.

\begin{proposition}\label{prop:cover vs finite type}
Let $\sigma$ be a pure-injective cosilting object in a compactly generated triangulated category $\Dcal$. Let $\Hcal$ be the heart of the associated t-structure and let $\mathfrak{t}=(\Tcal,\Fcal)$ be a torsion pair in $\Hcal$. The following statements are equivalent.
\begin{enumerate}
\item The torsion pair $\mathfrak{t}$ is of finite type.
\item The torsion-free class $\Fcal$ is covering in $\Hcal$.
\item Any injective cogenerator of $\Hcal$ admits an $\Fcal$-cover.
\item There is a quasi-cotilting object $C$ in $\Hcal$ such that $\Fcal=\Cogen{C}$.
\item The torsion pair $\mathfrak{t}$ is a cosilting torsion pair  arising from a pure-injective cosilting object.
\item The torsion pair $\mathfrak{t}$ is cosilting.
\end{enumerate}
\end{proposition}

Before we prove the proposition, we point out an immediate corollary, which goes back to~\cite[Theorem 2.8]{Ba-cotilti-pi}, later generalised in~\cite[Theorem 3.9]{CoSt}.

\begin{corollary}\label{cor:cotilting pure-injective}
Let $\sigma$ be a pure-injective cosilting object in a compactly generated triangulated category $\Dcal$ and let $\Hcal$ be the heart of the associated t-structure. Then every cosilting torsion pair in $\Hcal$ arises from a pure-injective cosilting object, i.e.\ $\Cosilt{\Hcal}=\CosiltP{\Hcal}$.
\end{corollary}

\begin{proof}[Proof of Proposition~\ref{prop:cover vs finite type}]
Let $\mathbb{T}$ denote the t-structure associated to $\sigma$. Recall from Theorem~\ref{thm:purity and silting} that $\mathbb{T}$ is a smashing nondegenerate t-structure and that $\Hcal$ is a Grothendieck category. 

(1) $\Rightarrow$ (2): If $\Fcal$ is closed under direct limits, it follows from {\cite[Theorem~3.2]{ElBa}}  that $\Fcal$ is covering.

(2) $\Rightarrow$ (3): This is trivial.

(4) $\Rightarrow$ (1): This follows from \cite[Theorem A]{PSV}.

(3) $\Rightarrow$ (4): Let $E$ be an injective cogenerator of $\Hcal$ and let  $f\colon F\longrightarrow E$ be an $\Fcal$-cover. Consider the exact sequence
$$\xymatrix{0\ar[r]&K\ar[r]^h&F\ar[r]^f&E}.$$
We  {claim that $C=F\oplus K$ is quasi-cotilting. First note that
$\Fcal=\Cogen{F}=\Cogen{C}$.} Indeed $\Cogen{F}$ is clearly contained in
$\Fcal$ and, conversely, for any object $X$ in $\Fcal$, any given monomorphism $g\colon X\longrightarrow E^I$ (which exists for some set $I$) will factor through the precover $f^I$ via a monomorphism.

Next, we show the equality $\Cogen{C}=\gen{\Cogen{C}}\cap {}^{\perp_1}C$. Let us first verify that $\Cogen{C}\subseteq  {}^{\perp_1}C$. If $X$ lies in $\Fcal$, consider an exact sequence
$$\xymatrix{0\ar[r]&F\ar[r]^\alpha&Y\ar[r]^\beta&X\ar[r]&0}.$$
{S}ince $\Fcal=\Cogen{F}$ is extension-closed, there is a monomorphism $\epsilon\colon Y\longrightarrow F^J$ for some set $J$. Since $E$ is injective, the map $f$ extends along $\epsilon\circ\alpha$ to a map $g\colon F^{I}\longrightarrow E$, i.e. $g\circ\epsilon\circ\alpha=f$. Since $f$ is an $\Fcal$-(pre)cover of $E$, there is $h\colon F^I\longrightarrow F$ such that $f\circ h=g$. Finally, we observe that since $f$ is right minimal, $f\circ h\circ\epsilon\circ\alpha=f$ implies that $h\circ \epsilon\circ \alpha$ is an isomorphism and, therefore, $\alpha$ splits. This shows that $\Ext^1_\Hcal(X,F)=0$. The {condition} $\Ext^1_\Hcal(X,K)=0$ follows from {the fact that $f$ is an $\Fcal$-cover by a result known as Wakamatsu's Lemma}. This completes the proof of this inclusion.

Let now $X$ be an object in $\gen{\Cogen{F}}\cap {^{\perp_1}K}$ and consider a short exact sequence
$$
\xymatrix{0\ar[r]&L\ar[r]^a&Y\ar[r]^b&X\ar[r]&0}
$$
with $Y$ in $\Cogen{{F}}$. Let $\phi\colon X\longrightarrow E$ be a non-zero map. Then $\phi\circ b$ factors through $f$ since $f$ is an $\Fcal$-(pre)cover, i.e. there is map $d\colon Y\longrightarrow F$ such that $f\circ d=\phi\circ b$. This induces a map $c\colon L\longrightarrow K$ such that $h\circ c=d\circ a$. Since $\Ext^1_\Hcal(X,K)=0$, it follows that $\Hom_\Hcal(a,K)$ is surjective and, thus, there is a map $g\colon Y\longrightarrow K$ such that $g\circ a=c$. In summary, up to this moment, we have the {solid part of the} following diagram with exact rows
$$
\xymatrix{0\ar[r]&L\ar[r]^a\ar[d]^c&Y\ar[ld]^g\ar[r]^b\ar[d]^d&X\ar[r]\ar[d]^\phi\ar@{.>}[ld]^\alpha&0\\ 0\ar[r]&K\ar[r]^h&F\ar[r]^f&E}
$$
{Now a standard argument (using for example a dual version of \cite[Lemma 2.8]{STH14}) yields a map $\alpha\colon X\longrightarrow F$, indicated by the dotted arrow above, such that $f\circ \alpha=\phi$.}
{We have shown that $\Hom_\Hcal (X,f)$ is a surjective map. As above, we infer that any given monomorphism $g\colon X\longrightarrow E^I$  factors through  $f^I$ via a monomorphism, }
showing that $X$ lies in $\Cogen{F}$, as wanted.

(1) $\Rightarrow$ (5): If $\mathfrak{t}$ is of finite type, $\mathbb{T}_{\mathfrak{t}^-}$ is a smashing t-structure with a Grothendieck heart by Theorem~\ref{thm:HRS-tilt AB5}, so it corresponds to a pure-injective cosilting object in $\Dcal$ by Theorem~\ref{thm:purity and silting}.

(5) $\Rightarrow$ (6): This is trivial.

(6) $\Rightarrow$ (4): This is has been proved in Lemma~\ref{lem:quasi-cotilting}.
\end{proof}

For pure-injective cosilting objects in compactly generated triangulated categories, we can now provide a more convenient characterisation for the existence of  left or right mutation.
We first need some preliminary results on the existence of approximations.

\begin{lemma}\label{lem: covers and envelopes}
Let $\Acal$ be {a (not necessarily additive)} category, $\Rcal$ a covering class and $\Ical$ an enveloping class.  The class $\Rcal$ is closed under $\Ical$-envelopes if and only if the class $\Ical$ is closed under $\Rcal$-covers.
\end{lemma}
\begin{proof}
Suppose that $\Rcal$ is closed under $\Ical$-envelopes. First we show that $\Ical$ is closed under retracts.
Suppose that we have an object $J$ in $\Ical$ together with a retraction $\pi\colon J\to I$ and its left inverse $\iota\colon I\to J$.
Consider an $\Ical$-envelope $f \colon I \to E_\Ical(I)$ of $I$.
As $f$ is an $\Ical$-envelope, there exists a factorisation $\iota = gf$ for some $g \colon E_\Ical(I) \to J$.  Then $(\pi g) f = \pi \iota = \id$ and so $f$ has a left inverse.  Moreover, we have that $f(\pi g)f = f \pi \iota = f$ so, by the minimality of $f$, we have that $f(\pi g)$ is an isomorphism, so $f$ has a right inverse. Thus, we conclude that $f$ is an isomorphism and $I$ lies in $\Ical$. Now, let $h \colon C_\Rcal(I) \to I$ be an $\Rcal$-cover of an object $I$ in $\Ical$ and let $e \colon C_\Rcal(I) \to E_\Ical(C_\Rcal(I))$ be an $\Ical$-envelope of $C_\Rcal(I)$.  Since $e$ is an $\Ical$-envelope and $h$ is a morphism to $\Ical$, we have a factorisation $h = me$ for some $m \colon E_\Ical(C_\Rcal(I)) \to I$.  Moveover, since $E_\Ical(C_\Rcal(I))$ lies in $\Rcal$ and $h$ is an $\Rcal$-cover, there is a factorisation $m = hk$ for some $k \colon  E_\Ical(C_\Rcal(I)) \to C_\Rcal(I)$.  Then we have $hke = me = h$, and so {$ke$ is an isomorphism} because $h$ is minimal.  We have shown that $C_\Rcal(I)$ is a {retract} of $E_\Ical(C_\Rcal(I))$ and so $C_\Rcal(I)$ lies in $\Ical$. The converse statement is dual.
\end{proof}

\begin{lemma}\label{rem:approximate every injective}
Let $\sigma$ be a pure-injective cosilting object in a compactly generated triangulated category $\Dcal$, and $\Escr=\Prod\Escr$ a  subcategory of $\Prod{\sigma}$. Let $(\Scal,\Rcal)$ be the torsion pair in $\Hcal_\sigma$ cogenerated by~$H^0_\sigma(\Escr)$.
\begin{enumerate}
\item $\Escr$ is an enveloping class;
\item If  $(\Scal,\Rcal)$  is a cosilting torsion pair, then every object in $\Prod{\sigma}$ admits an $\Escr$-cover. 
\end{enumerate}
\end{lemma}
\begin{proof}
(1) Recall from Lemma~\ref{lem:cosilting enveloping} that $\Escr$ is preenveloping. Since $\sigma$ is pure-injective and $\Hcal_\sigma$ is a Grothendieck category, it follows that, in fact, $\Escr$ is enveloping.
Indeed, this is the same as proving that $H^0_\sigma(\Escr)$ is enveloping in $\Hcal_\sigma$, and such envelopes can be constructed as injective envelopes of the torsion-free part with respect to the hereditary torsion pair $(\Scal,\Rcal)$.

(2) By Proposition~\ref{prop:cover vs finite type}, the torsion-free class $\Rcal$ is  a covering class in $\Hcal_\sigma$. Hence, if $\gamma$ is in $\Prod{\sigma}$, the injective object $H^0_\sigma(\gamma)$ of $\Hcal_\sigma$ admits an  $\Rcal$-cover. Moreover,  since $(\Scal,\Rcal)$ is a hereditary torsion pair, $\Rcal$ is closed under injective envelopes and, thus, by Lemma \ref{lem: covers and envelopes}, injectives are closed under $\Rcal$-covers. The assertion then follows by Lemma~\ref{lem:approx prop}.
\end{proof}

The following theorem refines Proposition~\ref{prop:general existence of mutations} in case  the ambient category is compactly generated and the cosilting object  is pure-injective. It also says that  mutation of such objects is automatically again pure-injective.

\begin{theorem}\label{existence right mutation}\label{existence left mutation}
Let $\sigma$ be a pure-injective cosilting object in a compactly generated triangulated category $\Dcal$, and $\Escr=\Prod\Escr$ a  subcategory of $\Prod{\sigma}$. Let $(\Scal,\Rcal)$ be the torsion pair in $\Hcal_\sigma$ cogenerated by~$H^0_\sigma(\Escr)$.

\begin{enumerate}
\item The following statements are equivalent.
\begin{enumerate}
\item[(a)] $\sigma$ admits a  left mutation $\sigma'$  with respect to $\Escr$;
\item[(b)] the torsion class $\Scal={}^{\perp_0}H^0_\sigma(\Escr)$ in $\Hcal_\sigma$ is closed under products (that is, it is a TTF class);
\item[(c)] the object $\varepsilon_0\oplus \varepsilon_1$  arising from an $\Escr$-envelope 
 $\sigma\to \varepsilon_0$ and its cone $\varepsilon_1$ is a cosilting object.
\end{enumerate}
\item The following statements are equivalent.
\begin{enumerate}
\item[(a)] $\sigma$ admits a  right mutation $\sigma'$  with respect to $\Escr$;
\item[(b)] the torsion-free class $\Rcal=\Cogen{H^0_\sigma(\Escr)}$ in $\Hcal_\sigma$ is closed under direct limits;
\item[(c)] the cosilting object $\sigma$ admits an $\Escr$-cover.
\end{enumerate}
\end{enumerate}
In both cases, if the equivalent conditions are satisfied, any mutation $\sigma'$ as in \emph{(a)} is pure-injective.
\end{theorem}
\begin{proof}
The last assertion of the theorem is a consequence of Corollary \ref{cor:cotilting pure-injective}.

(1): (a)$\Rightarrow$(c):
By definition, there a cosilting object $\sigma'$ and an approximation triangle relating $\sigma$ and $\sigma'$. Since $\sigma''=\varepsilon_0\oplus \varepsilon_1$ is a direct summand of product of copies of $\sigma'$ and $\sigma''$ cogenerates of $\Dcal$, we infer from Lemma~\ref{lem:(co)generate} that $\Prod{\sigma''}=\Prod{\sigma'}$, which proves the claim.

The implications
(a)$\Rightarrow$(b) and (c)$\Rightarrow$(b) follow immediately from Theorem~\ref{thm:mutation as HRS-tilt}(1).

(b)$\Rightarrow$(a): Since ${}^{\perp_0} H^0_\sigma(\Escr)$ is a torsion class, $\mathfrak{t}:=(\Tcal,{}^{\perp_0} H^0_\sigma(\Escr))$   is a torsion pair of finite type in the Grothendieck category $\Hcal_\sigma$, hence  a cosilting torsion pair associated to a pure-injective cosilting object by Proposition~\ref{prop:cover vs finite type}.
It follows from Proposition~\ref{prop:general existence of mutations}(1) that $\sigma$ admits a left mutation $\sigma'$  with respect to~$\Escr$. 

(2): Recall from Theorem~\ref{thm:purity and silting}(1) that $\Hcal_\sigma$ is a Grothendieck category. We apply Proposition~\ref{prop:cover vs finite type} to the torsion pair $(\Scal, \Rcal)$ in $\Hcal_\sigma$.

(a)$\Leftrightarrow$(b):  
By Proposition~\ref{prop:cover vs finite type}, a torsion pair in $\Hcal_\sigma$ is cosilting if and only if it is of finite type, and the associated cosilting object must be pure-injective in this case. It  follows from Proposition~\ref{prop:general existence of mutations} and Lemma~\ref{rem:approximate every injective} that (b) amounts to the existence of a right mutation $\sigma'$ of $\sigma$ with respect to $\Escr$.

(b)$\Leftrightarrow$(c): This follows combining Proposition~\ref{prop:cover vs finite type} with Lemma~\ref{rem:approximate every injective} and Lemma~\ref{lem:approx prop}.
\end{proof}

\begin{example}\label{expl:krone}
Let $A$ be the path algebra over an algebraically closed field $k$ of the Kronecker quiver
$$\xymatrix{\bullet\ar@<-.7ex>[r]\ar@<.7ex>[r]&\bullet}$$
Recall that the finite-dimensional indecomposable regular modules form a tubular family $(\tube_x)_{x\in\mathbb{X}}$ indexed by the projective line $\mathbb{X}=\mathbb{P}^1(k)$. We pick a subset $P\subseteq\mathbb{X}$, denote by $\bar{P}=\mathbb{X}\setminus P$ its complement,  and consider the torsion pair $(\Tcal_P, \Fcal_P)$ in $\Mod{A}$ generated by $\tube_P=\bigcup_{x\in P} \tube_x$. If $P=\varnothing$, we take the torsion pair generated by the preinjective modules (whose torsion class is indeed contained in $\Tcal_Q$ for all $Q\ne\varnothing$). It is a cosilting torsion pair cogenerated by the cosilting (in fact, even cotilting) module
$$C_P=G\oplus\prod \{S[-\infty]\,\mid\, S\in \tube_P\}\oplus \coprod\{S[\infty]\,\mid\, S\in \tube_{\bar{P}}\},$$ where $S[-\infty]$ and $S[\infty]$ denote the adic and the Pr\"ufer module corresponding to the simple regular module $S$, respectively, and $G$ is the generic module. Notice that the corresponding cosilting complex $\sigma_P$ is quasi-isomorphic to $C_P$ and lies in the heart $\Hcal_{\sigma_P}=\sigma_P\,^{\perp_{\not=0}}$.
       
Let us look at the two extreme cases
$$C_\emptyset=G\oplus \coprod\{S[\infty]\,\mid\, S\in \tube_{\mathbb{X}}\},\quad\text{and}\quad C_{\mathbb{X}}=G\oplus\prod\{S[-\infty]\,\mid\, S\in \tube_\mathbb{X}\}.$$
For any $P\subseteq\mathbb{X}$, we have that  
$\sigma_P$ is a right mutation of $\sigma_{\emptyset}$ at the set
$$\Escr=\Add{G\oplus \coprod\{S[\infty]\,\mid\, S\in \tube_{\bar{P}}\}}$$
(if $P\ne\mathbb{X}$, we can also express $\Escr$ as $\Prod{\{S[\infty]\,\mid\, S\in \tube_{\bar{P}}\}}$). Indeed, we can construct an $\Escr$-cover of $\sigma_{\emptyset}$ from the canonical sequences
  $$0\longrightarrow S[-\infty]\longrightarrow G^{(I)}\longrightarrow S[\infty]\longrightarrow 0,\quad S\in\tube_P$$
as in~\cite[Lemma 2.4]{BuKr}, which are easily seen to be $\Escr$-covers since $\Ext^1_A(G,S[-\infty])=0$ and $S[-\infty]$ is indecomposable. When taking a product of these short exact sequences for all $S\in P$ together with the trivial short exact sequences $0\longrightarrow0\longrightarrow S[\infty]\longrightarrow S[\infty]\longrightarrow 0$ for all $S\in\bar{P}$, we obtain a short exact sequence of the form
$$0\longrightarrow \prod \{S[-\infty]\,\mid\, S\in \tube_P\}\longrightarrow G^{(J)}\oplus \prod\{S[\infty]\,\mid\, S\in \tube_{\bar{P}}\}\overset{\phi}\longrightarrow \prod_{S\in\mathbb{X}} S[\infty]\longrightarrow 0.$$
Since the middle term lies in $\Escr$, the map $\phi$ is an $\Escr$-precover. Moreover, the right hand side term is a cotilting module equivalent to $C_\varnothing$ (the generic module is a summand of the term by  \cite[Proposition 4]{Ri}) while the sum of the left hand and the middle terms is a cotilting module equivalent to $C_P$,
so the last short exact sequence yields an approximation triangle witnessing that $\sigma_P$ is a right mutation of $\sigma_\emptyset$.

On the other hand,  $\sigma_\mathbb{X}$ does {\it not} admit  right mutation at   $\Escr=\Prod {\{S[-\infty]\,\mid\, S\in \tube_P\}}$ for any non-empty subset $P\subset\mathbb{X}$. In fact,   condition (1)(b) in Theorem~\ref{existence right mutation} fails, due to the fact that  the generic module  $G$ is not contained in the torsion-free class $\Cogen\Escr$ in $\Hcal_{\sigma_{\mathbb{X}}}$, although it can be realised as a direct summand of a direct limit of a direct system $S[-\infty]\longrightarrow S[-\infty]\longrightarrow\cdots$ for any simple regular $S$ by \cite[Proposition 4]{Ri}.
 
Similarly,  $\sigma_\emptyset$ does not admit a left mutation at $\Escr=\Prod {\{S[\infty]\,\mid\, S\in \tube_P\}}$ for any proper subset $P\subset\mathbb{X}$, because condition (2)(b) in Theorem~\ref{existence left mutation} fails. Namely, ${}^{\perp_0}H^0_{\sigma_\emptyset}(\Escr)$ contains any $S[\infty]$ with $S\in\tube_{\bar{P}}$, but the generic module $G$ is not contained in  ${}^{\perp_0}H^0_{\sigma_\emptyset}(\Escr)$, although it can be realised as a direct summand of a direct product of $S[\infty]$ for any simple regular $S$ by \cite[Proposition 4]{Ri}.
\end{example}

\begin{example}\label{prop: simple hereditary}
Let $\Dcal$ be a compactly generated triangulated category and let $\sigma$ be a pure-injective cosilting object in $\Dcal$. Recall from Theorems~\ref{thm:purity and silting} and~\ref{existence right mutation} that the associated t-structure has a Grothendieck heart $\Hcal$, and right mutations of $\sigma$ bijectively correspond to hereditary torsion pair of finite type in $\Hcal$. When $\Dcal = \D{R}$ for a commutative noetherian ring $R$, it follows from \cite[Theorem 5.1]{AH1} that every two-term cosilting complex in $\D{R}$ is a right mutation of the injective minimal cogenerator of $\Mod{R}$.

When $\Dcal=\D{\Qcoh{\mathbb{X}}}$  for a noetherian scheme $\mathbb{X}$, a class of mutations of the injective cogenerator of $\Qcoh{\mathbb{X}}$  was  studied in~\cite[\S6]{CoSt} (albeit not in this terminology).
However, in this case there can be more $2$-term cosilting complexes than hereditary torsion pairs of finite type (see~\cite[Example 6.14]{CoSt}, which is closely related to Example~\ref{expl:krone} as $\D{\Qcoh{\mathbb{P}^1_k}}\simeq\D{A}$).

Moreover, all cotilting modules over a commutative noetherian ring, and more generally,  all  cotilting modules of cofinite type over an arbitrary commutative ring can be constructed as iterated right mutations of an injective cogenerator of $\Mod{R}$, see \cite[\S4]{STH14} and \cite[\S8]{HrSt}.
\end{example}

In the following proposition, we will use the notation $\Ind{\sigma}$ for the collection of isoclasses of indecomposable objects in $\Prod{\sigma}$ where $\sigma$ is a pure-injective cosilting object.  Observe that, by Proposition~\ref{prop:summary cosilting}, we have that $\Ind{\sigma}$ is a set because the isomorphism classes of the indecomposable injective objects in the Grothendieck category $\Hcal_{\sigma}$ form a set.
We will show that if pure-injective cosilting objects $\sigma$ and $\sigma'$ are related by a mutation, there is a natural bijection between $\Ind{\sigma}$ and $\Ind{\sigma'}$. While in the case of silting mutation of compact silting objects over a finite-dimensional algebra this was implicit from the very beginning (see for example~\cite[Corollary 2.28]{AI}), in the case of cosilting mutation in our context, this phenomenon was noticed for derived categories of modules over commutative noetherian rings~\cite[Theorem 5.4]{STH14} and quasi-coherent sheaves on noetherian schemes~\cite[Remark 6.5]{CoSt}.

\begin{proposition}\label{prop: bij ind summands}
Let $\sigma$ be a pure-injective cosilting object in a compactly generated triangulated category $\Dcal$ and $\Escr=\mathsf{Prod}(\Escr)$ a subcategory of $\Prod{\sigma}$ such that there is a right mutation of $\sigma$ at $\Escr$, say $\sigma'$. For each object $\alpha$ in $\Ind{\sigma}\setminus \Ind\Escr$, consider the triangle induced by an $\Escr$-cover $\Phi$ of $\alpha$
\begin{equation} \tag{$\Delta_3$} \label{eq:mutation ind summand}
\xymatrix{\alpha'\ar[r]^\Omega&e_0\ar[r]^\Phi&\alpha\ar[r]&\alpha'[1]}.
\end{equation}
Then, the assignment $\alpha\mapsto \alpha'$ defines a bijection between $\Ind{\sigma}\setminus \Ind\Escr$ and $\Ind{\sigma'}\setminus \Ind\Escr$ and the map $\Omega$ in each such triangle is an $\Escr$-envelope. As a consequence, there is a bijection between $\Ind{\sigma}$ and $\Ind{\sigma'}$.
\end{proposition}

Note that by Corollary \ref{cor:inverse mutations} a result analogous to the one above is available for left mutations.

\begin{example} \label{expl:trivial mutation}
While reading the proof of the proposition, it is instructive to keep in mind that it also covers the case of trivial right mutations where $\Escr=0$ and $\sigma'=\sigma[-1]$.
\end{example}

\begin{proof}[Proof of Proposition~\ref{prop: bij ind summands}]
We first show that the assignment is well-defined. Let $\alpha$ be an indecomposable object in $\Prod{\sigma}\setminus \Escr$; this implies that $\alpha'\ne 0$.
Note also since $\Phi$ is an $\Escr$-cover of an indecomposable object, $\Phi$ is an indecomposable object in the category of morphisms in $\Dcal$. To see that, if $\Phi=\Phi_0\oplus\Phi_1$, then one of the $\Phi_i$ must be of the form $e_{0,i}\longrightarrow 0$, which contradicts the fact that $\Phi$ is a cover. Since the completion of a direct sum of maps to a triangle is isomorphic to the direct sum of the two triangles completing the summands, it follows that~\eqref{eq:mutation ind summand} is an indecomposable triangle (this makes sense since triangles in $\Dcal$ themselves form an additive category). Now it quickly follows that $\Omega$ is an $\Escr$-envelope. Indeed, an $\Escr$-envelope exists (Lemma~\ref{rem:approximate every injective}) and is a summand of $\Omega$. However, $\Omega$ must be indecomposable in the category of morphisms, or else \eqref{eq:mutation ind summand} could not be indecomposable in the category of triangles. By the same token, $\alpha'$ is indecomposable since otherwise that $\Escr$-envelope $\Omega$ could be expressed as a direct sum of $\Escr$-envelopes of summands of $\alpha'$. Finally, $\alpha'$ cannot lie in $\Escr$ as otherwise $\Omega$ had to be an isomorphism and $\alpha$ the zero object. This completes a proof of the fact that the assignment from the statement of the proposition is well-defined.

Further observe that, in particular, we have shown that $\alpha$ is determined up to isomorphism from $\alpha'$, and the assignment is injective. Regarding the surjectivity, suppose that $\alpha'$ is an indecomposable object in $\Prod{\sigma'}\setminus \Escr$. Since $\sigma'$ is (up to equivalence) the right mutation of $\sigma$ at $\Escr$, then $\sigma$ is (up to equivalence) the left mutation of $\sigma$ at $\Escr$ (see Corollary~\ref{cor:inverse mutations}). As every object in $\Prod{\sigma'}$ admits an $\Escr$-envelope (see Lemma~\ref{rem:approximate every injective}) we can take $\Omega\colon \alpha' \longrightarrow e_0$ to be this envelope and let $\alpha$ be its cone. We therefore obtain a triangle as in the statement of the proposition, and dual arguments to the ones presented above can be used to show that $\Phi$ is an $\Escr$-cover (which we know to exist, see Lemma~\ref{rem:approximate every injective}) and that $\alpha$ is indecomposable in $\Prod{\sigma}\setminus{\Escr}$, thus finishing the proof.
\end{proof}

\section{Silting mutation}\label{sec:silting mutation}
In this section, we state the dual results for silting objects. Then we establish a compatibility between silting and cosilting mutation, showing that cosilting mutation encompasses mutation between compact silting objects.

\begin{definition}\label{def:siltmut} Let $\Dcal$ be a triangulated category with coproducts. Let
$\sigma$ and $\sigma'$ be two silting objects in $\Dcal$,  and let $\Pscr=\Add{\sigma}\cap\Add{\sigma'}$. We say that
\begin{enumerate}
\item  $\sigma'$ is a \textbf{left mutation} of $\sigma$ if there is a triangle $\xymatrix{\sigma\ar[r]^\Phi& \varepsilon_0\ar[r]& \varepsilon_1\ar[r]& \sigma[1]}$ such that:
 \begin{itemize}
 \item $\Phi$ is a $\Pscr$-preenvelope  of $\sigma$ in $\Dcal$; and  
 \item $\varepsilon_0\oplus \varepsilon_1$ is a silting object equivalent to $\sigma'$.
 \end{itemize}
 \item $\sigma'$ is a \textbf{right mutation} of $\sigma$  if there is a triangle $\xymatrix{\sigma[-1]\ar[r]& \gamma_1\ar[r]& \gamma_0\ar[r]^\Phi&\sigma}$ such that:
 \begin{itemize}
 \item $\Phi$ is a $\Pscr$-precover  of $\sigma$ in $\Dcal$; and  
 \item $\gamma_0\oplus \gamma_1$ is a silting object equivalent to $\sigma'$.
 \end{itemize}
 \end{enumerate}
 We will also say that $\sigma'$ is a \textbf{left (or right) mutation} of $\sigma$ \textbf{with respect to $\Pscr$}.
\end{definition}

Silting mutation can also be expressed in terms of HRS-tilts. The proof of the following theorem is dual to the one of Theorem~\ref{thm:mutation as HRS-tilt}.

\begin{theorem}\label{thm:siltmutation as HRS}
Let $\Dcal$ be a triangulated category with coproducts, $\sigma$ and $\sigma'$ be two silting objects in $\Dcal$,  and let $\Pscr=\Add{\sigma}\cap\Add{\sigma'}$. Then we have that:
\begin{enumerate}
\item $\sigma'$ is a {left mutation} of $\sigma$ if and only if $\sigma$ admits a $\mathscr{P}$-preenvelope and
$\mathbb{T}_{\sigma'}$ is the left HRS-tilt of $\mathbb{T}_{\sigma}$ at the torsion pair $\mathfrak{t}=(\Gen{H^0_\sigma(\Pscr)}, \, H^0_\sigma(\Pscr)^{\perp_0})$ in $\Hcal_\sigma$; and
\item $\sigma'$ is a {right mutation} of $\sigma$ if and only if $H^0_\sigma(\Pscr)^{\perp_0}$ is closed under coproducts in $\Hcal_\sigma$ and
$\mathbb{T}_{\sigma'}$ is the right HRS-tilt of $\mathbb{T}_{\sigma}$ at the torsion pair $\mathfrak{t}=(H^0_\sigma(\Pscr)^{\perp_0},\Fcal)$ in $\Hcal_\sigma$.

\end{enumerate}
In both cases, the torsion pairs involved do not depend on the choice of the triangle in Definition~\ref{def:siltmut}.
\end{theorem}

Again, as noted for cosilting mutation in Remark~\ref{rem:different kinds of torsion}, the torsion pairs involved have different flavours: left mutation will yield an HRS-tilt at a cohereditary torsion pair, while right mutation will give rise to an HRS-tilt at a torsion pair $(\Tcal,\Fcal)$ for which $\Tcal$ is a TTF class.

\begin{definition}
Let $\Dcal$ be a triangulated category with coproducts and $\mathbb{T}$ a t-structure with heart $\Hcal$. A torsion pair $\mathfrak{t}:=(\Tcal,\Fcal)$ in $\Hcal$ is said to be a \textbf{silting torsion pair} if and only if there is a silting object $\sigma$ in $\Dcal$ such that $\mathbb{T}_{\mathfrak{t}^+}=\mathbb{T}_\sigma$.
\end{definition}

\begin{example}\label{expl:siltR}
Let $R$ be a ring. The modules  $T=H^0(\sigma)$ arising as zero cohomologies of a silting complex $\sigma: P_{1}\to P_0$ of length two concentrated in cohomological degrees $-1$ and $0$ are precisely the \textbf{silting modules} introduced in \cite{AMV1}. The t-structure $\mathbb{T}_\sigma$ then coincides with the left HRS-tilt of the standard t-structure of $\D R$ at the torsion pair $\mathfrak{t}:=(\Gen T,T^{\perp_0})$ in $\Mod{R}$  generated by $T$. In other words, $\mathfrak{t}$ is a silting torsion pair in the sense of the definition above,
and as in Example~\ref{expl:cosiltR}, one can show that all silting torsion pairs in the sense of the definition are of this form.
\end{example}

Again, we can extract from the theorem above a criterion for the existence of a mutation with respect to a given subset  $\Pscr$. Once again, we omit the proof as it is dual to the proof of Proposition~\ref{prop:general existence of mutations}.

\begin{proposition}\label{prop:general existence of silting mutations}
Let $\Dcal$ be a triangulated category with coproducts.
  Let $\sigma$ be a silting object and $\Pscr=\Add\Pscr$  a subcategory of $\Add{\sigma}$. Then we have that:
\begin{enumerate}
\item $\sigma$ admits a   left mutation $\sigma'$  with respect to $\Pscr$ if and only if $\sigma$ admits a $\mathscr{P}$-preenvelope and the pair $(\Gen{H^0_\sigma(\Pscr)},\,H^0_\sigma(\Pscr)^{\perp_0})$ in $\Hcal_\sigma$ is a silting torsion pair; and
\item $\sigma$ admits a   right mutation $\sigma'$  with respect to $\Pscr$ if and only if  there is a silting torsion pair of the form $\mathfrak{t}:=(H^0_\sigma(\Pscr)^{\perp_0},\Fcal)$ in $\Hcal_\sigma$.
\end{enumerate}
\end{proposition}

\begin{remark}\label{rem:silting is preenveloping}
In the case where $\Dcal$ is a compactly generated triangulated category, the condition that $\sigma$ admits a $\mathscr{P}$-preenvelope is redundant.

Indeed, let $\mathbb{T}=\mathbb{T}_\sigma$ be a  t-structure associated to a silting object $\sigma$,
let $\mathfrak{t}=(\Tcal,\Fcal)$ be a silting torsion pair in the heart $\Hcal_\sigma$
and let $\gamma$ be a silting object in $\Dcal$ such that $\mathbb{T}_{\mathfrak{t}^+}=\mathbb{T}_\gamma$. We know from \cite[Proposition 3.8]{AMV5} (see also \cite[Theorem 3.2.4]{Bondarko2}) that $\gamma^{\perp_{>0}}$ is a TTF class. Let $\Phi\colon \sigma\longrightarrow B$ denote a $\gamma^{\perp_{>0}}$-preenvelope of $\sigma$. In particular, $H^0_\sigma(B)$ lies in $\Tcal$. We claim that $\phi:=H^0_\sigma(\Phi)$ is a $\Tcal$-preenvelope of the projective generator $H^0_\sigma(\sigma)$. To that end, suppose that $f\colon H^0_\sigma(\sigma)\longrightarrow T$ is a morphism in $\Hcal_\sigma$ with $T$ in $\Tcal$. Then the composition $f\circ \pi: \sigma\longrightarrow T$, where $\pi\colon \sigma\longrightarrow H^0_\sigma(\sigma)$ is the natural truncation map, factors through $\Phi$ (since $T$ lies in $\gamma^{\perp_{>0}}$). In other words, there is $\alpha\colon B\longrightarrow T$ such that $f\circ \pi=\alpha\circ \Phi$. If we apply $H^0_\sigma$ to this equality we get $f=H^0_\sigma(\alpha)\circ \phi$, as wanted.

Finally, if $\Tcal=\Gen{H^0_\sigma(\Pscr)}$ as in Proposition~\ref{prop:general existence of silting mutations}(1), there is an epimorphism $p\colon H^0_\sigma(P)\longrightarrow H^0_\sigma(B)$ in $\Hcal_\sigma$ with $P$ in $\Pscr$ and the preenvelope $\phi\colon H^0_\sigma(\sigma)\longrightarrow H^0_\sigma(B)$ factors through $p$ as $H^0_\sigma(\sigma)$ is projective in $\Hcal_\sigma$. Clearly, the resulting map $\phi'\colon H^0_\sigma(\sigma)\longrightarrow H^0_\sigma(P)$ is a $\Tcal$-preenvelope as well and there is a map $\Phi'\colon\sigma\longrightarrow P$ such that $\phi'=H^0_\sigma(\Phi')$ by Proposition~\ref{prop:summary cosilting}(1). Finally, $\Phi'$ is a $\Pscr$-preenvelope by arguments dual to those in Lemma~\ref{lem:approx prop}.
\end{remark}

Let us now delve into the connection between silting mutation and cosilting mutation. For this purpose we restrict ourselves to the setting of compactly generated triangulated categories. Recall that, in a compactly generated triangulated category $\Dcal$, for any compact object $K$ there is an object $\mathbb{BC}(K)$, called the \textbf{Brown-Comenetz dual of $K$}, such that
$$\Hom_\mathbb{Z}(\Hom_\Dcal(K,-),\mathbb{Q}/\mathbb{Z})\cong \Hom_\Dcal(-,\mathbb{BC}(K)).$$
Note that, since $K$ is compact, $\Hom_\mathbb{Z}(\Hom_\Dcal(K,-),\mathbb{Q}/\mathbb{Z})$ sends pure triangles to short exact sequences and, thus, $\mathbb{BC}(K)$ is pure-injective by \cite[Corollary~2.5]{Krause-TC}. The following lemma is an easy observation.

\begin{lemma}\label{siltcosilt}
Let $\Dcal$ be a compactly generated triangulated category and $\sigma$ a compact silting object. Then $\mathbb{BC}(\sigma)$ is a pure-injective cosilting object such that $\mathbb{T}_\sigma=\mathbb{T}_{\mathbb{BC}(\sigma)}$.
\end{lemma}

Recall that in \cite[Section 4]{AI} the authors consider compact silting objects in compactly generated triangulated categories. To see that, for a compact object, the definition of silting given in \cite[Definition 4.1]{AI} is equivalent to the definition of silting given is Subsection \ref{sec: silting and cosilting}, we refer the reader to \cite[Corollary 4.7]{AI}.  The following theorem shows that the operation of mutation of compact silting objects defined in \cite{AI} is a special case of the operation of mutation of both silting objects and pure-injective cosilting objects.

\begin{theorem}\label{thm:compact silting mutation}
Let $\Dcal$ be a compactly generated triangulated category. Let $\sigma$ be a compact silting object in $\Dcal$ and $\mathfrak{p}=\add{\mathfrak{p}}$ a subcategory of $\add{\sigma}$ and define $\Pscr = \Add{\mathfrak{p}}$.
\begin{enumerate}
\item Any $\mathfrak{p}$-preenvelope of $\sigma$ is a $\Pscr$-preenvelope and any $\mathfrak{p}$-precover of $\sigma$ is a $\Pscr$-precover.
\item If $\sigma$ admits a $\mathfrak{p}$-preenvelope (respectively, a $\mathfrak{p}$-precover), then it has a compact left (respectively, right) mutation ${\sigma'}$ with respect to $\Pscr$. Moreover, the t-structure $\mathbb{T}_{\sigma'}$ is the cosilting t-structure associated to a pure-injective left  (respectively, right) mutation of the cosilting object $\mathbb{BC}(\sigma)$.
\end{enumerate}
\end{theorem}
\begin{proof} (1): Let $\Phi\colon \sigma\longrightarrow \varepsilon_0$ be a $\mathfrak{p}$-preenvelope of $\sigma$ and let $f \colon \sigma \longrightarrow P$ be a map to an object $P$ in $\Pscr$.  Without loss of generality we may assume that $P$ is a coproduct of objects in $\mathfrak{p}$.  Since $\sigma$ is compact, $f$ factors through a finite subsum of objects in $\mathfrak{p}$ and hence through $\Phi$.  Dually, let $\Psi \colon \gamma_0 \longrightarrow \sigma$ be a $\mathfrak{p}$-precover and $g \colon P \longrightarrow \sigma$ a map from an object $P$ in $\Pscr$.  Again, without loss of generality assume that $P$ is a coproduct of objects in $\mathfrak{p}$.  Since a factorisation through $\Psi$ exists for each summand of $P$, the universal property of the coproduct yields a factorisation for $g$.

(2): If $\sigma$ admits a $\mathfrak{p}$-preenvelope $\Phi\colon \sigma\longrightarrow \varepsilon_0$, the triangle
$$\xymatrix{\sigma\ar[r]^\Phi& \varepsilon_0\ar[r]& \varepsilon_1\ar[r]& \sigma[1]}$$
yields a compact silting object $\sigma'=\varepsilon_0\oplus\varepsilon_1$ by \cite[Theorem 2.31]{AI}.  By (1) this is the left mutation of $\sigma$ with respect to $\Pscr$ in the sense of Definition~\ref{def:siltmut}.

By (the proof of) Theorem~\ref{thm:siltmutation as HRS}, the t-structure $\mathbb{T}_{\sigma'}$ is the left HRS-tilt of $\mathbb{T}_\sigma$ at the torsion pair in $\Hcal_\sigma$ with
 torsion-free class $\Fcal=H^0_\sigma(\Pscr)^{\perp_0}=H^0_\sigma(\varepsilon_0)^{\perp_0}$ in $\Hcal_\sigma$. By Proposition~\ref{prop:summary cosilting}, $\Fcal=\varepsilon_0^{\perp_0}\cap\Hcal_\sigma$ in $\Dcal$.
Consider now the Brown-Comenetz dual of $\varepsilon_0$ and define $\Escr:=\Prod{\mathbb{BC}(\varepsilon_0)}$. Note that 
$\Fcal={}^{\perp_0} {\mathbb{BC}(\varepsilon_0)}\cap\Hcal_\sigma$ in $\Dcal$ since $\Hom_\Dcal(\varepsilon_0,X)=0$ if and only if $\Hom_\Dcal(X,\mathbb{BC}(\varepsilon_0))=0$ and, thanks to Proposition~\ref{prop:summary cosilting}, we infer that $\Fcal={}^{\perp_0} H^0_\sigma(\Escr)$ in $\Hcal_\sigma$. Recall from Lemma~\ref{siltcosilt} that $\mathbb{T}_\sigma=\mathbb{T}_{\mathbb{BC}(\sigma)}$ and $\mathbb{T}_{\sigma'}=\mathbb{T}_{\mathbb{BC}(\sigma')}$. By Theorem~\ref{thm:mutation as HRS-tilt}(1)
the t-structure $\mathbb{T}_{\sigma'}$ then coincides with the one associated to a  left mutation of $\mathbb{BC}(\sigma)$ with respect to $\Escr$.

Now we turn to the dual case.  If $\gamma_0\longrightarrow \sigma$ is a $\mathfrak{p}$-precover, then we see with analogous arguments that there is a compact silting object $\sigma'$ such that
$\mathbb{T}_{\sigma'}$ is the right HRS-tilt of $\mathbb{T}_\sigma$ at the torsion pair in $\Hcal_\sigma$ with
 torsion class  $\Tcal=H^0_\sigma(\gamma_0)^{\perp_0}={}^{\perp_0}\mathbb{BC}(\gamma_0)$, and Lemma~\ref{siltcosilt} yields again that $\mathbb{T}_\sigma=\mathbb{T}_{\mathbb{BC}(\sigma)}$ and $\mathbb{T}_{\sigma'}=\mathbb{T}_{\mathbb{BC}(\sigma')}$. Note that $\Tcal={}^{\perp_0} H^0_\sigma(\Escr)$ where $\Escr=\Prod{\mathbb{BC}(\gamma_0)}\subseteq\Prod{\mathbb{BC}(\sigma)}$.
Hence $\mathbb{T}_{\sigma'}$ is the right HRS-tilt of $\mathbb{T}_\sigma$ at the cosilting torsion pair cogenerated by $H^0_\sigma(\Escr)$.  Moreover, $\mathbb{BC}(\sigma)$  has an $\Escr$-cover by Lemma~\ref{rem:approximate every injective}. So we infer from Theorem~\ref{thm:mutation as HRS-tilt}(2) that the t-structure $\mathbb{T}_{\sigma'}$ coincides with the one associated to a  right mutation of $\mathbb{BC}(\sigma)$ with respect to $\Escr$.
\end{proof}

\section{Mutation and localisation}\label{sec:local}

In this section we will show that, in nice enough contexts, mutation can be understood as three step process: first restrict the t-structures (to certain subcategories); then shift one of the restricted t-structures; finally glue them back together. In order to prove this, we need to review some ideas concerning restricting and gluing.  

\subsection{Restricting and gluing along (co)localising sequences}
A sequence of exact functors between triangulated (respectively, abelian categories)
$$\xymatrix{\Bcal\ar[r]^{F}&\Dcal\ar[r]^{G}&\Ccal}$$
is said to be a \textbf{short exact sequence} if $F$ is fully faithful, {the Verdier quotient (respectively, the Serre quotient)} $\Dcal/\Bcal$ is well-defined, there are an equivalence $L\colon \Dcal/\Bcal \longrightarrow \Ccal$ and a natural {isomorphism} $\theta\colon G\longrightarrow L\circ q$ where $q\colon \Dcal\longrightarrow \Dcal/\Bcal$.

A short exact sequence  of triangulated (respectively, abelian) categories as above is said to be a \textbf{localising sequence}
if both $F$ and $G$ admit right adjoints. Dually, it is said to be a \textbf{colocalising sequence}
if both $F$ and $G$ admit left adjoints. A short exact sequence is said to be a \textbf{recollement} if it is both a localising and a colocalising sequence. Note that a localising sequence of triangulated categories
$$\xymatrix{\Bcal\ar[r]^{i_*}&\Dcal\ar[r]^{j^*}\ar@/^1.5pc/[l]_{i^!}&\Ccal\ar@/^1.5pc/[l]_{j_*}}$$
can be transformed into a colocalising sequence
$$\xymatrix{\Ccal\ar[r]^{j_*}&\Dcal\ar[r]^{i^!}\ar@/_1.5pc/[l]_{j^*}&\Bcal\ar@/_1.5pc/[l]_{i_*}}.$$
However, the same observation does not hold for abelian categories. This is due to the fact that in the triangulated setting adjoints of exact functors are exact, while this is not the case in the abelian setting. Recall that, for abelian categories, quotient functors by \textbf{Serre subcategories} (i.e. subcategories closed under subobjects, quotient objects and extensions) are always exact.

\begin{definition}
Given a short exact sequence of triangulated categories
$\xymatrix{\Bcal\ar[r]^{F}&\Dcal\ar[r]^{G}&\Ccal}$,
we say that a t-structure $\mathbb{T}=(\Xcal,\Ycal)$ in $\Dcal$ \textbf{restricts along the exact sequence} if $(\Xcal\cap \essim{F},\Ycal\cap \essim{F})$ is a t-structure in $\essim{F}$ and $(G(\Xcal),G(\Ycal))$ is a t-structure in $\Ccal$.
\end{definition}

\begin{proposition}\label{restricts}
Let $\mathbb{T}$ be a t-structure in a triangulated category $\Dcal$ and suppose there is a short exact sequence of triangulated categories
$$\xymatrix{\Bcal\ar[r]^{F}&\Dcal\ar[r]^{G}&\Ccal}$$
\begin{enumerate}
\item\cite[Lemma 3.3]{CR} The following statements are equivalent for a t-structure $\mathbb{T}$.
\begin{enumerate}
\item $\mathbb{T}$ restricts along the exact sequence.
\item ${G(\mathbb{T})}:=(G(\Xcal),G(\Ycal))$ is a t-structure in $\Ccal$.
\item ${\mathbb{T}\cap \essim{F}}:=(\Xcal\cap \essim{F},\Ycal\cap \essim{F})$ is a t-structure in $\essim{F}$, and the heart of this t-structure is a Serre subcategory of the heart of $\mathbb{T}$.
\end{enumerate}
\item If $\mathbb{T}$ restricts along the exact sequence, then there is an induced short exact sequence of abelian categories formed by the associated hearts. If the exact sequence of triangulated categories is a localising, respectively colocalising, sequence, then so is the corresponding sequence of hearts.
\end{enumerate}
\end{proposition}
\begin{proof}
For simplicity, since $F$ is fully faithful, we identify $\Bcal$ with $\essim{F}$ and assume, without loss of generality, that $F$ is the inclusion functor. Denote by $\Hcal$ the heart of $\mathbb{T}$ and by $\Hcal_\Bcal$ and $\Hcal_\Ccal$ the hearts of $G(\mathbb{T})$ and $\mathbb{T}\cap \essim{F}$ respectively. Define $\overline{F}\colon \Hcal_\Bcal\longrightarrow \Hcal$ to be the restriction of $F$ to $\Hcal_\Bcal$ and $\overline{G}\colon\Hcal\longrightarrow \Hcal_\Ccal$ to be the restriction of $G$ to $\Hcal$. Note that these functors are well-defined by the construction of the t-structures $G(\mathbb{T})$ and $\mathbb{T}\cap \essim{F}$. {Under the equivalent conditions of (1), it follows as in \cite[Lemma 3.9]{CR}  (see also \cite[Section 1.4]{BBD}, \cite[Proposition 2.5]{BR})  that} there is a short exact sequence of abelian categories
$$\xymatrix{\Hcal_\Bcal\ar[r]^{\overline{F}}&\Hcal\ar[r]^{\overline{G}}&\Hcal_\Ccal}.$$

{Next} we observe that if $F$ has a right adjoint, then so does $\overline{F}$. Indeed, if $R\colon \Dcal\longrightarrow \Bcal$ is a right adjoint to $F$, $Y$ is an object of {$\Bcal$, and $D$ is an object of $\Dcal$, then we have a canonical isomorphism
$$
\Hom_\Dcal(F(Y),D) \cong \Hom_\Bcal(Y,R(D)).
$$
In particular, $\Hom_\Bcal(Y,R(D))=0$ for each $Y\in\Xcal\cap \essim{F}$ whenever $D\in\Ycal$. It immediately follows that $R(\Ycal)\subseteq\Ycal\cap \essim{F}$ and that we have canonical isomorphisms
$$
\Hom_\Hcal(\overline{F}(Y),D) = \Hom_\Dcal(F(Y),D) \cong \Hom_\Bcal(Y,R(D)) \cong \Hom_{\Hcal_\Bcal}(Y,H^0_{\mathbb{T}_\Bcal}(R(D)))
$$
for any $Y\in\Hcal_\Bcal$ and $D\in\Hcal$.}
Therefore the following composition is a right adjoint to $\overline{F}$:
$$
\xymatrix{\Hcal\ar[r]^{\inc}&\Dcal\ar[r]^R&\Bcal\ar[r]^{H^0_{\mathbb{T}_\Bcal}}&\Hcal_\Bcal.}
$$
{An analogous argument shows that if $S\colon\Ccal\to\Dcal$ is a right adjoint to $G$, then the composition
$$
\xymatrix{\Hcal_\Ccal\ar[r]^{\inc}&\Ccal\ar[r]^S&\Bcal\ar[r]^{H^0_{\mathbb{T}}}&\Hcal}
$$
is a right adjoint to $\overline{G}\colon\Hcal\to\Hcal_\Ccal$. Indeed, if $X\in G(\Ycal)$, then $\Hom_\Dcal(D,S(X))\cong\Hom_\Ccal(G(D),X)=0$ for any $D\in\Xcal$. Thus, $S(G(\Ycal))\subseteq\Ycal$ and, for any $D$ in $\Hcal$ and $X$ in $\Hcal_\Ccal$, we have canonical isomorphisms
$$
\Hom_{\Hcal_\Ccal}(\overline{G}(D),X) = \Hom_\Ccal(G(D),X) \cong \Hom_\Dcal(D,S(X)) \cong \Hom_\Hcal(D,H^0_{\mathbb{T}}(S(X))).
$$

Finally, the assertion for left adjoints follows in similar fashion.}
\end{proof}

\begin{remark}\label{rem:glue}
Let $\mathbb{T}=(\Xcal,\Ycal)$ be a t-structure in a triangulated category $\Dcal$ which restricts along a short exact sequence of triangulated categories
$$\xymatrix{\Bcal\ar[r]^{i_*}&\Dcal\ar[r]^{j^*}&\Ccal}.$$
Then it follows from \cite{BBD} that:
\begin{enumerate} 
\item if $\mathbb{T}$ is a localising sequence, then $\Ycal=(\Ycal\cap\essim{i_*})\star j_*j^*(\Ycal)$, where $j_*$ is the right adjoint of $j^*$;
\item if $\mathbb{T}$ is a colocalising sequence, then $\Xcal=j_!j^*\Xcal\star (\Xcal\cap\essim{i_*})$, where $j_!$ is the left adjoint of $j^*$.
\end{enumerate}
\end{remark}

\subsection{Mutation and localising sequences}
{We are now going to explore the relation between mutation of pure-injective cosilting (respectively, pure-projective silting) objects and categorical localisations. First of all, given a cosilting object $\sigma$, we show that sets of pure-injective objects in $\Prod\sigma$  induce localising sequences both at triangulated and abelian level.}

\begin{proposition}\label{prop:restrict cosilting}
Let $\sigma$ be a cosilting object in a compactly generated triangulated category $\Dcal$. Let {$\Escr=\Prod\Escr$} be a  subcategory of $\Prod{\sigma}$ in which every object is pure-injective. Then there is a torsion pair $({}^{\perp_\mathbb{Z}}\Escr,\Ccal_\Escr)$ in $\Dcal$ and, thus, a localis{ing} sequence
$$\xymatrix{{}^{\perp_\mathbb{Z}}\Escr\ar[r]^{i_*}&\Dcal\ar[r]^{j^*}\ar@/^1.5pc/[l]_{i^!}&\Ccal_\Escr\ar@/^1.5pc/[l]_{j_*}}$$
where $i_*$ and $j_*$ are the inclusion functors. Moreover, the cosilting t-structure $\mathbb{T}_\sigma$ restricts along the sequence, thus giving rise to a localising sequence of abelian categories
$$\xymatrix{\Scal\ar[r]^{\overline{i_*}}&\Hcal_\sigma\ar[r]^{\overline{j^*}}\ar@/^1.5pc/[l]_{\overline{i^!}}&\Scal^{\perp_{0,1}}\ar@/^1.5pc/[l]_{\overline{j_*}}}$$
where $\overline{i_*}$ and $\overline{j_*}$ are the inclusion functors{, $\Scal={}^{\perp_0} H^0_\sigma(\Escr)$ is a hereditary torsion class}, and $\Scal^{\perp_{0,1}}$ is the associated Giraud subcategory in $\Hcal_\sigma$.
\end{proposition}
\begin{proof}
First, we observe that since $\Escr$ is made of pure-injective objects, {\cite[Corollary 6.9]{SS20}} shows that there is a torsion pair of the form $({}^{\perp_{\mathbb{Z}}}\Escr,\Ccal_\Escr)$ in $\Dcal$. This is equivalent to the existence of the claimed localisation sequence. We now show that $\mathbb{T}_\sigma=(\Xcal,\Ycal)$ restricts along this exact sequence using Proposition \ref{restricts}. Indeed, if $D$ lies in ${}^{\perp_{\mathbb{Z}}}\Escr$, consider its truncation triangle for $\mathbb{T}_\sigma$
$$\xymatrix{x(D)\ar[r]&D\ar[r]&y(D)\ar[r]&x(X)[1]}.$$
Clearly ${}^{\perp_{\mathbb{Z}}}\Escr={}^{\perp_{>0}}\Escr\cap {}^{\perp_{\leq 0}}\Escr$ and, since $\Escr$ is contained in $\Prod{\sigma}$, we have that $\Xcal\subseteq {}^{\perp_{\leq 0}}\Escr$ and $\Ycal\subseteq {}^{\perp_{>0}}\Escr$ and, thus, $x(D)$ lies in ${}^{\perp_{\leq 0}}\Escr$ and $y(D)$ lies in ${}^{\perp_{> 0}}\Escr$. Moreover, since ${}^{\perp_{\leq 0}}\Escr$ is suspended, $y(D)$ also lies in ${}^{\perp_{\leq 0}}\Escr$, and since ${}^{\perp_{> 0}}\Escr$ is cosuspended,  $x(D)$ lies in ${}^{\perp_{> 0}}\Escr$. This shows that the truncation triangle restricts to ${}^{\perp_{\mathbb{Z}}}\Escr$. Furthermore, note that $\Hcal_\sigma\cap {}^{\perp_{\mathbb{Z}}}\Escr$ is the subcategory of objects $X$ in $\Hcal_\sigma$ such that $\Hom_\Hcal(X,H^0_\sigma(E))=0$ for all $E$ in $\Escr$. {Recall that  $H^0_\sigma(E)$ is injective for all $E$ in $\Escr$ and, therefore, $\Hcal_\sigma\cap {}^{\perp_{\mathbb{Z}}}\Escr={}^{\perp_0} H^0_\sigma(\Escr)=\Scal$ is a hereditary torsion class in $\Hcal_\sigma$}. Our claim then follows from Proposition \ref{restricts} and from the fact that the right adjoint of the localising functor in a short exact sequence of abelian categories identifies the quotient category with the Giraud subcategory $\Scal^{\perp_{0,1}}$ associated to $\Scal$.
\end{proof}

Next, as promised earlier, we prove that if $\sigma$ and $\sigma'$ are pure-injective cosilting objects which are mutation of each other, the associated t-structures can be described in terms of some operations along the localising sequence induced by $\Escr=\Prod\sigma\cap\Prod{\sigma'}$.


\begin{lemma}\label{mutation restricts}
Let $\sigma$ and ${\sigma'}$ be pure-injective cosilting objects in a compactly generated triangulated category $\Dcal$. Let $\mathbb{T}_\sigma=(\Xcal,\Ycal)$ and $\mathbb{T}_{\sigma'}=(\Xcal',\Ycal')$ be the associated cosilting t-structures and let
$$\xymatrix{{}^{\perp_\mathbb{Z}}\Escr\ar[r]^{i_*}&\Dcal\ar[r]^{j^*}&\Ccal_\Escr}$$
be the localising sequence induced by $\Escr=\Prod{\sigma}\cap\Prod{{\sigma'}}$.
\begin{enumerate}
\item If {${\sigma'}$ is a right mutation of $\sigma$}, then we have the following equalities {for the restricted} t-structures:
\begin{enumerate}
\item $\mathbb{T}_{\sigma'}\cap{}^{\perp_{\mathbb{Z}}}\Escr=(\mathbb{T}_\sigma\cap{}^{\perp_{\mathbb{Z}}}\Escr)[-1]$ in ${}^{\perp_{\mathbb{Z}}}\Escr$; and
\item $j^*(\mathbb{T}_{\sigma'})=j^*(\mathbb{T}_{\sigma})$ in $\Ccal_\Escr$.
\end{enumerate}
\item If {${\sigma'}$ is a left mutation of $\sigma$}, then we have the following equalities {for the restricted} t-structures:
\begin{enumerate}
\item $\mathbb{T}_{\sigma'}\cap{}^{\perp_{\mathbb{Z}}}\Escr=(\mathbb{T}_\sigma\cap{}^{\perp_{\mathbb{Z}}}\Escr)[1]$ in ${}^{\perp_{\mathbb{Z}}}\Escr$; and
\item $j^*(\mathbb{T}_{\sigma'})=j^*(\mathbb{T}_{\sigma})$ in $\Ccal_\Escr$.
\end{enumerate}
\end{enumerate}
\end{lemma}
\begin{proof}
We prove (1). The assertion (2) follows analogously, taking into account Corollary~\ref{cor:inverse mutations}(3).

By assumption we know that $\mathbb{T}_{\sigma'}$ is the right HRS-tilt of $\mathbb{T}_{\sigma}$ at the torsion pair $(\Scal, \Rcal)$ in $\Hcal_\sigma$.
Recall   from Proposition~\ref{prop:restrict cosilting} that the t-structures $\mathbb{T}_\sigma$ and $\mathbb{T}_{\sigma'}$ restrict along the localising sequence determined by $\Escr$, and
the heart of $\mathbb{T}_\sigma\cap{}^{\perp_{\mathbb{Z}}}\Escr$ is $\Hcal_{\sigma}\cap{}^{\perp_{\mathbb{Z}}}\mathscr{E}={}^{\perp_0} H^0_\sigma(\Escr)=\Scal$. We show that the heart of $\mathbb{T}_{\sigma'}\cap{}^{\perp_{\mathbb{Z}}}\Escr$
coincides with $\Scal[-1]$. To this end, we consider the torsion  pair
$(\Rcal,\Scal[-1])$ in $\Hcal_{\sigma'}$.
It is clear that
$\Scal[-1]\subseteq \Hcal_{\sigma'}\cap{}^{\perp_{\mathbb{Z}}}\mathscr{E}$.
For the converse, we pick an object $X$ in $\Hcal_{\sigma'}\cap{}^{\perp_{\mathbb{Z}}}\mathscr{E}$ with torsion decomposition  
$$\xymatrix{0\ar[r]&r(X)\ar[r]&X\ar[r]&X/r(X)\ar[r]&0}$$
where $r(X)$ is in $\Rcal$,  and $X/r(X)$ is in $\Scal[-1]$ and thus in ${}^{\perp_{\mathbb{Z}}}\mathscr{E}\cap\Hcal_{\sigma'}$. Since this exact sequence yields a triangle in $\Dcal$ and since ${}^{\perp_{\mathbb{Z}}}\mathscr{E}$ is triangulated, it follows that $r(X)$ lies in ${}^{\perp_{\mathbb{Z}}}\mathscr{E}$. But then $r(X)$ lies in $\Rcal$ and in  $\Hcal_\sigma\cap{}^{\perp_{\mathbb{Z}}}\Escr=\Scal$. We conclude that  $r(X)=0$, as wanted.

Now, combining the fact that the heart of $\mathbb{T}_{\sigma'}\cap{}^{\perp_{\mathbb{Z}}}\Escr$
lies in $\Ycal[-1]$ with the inclusions $\Ycal[-1]\subseteq \Ycal'\subseteq \Ycal$, we easily obtain the equality in (1)(a).
Next, we check the equality in (1)(b). Since the right adjoint $j_*$ of the quotient functor $j^*$ is fully faithful, this amounts to verifying $j_*j^*(\Ycal)=j_*j^*(\Ycal')$. Observe that $\Ycal'\subseteq \Ycal$ implies $j_*j^*(\Ycal)\supseteq j_*j^*(\Ycal')$. For the reverse inclusion, we pick  an object $X$ in $j_*j^*(\Ycal)$ and consider a truncation triangle with respect to the t-structure $\mathbb{T}_{\sigma'}$,
$$\xymatrix{A\ar[r]&X\ar[r]&B\ar[r]&A[1]}$$
with $A$ in $\Xcal'$ and $B$ in $\Ycal'$. Again, since $j_*j^*(\Ycal')\subseteq \Ycal'\subseteq  \Ycal$ and since the latter is cosuspended, we have that $A$ lies in $\Xcal'\cap \Ycal$, which coincides with the torsion class $\Scal={}^{\perp_0}H^0_\sigma(\Escr)$ of $\Hcal_\sigma$. But this means that $A$ lies in ${}^{\perp_{\mathbb{Z}}}\Escr$ and, in particular, it has no maps to objects in the essential image of $j_*$, where $X$ lies. Hence $X$ is isomorphic to $B$, proving the desired equality.
\end{proof}

Now we can give a precise description of mutated cosilting t-structures associated to pure-injective cosilting objects in terms of localisation.

\begin{theorem}\label{existence right mutation2}
Let $\sigma$ be a pure-injective cosilting object in a compactly generated triangulated category $\Dcal$ with associated t-structure $\mathbb{T}_\sigma=(\Xcal, \Ycal)$. Let $\Escr=\Prod\Escr$ be  a  subcategory  of $\Prod{\sigma}$, and consider the localising sequence induced by $\Escr$
$$\xymatrix{{}^{\perp_\mathbb{Z}}\Escr\ar[r]^{i_*}&\Dcal\ar[r]^{j^*}&\Ccal_\Escr}.$$
\begin{enumerate}
\item $\sigma$ admits a right mutation  with respect to $\Escr$ if and only if there is a cosilting object $\sigma'$ with associated coaisle $\Ycal'=(\Ycal\cap {}^{\perp_\mathbb{Z}}\Escr)[-1]\star j_*j^*(\Ycal)$.
\item $\sigma$ admits a left mutation with respect to $\Escr$ if and only if
 there is a cosilting object $\sigma'$ with associated coaisle $\Ycal'=(\Ycal\cap {}^{\perp_\mathbb{Z}}\Escr)[1]\star j_*j^*(\Ycal)$.
\end{enumerate}
\end{theorem}
\begin{proof}
Recall that a mutation of a pure-injective cosilting object is again pure-injective by Theorem~\ref{existence right mutation}.
The only-if part in (1) and (2) follows directly from Lemma \ref{mutation restricts} and Remark~\ref{rem:glue}. For the if-part, let us consider the hereditary torsion pair $(\Scal,\Rcal)$ in $\Hcal_\sigma$ cogenerated by $H^0_\sigma(\Escr)$ and verify the conditions (1)(b) and (2)(b) in Theorem~\ref{existence right mutation}, respectively.
In fact, we show in both cases that $\sigma'$ is the corresponding mutation.

(1): We have to show that
the torsion-free class $\Rcal$ is closed under direct limits. By  Proposition~\ref{prop:cover vs finite type} this amounts to proving  that $(\Scal,\Rcal)$ is a cosilting torsion pair. Indeed, we claim that  $\mathbb{T}_{\sigma'}$ is a right HRS-tilt of $\mathbb{T}_\sigma$ at $(\Scal,\Rcal)$. For a proof, we apply Proposition~\ref{prop:intermediate} and  verify that $\Ycal[-1]\subseteq\Ycal'\subseteq\Ycal$ and $\Rcal=\Hcal_\sigma\cap\Hcal_{\sigma'}$.

Observe first that $\Ycal'\subseteq\Ycal$ since $j_*j^*(\Ycal)\subseteq\Ycal$ by Proposition~\ref{prop:restrict cosilting} and Remark~\ref{rem:glue}. On the other hand, let $Y$ be an object in $\Ycal[-1]$ and
consider the triangle associated to $Y$ given by the localising sequence induced by $\Escr$,
\begin{equation}\tag{$\Delta_4$}\label{eq:loc-Escr}
\xymatrix{A\ar[r]& Y\ar[r]^{\alpha\ \ }& j_*j^*Y\ar[r]& A[1]}.
\end{equation}
As $j_*j^*Y\subseteq\Ycal[-1]\subseteq\Ycal$ and, consequently, $A\in(\Ycal\cap {}^{\perp_\mathbb{Z}}\Escr)[-1]$, we have $\Ycal[-1]\subseteq\Ycal'$. 

We now finish the proof of part (1) by showing that $\Hcal_\sigma\cap \Ycal'=\Hcal_\sigma\cap\Hcal_{\sigma'}=\Rcal$.
To that end, let $Y$ be an object in $\Hcal_\sigma$ and consider the triangle \eqref{eq:loc-Escr} associated to $Y$ given by the localising sequence induced by $\Escr$. By assumption, $Y$ lies in $\Ycal'$ if and only if $A$ lies in $(\Ycal\cap {}^{\perp_\mathbb{Z}}\Escr)[-1]$. Now, applying the functor $H^0_\sigma$ to the triangle, and using that $j_*j^*(\Ycal)\subseteq\Ycal$, we see that $A$  lies in $\Ycal$, and we get an exact sequence
$$\xymatrix{0\ar[r]&H^0_\sigma(A)\ar[r]&Y\ar[r]^{H^0_\sigma(\alpha)\ \ \ \ \ \ }&H^0_\sigma(j_*j^*Y)}.$$
Observe that $A$ lies in $(\Ycal\cap {}^{\perp_\mathbb{Z}}\Escr)[-1]$ if and only if $H^0_\sigma(A)=0$, which means that $H^0_\sigma(\alpha)$ is a monomorphism. But $H^0_\sigma(\alpha)$ is the reflection of $Y$ in the Giraud subcategory $\Scal^{\perp_{0,1}}$ (see Proposition~\ref{prop:restrict cosilting}) and, hence, it is a monomorphism if and only if $Y$ lies in the torsion-free class $\Rcal$.

(2): We have to show that the torsion class  $\Scal$  is closed under direct products.
 For this purpose, it suffices to show that $\Scal=H^{-1}_\sigma(\Ycal')$. Indeed, $\Ycal'$ is contained in $\Ycal[1]$, and by definition of products in the heart $\Hcal_\sigma$, the functor $H^{-1}_\sigma$ sends products in $\Ycal[1]$ to products in $\Hcal_\sigma$. Hence $H^{-1}_\sigma(\Ycal')$ is closed under products, because so is $\Ycal'$ being the coaisle of a  t-structure.
 
Now, as $j_*j^*(\Ycal)$ is contained in $\Ycal$, we have by assumption
$$H^{-1}_\sigma(\Ycal')=H^{-1}_\sigma((\Ycal\cap {}^{\perp_\mathbb{Z}}\Escr)[1]\star j_*j^*(\Ycal))=H^{-1}_\sigma((\Ycal\cap {}^{\perp_\mathbb{Z}}\Escr)[1])=H^0_\sigma(\Ycal\cap{}^{\perp_\mathbb{Z}}\Escr)$$
which clearly consists of the objects $X$ in $\Hcal_\sigma$ such that $\Hom_\Dcal(X,\Escr)=0$. Since $\Escr$ is contained in $\Prod{\sigma}$, the latter are precisely the objects of $\Scal={}^{\perp_0}H^0_\sigma(\Escr)$, as wanted.
\end{proof}

\begin{example}\label{expl:krone2}
We continue Example~\ref{expl:krone} over the Kronecker algebra $A$.
Let  $P\subset\mathbb{X}$, and consider the right mutation $\sigma_P$ of $\sigma_{\emptyset}$
 at the set
 $\Escr=\Prod{\{S[\infty]\,\mid\, S\in \tube_{\bar{P}}\}}.$ Notice that the  heart $\Hcal$ associated with $\sigma_{\emptyset}$ is equivalent to the category of quasi-coherent sheaves over the projective line $\mathbb{X}$, where the simple sheaves are in bijection with  the simple regular $A$-modules.  Moreover,
 the hereditary torsion pair  $(\Scal,\Rcal)$ in $\Hcal$ cogenerated by $\Escr$ corresponds to the torsion pair generated by the simple sheaves which are determined by the family of tubes $\tube_P=\bigcup_{x\in P} \tube_x$. Combining Proposition~\ref{prop:restrict cosilting} with \cite[Corollary 5.8]{ak1} we obtain a localising sequence  of hearts
 $$\xymatrix{\Scal=\varinjlim_\Hcal\tube_P\ar[r]^{}&\Hcal\ar[r]^{}&\Mod{A_P}}$$
 where  $A_P$ is a hereditary ring obtained as the universal localisation of $A$ at the modules from $\tube_P$.  
On the other hand, the heart $\Hcal'$ associated with the mutation $\sigma_{P}$ is a locally coherent Grothendieck category which is  neither hereditary nor locally noetherian if $P\ne\varnothing$.
\end{example}

Finally, we turn to the dual case. We first need a technical lemma.

\begin{lemma} \label{lem:definability from pure-proj}
Let $\Dcal$ be a compactly generated triangulated category and $P\in\Dcal$ a pure-projective object. Then there exists a set of maps $\Ical$ between compact objects such that
\[ P^{\perp_0} = \{ X\in\Dcal \mid \Hom_\Dcal(f,X)=0 \textrm{ for each } f\in\Ical \} \]
\end{lemma}

\begin{proof}
Recall from~\S\ref{sec: silting and cosilting} that the category of pure-projective objects is equivalent to that of projective $\Dcal^c$-modules via $\mathbf{y}\colon \Dcal\longrightarrow \Mod{\Dcal^c}$. By a Theorem of Kaplansky, every projective module is a direct sum of countably generated ones~\cite[Lemma 36.3]{Mitchell}, so we can without loss of generality assume that $\mathbf{y}P$ is countably generated.
In particular, there is a sequence of compact objects of $\Dcal$,
\[ \xymatrix{C_1\ar[r]^-{f_1}&C_2\ar[r]^-{f_2}&C_3\ar[r]^-{f_3}&\cdots} \]
such that $\mathbf{y}P=\varinjlim\mathbf{y}C_n$ and, given any $X\in\Dcal$, we have isomorphisms
\[
\Hom_\Dcal(P,X) \cong
\Hom_{\Mod{\Dcal^c}}(\mathbf{y}P,\mathbf{y}X) \cong
\varprojlim \Hom_{\Mod{\Dcal^c}}(\mathbf{y}C_n,\mathbf{y}X) \cong
\varprojlim \Hom_\Dcal(C_n,X).
\]
(the outer isomorphisms follow by the Yoneda lemma and the fact that $\mathbf{y}$ preserves coproducts and summands).
Moreover, as one can trace back to \cite[Theorem 1.9]{Wh} and as is explained in~\cite[\S1]{HerPr}, up to passing to a cofinal subsystem, one can assume that there exist morphisms
\[ \xymatrix{C_2\ar@{<-}[r]^-{g_2}&C_3\ar@{<-}[r]^-{g_3}&C_4\ar@{<-}[r]^-{g_4}&\cdots} \]
such that $g_{n+1}f_{n+1}f_n = f_n$ for each $n>0$.

We claim that, given any $X\in\Dcal$, we have $\Hom_\Dcal(P,X)=0$ if and only if $\Hom_\Dcal(f_n,X)=0$ for each $n>0$. The if part being clear, we focus on the only if-part. Applying $\Hom_\Dcal(-,X)$ to the direct system above, we obtain an inverse system of abelian groups
\[ \xymatrix{\Hom_\Dcal(C_1,X)\ar@{<-}[r]^-{f_1^*}&\Hom_\Dcal(C_2,X)\ar@{<-}[r]^-{f_2^*}&\Hom_\Dcal(C_3,X)\ar@{<-}[r]^-{f_3^*}&\cdots} \]
and one readily verifies that $\im{f_n^*}=\im{f_n^*f_{n+1}^*}$ for each $n>0
$, as $f_n^*(h)=hf_n=hg_{n+1}f_{n+1}f_n=f_n^*f_{n+1}^*(hg_{n+1})$ for any $h\colon C_{n+1}\to X$. Using the same argument as in~\cite[Lemma 4.5]{SaSt}, we deduce that also the image of the limit map
\[ \Hom_\Dcal(C_n,X) \longleftarrow \varprojlim \Hom_\Dcal(C_n,X) \cong \Hom_\Dcal(P,X) \]
equals $\im{f_n^*}$ for each $n>0$. Hence, if $\Hom_\Dcal(P,X)=0$, then $\im{f_n^*}=0$ for each $n>0$, or in other words $\Hom_\Dcal(f_n,X)=0$ for each $n>0$.
\end{proof}

Now we see that pure-projective silting objects even induce recollements of triangulated categories and of the associated hearts.

\begin{proposition}\label{restrict silting}
Let $\sigma$ be a silting object in an compactly generated triangulated category $\Dcal$. Let {$\Pscr=\Add\Pscr$} be a subcategory of $\Add{\sigma}$ in which every object is pure-projective. Then there is a TTF triple $(\Scal_\Pscr,{\Pscr}^{\perp_\mathbb{Z}},\Ccal_\Pscr)$ in $\Dcal$ and, thus, a recollement
$$\xymatrix{{\Pscr}^{\perp_\mathbb{Z}}\ar[r]^{i_*}&\Dcal\ar[r]^{j^*}\ar@/_1.5pc/[l]_{i^*}\ar@/^1.5pc/[l]_{i^!}&\Scal_\Pscr\ar@/_1.5pc/[l]_{j_!}\ar@/^1.5pc/[l]_{j_*}}$$
where $i_*$ and $j_!$ are the inclusion functors. Moreover, the silting t-structure $\mathbb{T}_\sigma$ restricts along the sequence, thus giving rise to a recollement of abelian categories
$$\xymatrix{\Tcal\ar[r]^{\overline{i_*}}&\Hcal_\sigma\ar[r]^{\overline{j^*}}\ar@/^1.5pc/[l]_{\overline{i^!}}\ar@/_1.5pc/[l]_{\overline{i^*}}&\Tcal^{\perp_{0,1}}\ar@/^1.5pc/[l]_{\overline{j_*}}\ar@/_1.5pc/[l]_{\overline{j_!}}}$$
where $\overline{i_*}$ and $\overline{j_*}$ are the inclusion functors{, $\Tcal=H^0_\sigma(\Pscr)^{\perp_0}$ is a TTF class}, and $\Tcal^{\perp_{0,1}}$ is the associated Giraud subcategory in $\Hcal_\sigma$.

If $\sigma$ and $\sigma'$ are pure-projective silting objects with associated t-structures  $\mathbb{T}_\sigma=(\Xcal,\Ycal)$ and $\mathbb{T}_{\sigma'}=(\Xcal',\Ycal')$,  and $\sigma'$ is a left mutation of $\sigma$ at $\Pscr$,  then we have $\Xcal'=j_!j^*(\Xcal)\star (\Xcal\cap \mathscr{P}^{\perp_{\mathbb{Z}}})[1]$, while in the case of right mutation we have $\Xcal'=j_!j^*(\Xcal)\star (\Xcal\cap \mathscr{P}^{\perp_{\mathbb{Z}}})[-1]$.
\end{proposition}
\begin{proof}
The class $\Pscr^{\perp_\mathbb{Z}}$ is preenveloping by Lemma~\ref{lem:definability from pure-proj} and~\cite[Proposition 3.11]{Krause-TC}, and the inclusion $i_*\colon{\Pscr}^{\perp_\mathbb{Z}}\longrightarrow\Dscr$ has a left adjoint $i^*$ by~\cite[Proposition 5.1]{Nee-adj}. In fact, the latter tells us that there is colocalising sequence
$$\xymatrix{{\Pscr}^{\perp_\mathbb{Z}}\ar[r]^{i_*}&\Dcal\ar[r]^{j^*}\ar@/_1.5pc/[l]_{i^*}&\Scal_\Pscr\ar@/_1.5pc/[l]_{j_!}}$$
It follows from that ${\Pscr}^{\perp_\mathbb{Z}}$ is itself a compactly generated triangulated category and that the sequence is also localising; see~\cite[Proposition 2.6 and Lemma~4.1]{Krause-TC} (see also \cite[Proposition 6.3]{LV}).
Hence, we get the recollement. The proof that the t-structure restricts along the recollement is analogous to Proposition~\ref{prop:restrict cosilting}. From \cite{BBD} it follows that there is such a recollement of hearts. Finally, the last statement is shown with arguments similar to those in the proof of Lemma~\ref{mutation restricts}.
\end{proof}

\section{Mutation of torsion pairs}
\label{sec:mutation and torsion pairs}

We now wish to consider mutations of cosilting torsion pairs.  We will make use of the  notation set up in Definition \ref{def: Cosilt}. We will also freely use the fact that a cosilting object $\sigma'$ is a right mutation of a cosilting object $\sigma$ if and only if $\sigma$ is a left mutation of $\sigma'$, cf.~Corollary~\ref{cor:inverse mutations}.

\begin{definition}\label{def: mutation of tp}
Let  $\Dcal$ be a triangulated category with products and coproducts and $\mathbb{T}$ a t-structure with heart $\Hcal$. Furthermore, let $\mathfrak{u}=(\Ucal, \Vcal)$ and $\mathfrak{t}=(\Tcal, \Fcal)$ be in $\Cosilt \Hcal$ with associated cosilting objects $\sigma_\mathfrak{u}$ and $\sigma_\mathfrak{t}$ in $\Dcal$. If $\sigma_\mathfrak{t}$ is a right mutation of $\sigma_\mathfrak{u}$, then we will say that $\mathfrak{t}$ is a \textbf{right mutation} of $\mathfrak{u}$ and $\mathfrak{u}$ is a \textbf{left mutation} of $\mathfrak{t}$.
\end{definition}

By Corollary~\ref{cor:inverse mutations}, the definition is independent of the cosilting objects we choose to represent $\mathfrak{u}$ and $\mathfrak{t}$.

\subsection{Inclusions of torsion pairs and filtration triples}
If a cosilting torsion pair $\mathfrak{t}=(\Tcal, \Fcal)$ is a right mutation of some cosilting torsion pair $\mathfrak{u}=(\Ucal, \Vcal)$, then, by Theorem~\ref{thm:mutation as HRS-tilt}, we have that $\mathbb{T}_{\mathfrak{t}^-}$ is a right HRS-tilt of $\mathbb{T}_{\mathfrak{u}^-}$.  By Proposition~\ref{prop:intermediate}, we have that $\Xcal_{\mathfrak{u}^-} \subseteq \Xcal_{\mathfrak{t}^-}$ and so $\Ucal \subseteq \Tcal$.  Consequently, we begin by studying such nested torsion pairs, which are known to give rise to filtrations (see \cite{BKT}).

\begin{definition}
Let $\Ucal, \Scal, \Fcal$ be full subcategories of an abelian category $\Acal$.  We will call $(\Ucal,\Scal,\Fcal)$ a \textbf{filtration triple} if $\Hom_\Acal(\Ucal, \Scal) = 0$, $\Hom_\Acal(\Ucal, \Fcal) =0$, $\Hom_\Acal(\Scal, \Fcal) =0$ and, for every object $X$ in $\Acal$, there exists a filtration \[ 0 = X_0 \subseteq X_1 \subseteq X_2 \subseteq X_3 = X\] such that $X_1/X_0 = X_1$ lies in $\Ucal$, $X_2/X_1$ lies in $\Scal$ and $X_3/X_2 = X/X_2$ lies in $\Fcal$.
\end{definition}

\begin{proposition}\label{prop: tp vs triples}
Let $\Acal$ be an abelian category.  There is a bijection between:
\begin{enumerate}
\item pairs of torsion pairs $(\Ucal, \Vcal)$, $(\Tcal, \Fcal)$ in $\Acal$ with $\Ucal \subseteq \Tcal$; and
\item filtration triples $(\Ucal, \Scal, \Fcal)$ in $\Acal$.
\end{enumerate}  The mutually inverse bijections are given by:
\[
(\Ucal, \Vcal), (\Tcal, \Fcal) \mapsto (\Ucal, \Vcal\cap\Tcal, \Fcal)
\quad\text{and}\quad
(\Ucal, \Scal, \Fcal) \mapsto (\Ucal, \Scal\star\Fcal), (\Ucal\star\Scal, \Fcal).
\]

\end{proposition}
\begin{proof}
Let $(\Ucal, \Vcal)$, $(\Tcal, \Fcal)$ be a pair of torsion pairs in $\Acal$ such that $\Ucal \subseteq \Tcal$.  We show that $(\Ucal, \Vcal\cap\Tcal, \Fcal)$ is a filtration triple in $\Acal$.  Let $\Scal := \Vcal\cap\Tcal$.  Firstly, the Hom-orthogonality conditions are clear because $\Hom_\Acal(\Ucal, \Vcal) = 0 = \Hom_\Acal(\Tcal, \Fcal)$.  Let $X$ be an arbitrary object in $\Acal$ {with torsion decompositions 
$$
0\longrightarrow X_1\longrightarrow X\longrightarrow X/X_1\longrightarrow 0 
\quad\text{and}\quad
0\longrightarrow X_2\longrightarrow X\longrightarrow X/X_2\longrightarrow 0
$$
where $X_1$ is contained in $\Ucal$, $X/X_1$ is in $\Vcal$, $X_2$ is in $\Tcal$ and $X/X_2$ is in $\Fcal$.
  Then  $X_2/X_1 $ is the torsion-free part of $X_2$ with respect to $(\Ucal, \Vcal)$ and is a quotient of $X_2$, and so $X_2/X_1$ lies in $\Scal$. Hence
 $0=X_0\subseteq X_1\subseteq X_2\subseteq X$ is the desired filtration of $X$.}
  
Let $(\Ucal, \Scal, \Fcal)$ be a filtration triple in $\Acal$.  We show that $(\Ucal\star\Scal, \Fcal)$ is a torsion pair in $\Acal$; the proof that $(\Ucal, \Scal\star\Fcal)$ is a torsion pair is similar.  If we consider objects $U$ in $\Ucal$, $S$ in $\Scal$ and an exact sequence 
$$0 \longrightarrow U \longrightarrow X \longrightarrow S \longrightarrow 0$$ 
and we apply $\Hom_\Acal(-, \Fcal)$, then we obtain that $\Hom_\Acal(X, \Fcal) =0$ and hence $\Hom_\Acal(\Ucal\star\Scal, \Fcal) = 0$.  Moreover, by the definition of filtration triple, for each object $A$ in $\Acal$, we have a short exact sequence 
$$0 \longrightarrow A_2 \longrightarrow A \longrightarrow A/A_2 \longrightarrow 0$$
with $A/A_2$ in $\Fcal$, $A_2$ in $\Ucal\star\Scal$ because there is a short exact sequence 
$$0 \longrightarrow A_1 \longrightarrow A_2 \longrightarrow A_1/A_2 \longrightarrow 0$$ 
where $A_1$ lies in $\Ucal$ and $A_1/A_2$ lies in $\Scal$.  Thus $(\Ucal\star\Scal, \Fcal)$ is a torsion pair.
\end{proof}

\begin{proposition}\label{Prop: HRS tilt of triples}
Let $\Hcal$ be the heart of a t-structure {in a triangulated category $\Dcal$} and let $(\Ucal, \Scal, \Fcal)$ be a filtration triple in $\Hcal$.  The following statements hold. 
\begin{enumerate}
\item $(\Fcal, \Ucal[-1], \Scal[-1])$ is a filtration triple in $\Hcal_{\mathfrak{t}^-}$ where $\mathfrak{t}=(\Ucal\star\Scal, \Fcal)$.
\item $(\Scal, \Fcal, \Ucal[-1])$ is a filtration triple in $\Hcal_{\mathfrak{u}^-}$ where $\mathfrak{u}=(\Ucal, \Scal\star\Fcal)$.
\end{enumerate}
\end{proposition}
\begin{proof}
{We prove statement {(1)}; the proof of (2) is similar.}  We show that $(\Fcal, \Ucal[-1], \Scal[-1])$ is a filtration triple in $\Hcal_{\mathfrak{t}^-}$. First we prove the Hom-orthogonality conditions: \[\Hom_{\Hcal_{\mathfrak{t}^-}}(\Ucal[-1], \Scal[-1]) = \Hom_{\Dcal}(\Ucal[-1], \Scal[-1]) \cong \Hom_{\Dcal}(\Ucal, \Scal) = \Hom_{\Hcal}(\Ucal, \Scal){=0}. \] 
 Also, since $\Ucal,\Scal \subset \Tcal:=\Ucal\star\Scal$ and $(\Fcal, \Tcal[-1])$ is a torsion pair in $\Hcal_{\mathfrak{t}^-}$, it follows that $$\Hom_{\Hcal_{\mathfrak{t}^-}}(\Fcal, \Ucal[-1]) = 0 = \Hom_{\Hcal_{\mathfrak{t}^-}}(\Fcal, \Scal[-1]).$$
Next we prove the existence of a filtration by $(\Fcal, \Ucal[-1], \Scal[-1])$ for an arbitrary object $Y\in \Hcal_{\mathfrak{t}^-}$.  The torsion pair $(\Fcal, \Tcal[-1])$ with $\Tcal[-1] = \Ucal[-1]\star\Scal[-1]$ induces the following commutative diagram in $\Hcal_{\mathfrak{t}^-}$ with exact rows and columns, where the object $X$ is obtained as a pullback. 
\[ \xymatrix{
& & 0 \ar[d] & 0 \ar[d] & \\
0 \ar[r] & F \ar@{=}[d] \ar[r] & X \ar[d] \ar[r] & U[-1] \ar[d] \ar[r] & 0 \\
0 \ar[r] & F \ar[r] & Y \ar[d] \ar[r] & T[-1] \ar[d] \ar[r] & 0 \\
& & S[-1] \ar[d] \ar@{=}[r] & S[-1] \ar[d] & \\
& & 0 & 0 &
}\] 
Then, the sequence $0 = Y_0 \subseteq F \subseteq X \subseteq Y$ is a $(\Fcal, \Ucal[-1], \Scal[-1])$-filtration of $Y$.
\end{proof}
Let us summarise the situation in Proposition~\ref{Prop: HRS tilt of triples} as follows
$$\xymatrix{
  {(\Ucal,\Scal,\Fcal)\text{ in }\Hcal}\ar@{~>}[d]_(.45){\mathfrak{u}^-} \ar@{~>}[r]^(.45){\mathfrak{t}^-\quad} &  \; (\Fcal,\Ucal[-1],\Scal[-1])\text{ in }\Hcal_{\mathfrak{t}^-}   \\
(\Scal,\Fcal,\Ucal[-1])\text{ in }  \Hcal_{\mathfrak{u}^-}  }$$
  
We are now going to see that this diagram can be completed to a commutative triangle.  Recall that the collection of torsion pairs in an abelian category $\Hcal$ has a partial order: we say that $\mathfrak{u}$ is less than $\mathfrak{t}$ if $\Ucal \subseteq \Tcal$. This gives rise to a poset which we denote by  $\tors\Hcal$. 

\begin{proposition}\label{prop: HRS-tilt triangle}
\label{cor: tp in mod = tp in heart}
Let $\Hcal$ be the heart of a t-structure $\mathbb{T}=(\Xcal, \Ycal)$ in a triangulated category $\Dcal$, and let $\mathfrak{u}=(\Ucal,\Vcal)$ be a torsion pair in $\Hcal$.
The assignment taking $\mathfrak{t}=(\Tcal, \Fcal)$ to $\mathfrak{s} = (\Tcal\cap \Vcal, \: \Fcal\star \Ucal[-1])$ induces an  order-preserving bijection between:
\begin{enumerate}
\item torsion pairs $\mathfrak{t}=(\Tcal, \Fcal)$ in $\Hcal$ with $\Ucal\subseteq\Tcal$; and
\item torsion pairs $\mathfrak{s}=(\Scal, \Rcal)$ in $\Hcal_{\mathfrak{u}^-}$ with $\Scal\subseteq\Vcal$.
\end{enumerate}
Moreover, if $\mathfrak{t}$ and $\mathfrak{s}$ correspond to each other under this bijection, then $(\mathbb{T}_{\mathfrak{u}^-})_{\mathfrak{s}^-}=\mathbb{T}_{\mathfrak{t}^-}$. That is, we have a commutative diagram
$$\xymatrix{
  \mathbb{T}\ar@{~>}[d]_{\mathfrak{u}^-} \ar@{~>}[r]^{\mathfrak{t}^-} &  \; \mathbb{T}_{\mathfrak{t}^-}   \\
  \mathbb{T}_{\mathfrak{u}^-}  \ar@{~>}[ur]_{\mathfrak{s}^-}
}$$
\end{proposition}
\begin{proof}
 In view of Proposition \ref{prop: tp vs triples}, the statement can be rephrased in terms of a bijection between filtration triples $(\Ucal, \Scal, \Fcal)$ in $\Hcal$ with $\Scal\star\Fcal=\Vcal$ and filtration triples $(\Scal, \Fcal, \Ucal[-1])$ in 
 $\Hcal_{\mathfrak{u}^-}$ with the same property.
By Proposition~\ref{Prop: HRS tilt of triples}, every  filtration triple $(\Ucal, \Scal, \Fcal)$ in $\Hcal$ induces a filtration triple $(\Scal, \Fcal, \Ucal[-1])$ in $\Hcal_{\mathfrak{u}^-}$, which in turn induces a filtration triple $(\Ucal[-1], \Scal[-1], \Fcal[-1])$ in $(\Hcal_{\mathfrak{u}^-})_{\mathbf{v}^-}$, where $\mathbf{v}=(\Vcal,\Ucal[-1])$ is the tilted torsion pair in 
$\Hcal_{\mathfrak{u}^-}$.  By Remark~\ref{rem:inversetilt} we have $(\Hcal_{\mathfrak{u}^-})_{\mathbf{v}^-}=\Hcal[-1]$,
and the latter filtration triple corresponds to the filtration triple $(\Ucal, \Scal, \Fcal)$ in $\Hcal$. This establishes the desired bijection, which is order-preserving by construction.

Next, we prove the stated equality of t-structures by comparing the coaisles. We have 
$\Ycal_{\mathfrak{t}^-}=\Fcal\star\Ycal[-1]$ and
$\Ycal_{\mathfrak{u}^-}=\Vcal\star\Ycal[-1]$. Then, keeping in mind that $\Rcal=\Fcal\star\Ucal[-1]$, we obtain
$$
(\Ycal_{\mathfrak{u}^-})_{\mathfrak{s}^-}=\Rcal\star\Ycal_{\mathfrak{u}^-}[-1]= \Fcal\star\Ucal[-1]\star \Vcal[-1]\star\Ycal[-2]=\Fcal\star\Hcal[-1]\star \Ycal[-2]=\Fcal\star\Ycal[-1]=\Ycal_{\mathfrak{t}^-}.
\qedhere
$$
\end{proof}

\subsection{Mutation of cosilting torsion pairs} In this subsection, we apply Proposition \ref{prop: HRS-tilt triangle} to cosilting torsion pairs and determine when they are related by mutation.

We will see that this can be expressed in terms of the notion of a wide subcategory. Recall that a full subcategory $\Wcal$ of an abelian category $\Hcal$ is an \textbf{exact abelian subcategory} if it is closed under kernels and cokernels.  The subcategory $\Wcal$ is \textbf{wide} if it is an exact abelian subcategory that is closed under extensions. We will need the following lemma.

\begin{lemma}\label{lem: wide}
Let $\Dcal$ be an arbitrary triangulated category and let $\Hcal$ and $\Hcal'$ be hearts of  t-structures $(\Xcal, \Ycal)$ and $(\Xcal', \Ycal')$ respectively.   Then the following statements hold.\begin{enumerate}
\item Let $\Wcal$ be a full subcategory of $\Hcal$.  Then $\Wcal$ is an exact abelian subcategory of $\Hcal$ if and only if $\Wcal\star\Wcal[1]\subseteq\Wcal[1]\star\Wcal$ (equivalently, the cone of every morphism $f$ in $\Wcal$ is contained in $\Wcal[1]\star\Wcal$).
\item Let $\Wcal$ be a full subcategory contained in $\Hcal \cap \Hcal'$. Then $\Wcal$ is an exact abelian subcategory of $\Hcal$ if and only if $\Wcal$ is an exact abelian subcategory of $\Hcal'$.
\item Let $\Wcal$ be a full subcategory contained in $\Hcal \cap \Hcal'$. Then $\Wcal$ is a wide subcategory of $\Hcal$ if and only if $\Wcal$ is a wide subcategory of $\Hcal'$.
\end{enumerate}
\end{lemma}
\begin{proof}
Statements (2) and (3) follow immediately from the first statement because then the required closure conditions depend only on the ambient triangulated category and not on the specific t-structures.  We therefore prove statement (1). Let $f$ be a morphism in $\Wcal$ and let $L := \cone{f}$.  Consider the triangle $L_\Xcal \to L \to L_\Ycal \to L_\Xcal[1]$ corresponding to the t-structure $(\Xcal, \Ycal)$.  Then $\ker{f} = L_\Xcal[-1]$ and $\coker{f}= L_\Ycal$, so the statement follows.
\end{proof}

Let us now consider the following situation.

\begin{setup}\label{setup general} 
Let $\Dcal$ be a compactly generated triangulated category with a t-structure $\mathbb{T} = (\Xcal, \Ycal)$ and let $\Hcal$ be the heart of $\mathbb{T}$. Let further $\mathfrak{u}=(\Ucal,\Vcal)$ and $\mathfrak{t}=(\Tcal, \Fcal)$ be torsion pairs in $\CosiltP\Hcal$  such that $\Ucal \subseteq \Tcal$.  Let us fix the following notation:
\begin{itemize}
\item $\mathfrak{s}=(\Scal, \Rcal)= (\Tcal\cap \Vcal, \: \Fcal\star \Ucal[-1])$ is the torsion pair in $\Hcal_{\mathfrak{u}^-}$ uniquely determined by $\mathfrak{u}$ and $\mathfrak{t}$ according to Proposition \ref{prop: HRS-tilt triangle};
\item $\mathfrak{r} = (\Rcal, \Scal[-1])$ is the tilted torsion pair of $\mathfrak{s}$ in $(\mathbb{T}_{\mathfrak{u}^-})_{\mathfrak{s}^-}=\mathbb{T}_{\mathfrak{t}^-}$.
\end{itemize}
\end{setup}
  
\noindent We can visualise the setup in the following commutative diagram:
$$\xymatrix{
  \mathbb{T}\ar@{~>}[d]_{\mathfrak{u}^-} \ar@{~>}[r]^{\mathfrak{t}^-} &  \; \mathbb{T}_{\mathfrak{t}^-}  \ar@<-1ex>@{~>}[dl]_{\mathfrak{r}^+}   \\
  \mathbb{T}_{\mathfrak{u}^-}  \ar@{~>}[ur]_{\mathfrak{s}^-}                     }$$

\begin{theorem}\label{thm: mutation general}
Suppose we are in Setup~\ref{setup general}. 
Then $\mathfrak{s}$ is in $\CosiltP{\Hcal_{\mathfrak{u}^-}}$ and $\mathfrak{r}$ is in 
$\CosiltP{\Hcal_{\mathfrak{t}^-}}$.  Moreover, the following statements are equivalent.
\begin{enumerate}
\item $\mathfrak{t}$ is a right mutation of $\mathfrak{u}$.
\item $\Scal$ is a wide subcategory of $\Hcal$.
\item $\mathfrak{s}=(\Scal, \Rcal)$ is a hereditary torsion pair in $\Hcal_{\mathfrak{u}^-}$.
\item $\Scal[-1]$ is a TTF class in $\Hcal_{\mathfrak{t}^-}$.
\end{enumerate}
\end{theorem}
\begin{proof}
Let $\sigma_\mathfrak{t}$ and $\sigma_\mathfrak{u}$ denote the pure-injective cosilting objects in $\Dcal$ associated to $\mathfrak{u}$ and $\mathfrak{t}$ respectively.  The fact that $\mathfrak{s}$ is a cosilting torsion pair follows immediately from the fact that  
$(\mathbb{T}_{\mathfrak{u}^-})_{\mathfrak{s}^-}=\mathbb{T}_{\mathfrak{t}^-} = \mathbb{T}_{\sigma_\mathfrak{t}}$ by Proposition \ref{prop: HRS-tilt triangle}. Similarly, we have that $(\mathbb{T}_{\mathfrak{t}^-})_{\mathfrak{r}^-}=\mathbb{T}_{\mathfrak{u}^-}{[-1]} = \mathbb{T}_{\sigma_\mathfrak{u}}{[-1]}$ by Remark~\ref{rem:inversetilt}.  From Corollary~\ref{cor:cotilting pure-injective} we infer that
 $\mathfrak{s}$ is in $\CosiltP{\Hcal_{\mathfrak{u}^-}}$ and $\mathfrak{r}$ is in $\CosiltP{\Hcal_{\mathfrak{t}^-}}$.
 
(1)$\Leftrightarrow$(3):  Since $\sigma_\mathfrak{t}$ and $\sigma_\mathfrak{u}$ are pure-injective, the hearts $\Hcal_{\mathfrak{u}^-} = \Hcal_{\sigma_\mathfrak{u}}$ and  $\Hcal_{\mathfrak{t}^-} = \Hcal_{\sigma_\mathfrak{t}}$ are Grothendieck categories by Theorem~\ref{thm:purity and silting}.  It is well-known that, in a Grothendieck category, a torsion pair is hereditary if and only if it is cogenerated by a class of injective objects.

Now, if $\sigma_\mathfrak{t}$ is  a right mutation of $\sigma_\mathfrak{u}$, then we know from Theorem~\ref{thm:mutation as HRS-tilt} that the t-structure  $\mathbb{T}_{\sigma_\mathfrak{t}} = \mathbb{T}_{\mathfrak{t}^-}$ is   the right HRS-tilt of $\mathbb{T}_{\sigma_\mathfrak{u}} = \mathbb{T}_{\mathfrak{u}^-}$ at the torsion pair $({}^{\perp_0}H^0_{\mathfrak{u}^-}(\Escr),\,\Cogen{H^0_{\mathfrak{u}^-}(\Escr)})$ in $\Hcal_{\mathfrak{u}^-}$ induced by $\Escr=\Prod{\sigma_\mathfrak{u}}\cap\Prod{\sigma_\mathfrak{t}}$. It follows that the torsion pairs $\mathfrak{s}=(\Scal,\Rcal)$ and $({}^{\perp_0}H^0_{\mathfrak{u}^-}(\Escr),\,\Cogen{H^0_{\mathfrak{u}^-}(\Escr)})$ coincide.  We know that $H^0_{\mathfrak{u}^-}(\Escr)$ is a set of injective objects in $\Hcal_{\mathfrak{u}^-}$ (see Proposition~\ref{prop:summary cosilting}) and so we have shown that (3) holds. 

Conversely, suppose that $\Scal = {}^{\perp_0}\Ical$ for some class $\Ical = \Prod\Ical$ of injective objects in $\Hcal_{\mathfrak{u}^-}$. Then there is a class $\Escr = \Prod\Escr \subseteq\Prod{\sigma_\mathfrak{u}}$ such that $\Ical = H^0_{\mathfrak{u}^-}(\Escr)$ (see Proposition~\ref{prop:summary cosilting} again) and, moreover, $\mathfrak{s} = ({}^{\perp_0}H^0_{\mathfrak{u}^-}(\Escr), \Cogen{H^0_{\mathfrak{u}^-}(\Escr)})$ is a cosilting torsion pair in $\Hcal_{\mathfrak{u}^-}$ because $(\mathbb{T}_{\mathfrak{u}^-})_{\mathfrak{s}^-}=\mathbb{T}_{\mathfrak{t}^-} = \mathbb{T}_{\sigma_\mathfrak{t}}$  by Proposition \ref{prop: HRS-tilt triangle}. It follows from Proposition~\ref{prop:cover vs finite type} and Theorem~\ref{existence left mutation} that $\sigma_{\mathfrak{u}^-}$  admits a right mutation $\tilde{\sigma}$ with respect to $\Escr$, which is equivalent to $\sigma_\mathfrak{t}$ because   $\mathbb{T}_{\tilde\sigma}$ is the right HRS-tilt of $\mathbb{T}_{\mathfrak{u}^-}$ at $\mathfrak{s}$.

(2)$\Leftrightarrow$(3): Note that $\Scal$ is a full subcategory of $\Dcal$ that is contained in $\Hcal \cap \Hcal_{\mathfrak{u}^-}$ so, by Lemma \ref{lem: wide}(3), we have that $\Scal$ is a wide subcategory of $\Hcal$ if and only if $\Scal$ is a wide subcategory of $\Hcal_{\mathfrak{u}^-}$.  Since $\Scal$ is a torsion class in $\Hcal_{\mathfrak{u}^-}$, it is closed under extensions and quotients.  The equivalence of (2) and (3) then follows from the fact that $\mathfrak{s}$ is hereditary if and only if $\Scal$ is closed under subobjects in $\Hcal_{\mathfrak{u}^-}$ if and only if $\Scal$ is closed under kernels in $\Hcal_{\mathfrak{u}^-}$ if and only if $\Scal$ is a wide subcategory of $\Hcal_{\mathfrak{u}^-}$.

(2)$\Leftrightarrow$(4): By analogous arguments as in the previous paragraph, we observe that $\Scal$ is wide in $\Hcal$ if and only if $\mathfrak{r} = (\Rcal, \Scal[-1])$ is a cohereditary torsion pair in $\Hcal_{\mathfrak{t}^-}$, that is, if and only if $\Scal[-1]$ is closed under quotient objects in $\Hcal_{\mathfrak{t}^-}$. Since a torsion-free class in a Grothendieck category is always closed under coproducts, this happens if and only if $\Scal[-1]$ is a torsion class in $\Hcal_{\mathfrak{t}^-}$.
\end{proof}

Inspired by Theorem~\ref{thm: mutation general}(2),
we conclude the section with a useful criterion for when the intersection of a torsion and a torsion-free class is a wide subcategory. It is closely related to a construction of Ingalls and Thomas~\cite[\S2.3]{InTh}, which was generalised in~\cite[\S3]{MSt}.

\begin{proposition}\label{prop:wide intervals}
Let $\Hcal$ be an abelian category and $\mathfrak{u}=(\Ucal,\Vcal)$ and $\mathfrak{t}=(\Tcal, \Fcal)$ be torsion pairs in $\Hcal$ such that $\Ucal \subseteq \Tcal$. Then the following are equivalent:
\begin{enumerate}
\item $\Tcal\cap\Vcal$ is a wide subcategory of $\Hcal$.
\item If $g \colon T \to V$ is a map in $\Hcal$ with $T$ in $\Tcal$ and $V$ in $\Vcal$, then $\ker{g}$ lies in $\Tcal$ and $\coker{g}$ lies in $\Vcal$.
\end{enumerate}
\end{proposition}

\begin{proof}
Let us denote $\Wcal:=\Tcal\cap\Vcal$.	

(1) $\Rightarrow$ (2): Let $g \colon T \to V$ be a morphism in $\Hcal$ with $T$ in $\Tcal$ and $V$ in $\Vcal$.  By Proposition \ref{prop: tp vs triples}, we have that $T$ lies in $\Ucal \star \Wcal$ and $V$ lies in $\Wcal \star \Fcal$, so there are short exact sequences 
\[\xymatrix{0 \ar[r]& U \ar[r]& T \ar[r]^b& S_1 \ar[r]& 0} \quad\text{and}\quad \xymatrix{0 \ar[r]& S_2 \ar[r]^a& V \ar[r]& F \ar[r]& 0}\] 
with $S_1$ and $S_2$ in $\Wcal$, $U$ in $\Ucal$ and $F$ in $\Fcal$.  Since $\Hom_A(U,V) =0$ and $\Hom_A(S_1, F) = 0$, we may use the kernel/cokernel properties to obtain a morphism $f \colon S_1 \to S_2$ such that $g=afb$.  By assumption we have that both $\ker{f}$ and $\coker{f}$ lie in $\Wcal$.  By taking the pullback of the canonical embedding $\ker{f} \to S_1$ along $b$, we obtain a short exact sequence 
$$\xymatrix{0 \ar[r]& U \ar[r]& K \ar[r]& \ker{f} \ar[r]& 0.}$$  By checking the universal property, it is straightforward to show that $K \cong \ker{g}$ and, hence, $\ker{g}$ lies in $\Ucal\star \Wcal=\Tcal$.  A dual argument yields that there is a short exact sequence 
$$\xymatrix{0 \ar[r]& \coker{f} \ar[r]& \coker{g} \ar[r]& F \ar[r]& 0}$$
and hence $\coker{g}$ lies in $\Wcal\star \Fcal=\Vcal$.

(2) $\Rightarrow$ (1): Let $g \colon T \to V$ be a map with $T$ and $V$ in $\Wcal$.  Then $\ker{g}$ lies in $\Tcal$ by assumption and $\ker{g}$ lies in $\Vcal$ because $\Vcal$ is a torsion-free class in $\Hcal$.  Therefore, $\ker{g}$ lies in $\Wcal$. Similarly, we have that $\coker{g}$ lies in $\Wcal$.
\end{proof}

\section{Mutations of torsion pairs in $\Db R$}\label{Sec Mut in small} 
In this section we will assume that $\Dcal$ is the derived category $\D R$ of a left coherent ring $R$.  We will use the techniques developed in the previous sections to study torsion pairs in hearts of t-structures in $\Db R$. In order to do that, we must first lift these t-structures to the whole derived category $\Dcal$ and then extend the torsion pairs from the original heart to the lifted heart. Let us begin with the process of extending torsion pairs within a locally coherent Grothendieck category (of which $\Mod R$ is, by assumption on $R$, an example). 

\begin{proposition}\cite[Lemma 4.4]{CBlfp}\label{thm: CB bijections}
Let $\Acal$ be a locally coherent Grothendieck category and $\mathfrak{t}=(\Tcal,\Fcal)$ a torsion pair in $\fp \Acal$. Then
\begin{enumerate}
\item the pair $\overrightarrow{\mathfrak{t}}=(\overrightarrow{\Tcal},\overrightarrow{\Fcal}):=({}^\perp(\Tcal^\perp),\Tcal^\perp)$ in $\Acal$ is a torsion pair, called the \textbf{lift of $\mathfrak{t}$};
\item the assignment of a torsion pair in $\fp \Acal$ to its lift in $\Acal$ induces a bijection between
\begin{enumerate}
\item torsion pairs in $\fp \Acal$;
\item torsion pairs $(\Xcal,\Ycal)$ of finite type in $\Acal$ such that $\Xcal\cap \fp\Acal$ is a torsion class in $\fp \Acal$.
\end{enumerate}
\end{enumerate}
If $\Acal$ is \textbf{locally noetherian} (i.e. if $\Acal$ is a Grothendieck category with a set of noetherian generators), then the assignement above establishes a bijection between torsion pairs in $\fp\Acal$ and torsion pairs of finite type in $\Acal$.
\end{proposition}

Note that torsion pairs of finite type in $\Mod{R}$, for any ring $R$, are precisely the ones in $\CosiltP{\Mod{R}}$. This also holds for a more general class of hearts in $\D R$, see Proposition~\ref{prop:cover vs finite type}.

Let us now consider an analogous result for certain t-structures in $\Db R$. Recall that a t-structure $\mathbb{T} = (\Xcal, \Ycal)$ in $\D R$ or in $\Db R$ is called \textbf{intermediate} if there are integers $m\geq n$ such that $$\mathrm{D}^{\geq 0}[m] \subseteq  \Ycal \subseteq \mathrm{D}^{\geq 0}[n]$$ where $\mathrm{D}^{\geq 0}$ is the standard coaisle in $\D R$ or in $\Db R$, respectively. 

\begin{proposition}\cite[Lemma 3.1 and Corollary 4.2]{MZ}\label{MarksZvonareva}
Let $R$ be a left coherent ring and let $\mathbb{T}:=(\Ucal,\Vcal)$ be an intermediate t-structure in $\Db R$. Then 
\begin{enumerate}
\item the pair $\overrightarrow{\mathbb{T}}=(\overrightarrow{\Ucal},\overrightarrow{\Vcal}):=({}^\perp(\Ucal^\perp),\Ucal^\perp)$ in $\D R$ is a t-structure, called the \textbf{lift of $\mathbb{T}$};
\item the assignment of a t-structure in $\Db R$ to its lift in $\D R$ induces a bijection between
\begin{enumerate}
\item intermediate t-structures in $\Db R$;
\item intermediate, compactly generated t-structures in $\D R$ with a locally coherent heart $\overrightarrow{\Hcal}$ such that $\fp\overrightarrow{\Hcal}=\overrightarrow{\Hcal}\cap \Db R$.
\end{enumerate}
\end{enumerate}
\end{proposition}

Note that the t-structures in $\D R$ obtained as lifts of t-structures in $\Db R$ correspond to pure-injective cosilting objects (see \cite[Theorem 4.9]{AMV3}), that is, for every intermediate t-structure $\mathbb{T}$ in $\Db R$ with heart $\Hcal$, there is a pure-injective cosilting object $\sigma$ such that $\overrightarrow{\mathbb{T}}=\mathbb{T}_\sigma$. The t-structure $\overrightarrow{\mathbb{T}}=\mathbb{T}_\sigma$ has a locally coherent Grothendieck heart $\Hcal_\sigma$ with $\fp\Hcal_\sigma=\Hcal$. We will often denote this heart by $\overrightarrow{\Hcal}$, see the remark below.

\begin{remark}
Let us briefly justify the notation $\overrightarrow{(-)}$ used in the assignments discussed in the two theorems. In fact, if $\Hcal$ is a locally coherent Grothendieck category it is shown in \cite{CBlfp} that, for a torsion pair $(\Tcal,\Fcal)$ in $\fp\Hcal$, we have that ${}^\perp(\Tcal^\perp)$ and $\Tcal^\perp$ are the closure under direct limits of $\Tcal$ and $\Fcal$, respectively, inside $\Hcal$. Similarly, if $(\Ucal,\Vcal)$ is an intermediate t-structure in $\Db R$ for a left coherent ring $R$, it is shown in \cite{MZ} that ${}^\perp(\Ucal^\perp)$, $\Ucal^\perp$ and the heart ${}^\perp(\Ucal^\perp)[-1]\cap \Ucal^\perp$ are the closure under directed homotopy colimits of $\Ucal$, $\Vcal$ and of the heart $\Ucal[-1]\cap\Vcal$,  respectively, inside $\D R$. Recall that in the derived category of a ring, we may consider directed homotopy colimits as the derived functor of the direct limit functor.
\end{remark}

Finally, the following proposition relates the two lifts enunciated in the theorems above via HRS-tilting. This result is essentially contained in \cite[Proposition 5.1]{Saorin} and \cite[Proposition 5.1]{MZ}. We include a proof since the formulations therein are slightly different. The 

\begin{proposition}\label{lift compatible HRS}
Let $R$ be a left coherent ring and $\mathbb{T}=(\Xcal,\Ycal)$ an intermediate t-structure in $\Db R$ with heart $\Hcal$. Consider a torsion pair $\mathfrak{t}=(\Tcal,\Fcal)$ in $\Hcal$ and a torsion pair $\mathfrak{p}=(\Pcal,\Qcal)$ in the heart $\overrightarrow{\Hcal}$ of $\overrightarrow{\mathbb{T}}$ in $\D R$. Then $\mathfrak{p}$ is the lift of $\mathfrak{t}$ to $\overrightarrow\Hcal$ if and only if $\overrightarrow{\mathbb{T}}_{\mathfrak{p}^-}$ is the lift of $\mathbb{T}_{\mathfrak{t}^-}$ to $\D R$, i.e.
$$\mathfrak{p}=\overrightarrow{\mathfrak{t}}\Leftrightarrow \overrightarrow{\mathbb{T}}_{\mathfrak{p}^-}=\overrightarrow{\mathbb{T}_{\mathfrak{t}^-}}$$
In particular, for any torsion pair $\mathfrak{t}$ in $\Hcal$, the heart of $\overrightarrow{\mathbb{T}}_{\overrightarrow{\mathfrak{t}}^-}$ is a locally coherent Grothendieck category with $\Hcal_{\mathfrak{t}^-}$ as its subcategory of finitely presented objects.
\end{proposition}
\begin{proof}
Suppose that $\overrightarrow{\mathbb{T}}_{\mathfrak{p}^-}=\overrightarrow{\mathbb{T}_{\mathfrak{t}^-}}$. 
Since $H^0_{\overrightarrow{\mathbb{T}}}$ sends directed homotopy colimits in $\D R$ to direct limits in $\overrightarrow{\Hcal}$ (see \cite[Lemma 5.7]{SSV}), it follows that an object $X$ of $\D R$ lies in $\overrightarrow{\Tcal}$ if and only if it lies in $\overrightarrow \Hcal\cap \overrightarrow{\Xcal_{\mathfrak{t}^-}}$. This latter intersection coincides by assumption with $\overrightarrow{\Xcal}_{\mathfrak{p}^-}\cap \overrightarrow{\Hcal}$ which is precisely $\Pcal$, thus proving the desired equality.

Conversely, suppose that $\mathfrak{p}=\overrightarrow{\mathfrak{t}}$. Let $X$ be an object in $\overrightarrow{\Xcal_{\mathfrak{t}^-}}$. Then, there is a directed coherent diagram (i.e.\ an object of the derived category $\D{\Mod{R}^I}$  of $I$-shaped diagrams of $R$-modules) that gives rise to $(X_i)_{i\in I}$ in $\Xcal_{\mathfrak{t}^-}^I$ such that $\mathbb{L}\varinjlim_{i\in I}X_i=X$. Again since $H^0_{\overrightarrow{\mathbb{T}}}$ sends directed homotopy colimits in $\D R$ to direct limits in $\overrightarrow{\Hcal}$, we have that $H^0_{\overrightarrow{\mathbb{T}}}(X)$ lies in $\overrightarrow{\Tcal}$, and this latter class coincides with $\Pcal$ by assumption. As a consequence, since $\overrightarrow{\mathbb{T}}$ is intermediate and, thus, nondegenerate, we have that $\overrightarrow{\Xcal_{\mathfrak{t}^-}}\subseteq \overrightarrow{\Xcal}_{\mathfrak{p}^-}$ (see also Remark \ref{rem:inversetilt}). Conversely, since both $\overrightarrow{\Tcal}$ and $\overrightarrow{\Xcal}$ are contained in $\overrightarrow{\Xcal_{\mathfrak{t}^-}}$, it follows that $\overrightarrow{\Xcal}_{\mathfrak{p}^-}=\overrightarrow{\Xcal}\star \Pcal=\overrightarrow{\Xcal}\star \overrightarrow{\Tcal}$ is contained in $\overrightarrow{\Xcal_{\mathfrak{t}^-}}$. The final statement follows from Proposition \ref{MarksZvonareva}.
\end{proof}

Informally, one of the implications of the statement above tells us that the tilt at the lifted torsion pair coincides with the lift of the tilted t-structure. In other words, the operations \textit{lift} and \textit{tilt}, when correctly interpreted, commute.

We are now ready to establish the setup with which we will work in this section.

\begin{setup}\label{setup: restricted}
Let $R$ be a left coherent ring and $\mathbb{T}$ an intermediate t-structure in $\Db R$ with heart $\Hcal$. Consider two torsion pairs $\mathfrak{u}=(\Ucal, \Vcal)$ and $\mathfrak{t}=(\Tcal, \Fcal)$ in $\Hcal$ with $\Ucal \subseteq \Tcal$ and let us fix the notation:
\begin{itemize}
\item $\sigma$ denotes a pure-injective cosilting object in $\D R$ such that $\overrightarrow{\mathbb{T}}=\mathbb{T}_\sigma$;
\item $\mathfrak{s}=(\Scal, \Rcal)= (\Tcal\cap \Vcal, \: \Fcal\star \Ucal[-1])$ denotes the torsion pair in $\Hcal_{\mathfrak{u}^-}$ determined by $\mathfrak{u}$ and $\mathfrak{t}$ by Proposition \ref{prop: HRS-tilt triangle};
\item $\mathfrak{r} = (\Rcal, \Scal[-1])$ denotes the tilted torsion pair of $\mathfrak{s}$ in $(\mathbb{T}_{\mathfrak{u}^-})_{\mathfrak{s}^-}=\mathbb{T}_{\mathfrak{t}^-}$. 
\end{itemize}
\end{setup}

An important example of Setup \ref{setup: restricted} is given by taking $\mathbb{T}$ to be the standard t-structure in $\D R$ and $\mathfrak{u}=(\Ucal, \Vcal)$ and $\mathfrak{t}=(\Tcal, \Fcal)$ any torsion pairs in $\mod R$ with $\Ucal\subseteq \Tcal$.  

The following lemma makes it clear that, in our setup, Proposition \ref{prop: HRS-tilt triangle} is compatible with the operations of lifting of torsion pairs and t-structures. This will be useful for us later on. 

\begin{lemma}\label{lem: s restricts} Suppose we are in Setup \ref{setup: restricted} and consider the torsion pair $\overrightarrow{\mathfrak{t}}$ and $\overrightarrow{\mathfrak{u}}$ in $\overrightarrow\Hcal$. Then the torsion pairs $(\overrightarrow\Tcal\cap \overrightarrow\Vcal,\overrightarrow\Fcal\star\overrightarrow\Ucal[-1])$ in $\overrightarrow{\Hcal}_{\overrightarrow{\mathfrak{u}}^-}$ and $(\overrightarrow\Fcal\star\overrightarrow\Ucal[-1],(\overrightarrow\Tcal\cap \overrightarrow\Vcal)[-1])$ in $\overrightarrow{\Hcal}_{\overrightarrow{\mathfrak{t}}^-}$ given by Proposition \ref{prop: HRS-tilt triangle} coincide with $\overrightarrow{\mathfrak{s}}$ and $\overrightarrow{\mathfrak{r}}$, respectively.
\end{lemma}
\begin{proof}
Note that, by Proposition \ref{lift compatible HRS}, we have $\overrightarrow{\Hcal}_{\overrightarrow{\mathfrak{u}}^-}=\overrightarrow{\Hcal_{\mathfrak{u}^-}}$ and $\overrightarrow{\Hcal}_{\overrightarrow{\mathfrak{t}}^-}=\overrightarrow{\Hcal_{\mathfrak{t}^-}}$. By Proposition \ref{prop: HRS-tilt triangle}, the torsion pair for which a right HRS-tilt allows us to pass from $\overrightarrow{\Hcal_{\mathfrak{u}^-}}$ to $\overrightarrow{\Hcal_{\mathfrak{t}^-}}$ is uniquely determined as the torsion pair $\mathfrak{p}:=(\overrightarrow\Tcal\cap \overrightarrow\Vcal,\overrightarrow\Fcal\star\overrightarrow\Ucal[-1])$. On the other hand, it follows from Proposition \ref{lift compatible HRS} that this torsion pair must be $\overrightarrow{\mathfrak{s}}$. An analogous argument holds for the equality $(\overrightarrow\Fcal\star\overrightarrow\Ucal[-1],(\overrightarrow\Tcal\cap \overrightarrow\Vcal)[-1])=\overrightarrow{\mathfrak{r}}$ in $\overrightarrow{\mathbb{T}_{\mathbf{t}^-}}$.
\end{proof}

In the context of the Lemma above, when $\overrightarrow{\mathfrak{t}}$ is a right mutation of $\overrightarrow{\mathfrak{u}}$, the torsion pair $\overrightarrow{\mathfrak{s}}$ is a hereditary torsion pair of finite type, as shown in Theorem \ref{thm: mutation general} and Proposition~\ref{prop:cover vs finite type}. We will make use of the close relationship between such torsion pairs and the spectrum of locally coherent Grothendieck categories, which we summarise in the next theorem.

\begin{theorem}[\cite{KraLoc,Herzog}]\label{thm: Krause Herzog} Let $\Acal$ be a locally coherent Grothendieck category $\Acal$. The (isoclasses of) indecomposable injective objects form a topological space, $\Spec\Acal$, with a basis of open subsets given by sets of the form
$${\mathcal O}(C)=\{E\in\Spec\Acal\mid \Hom_\Hcal(C,E)\neq0\},\; C\in\fp\Acal.$$ 
There are bijections between:
\begin{enumerate}
\item[(a)] hereditary torsion pairs  of finite type in $\Acal$;
\item[(b)]  Serre subcategories of $\fp\Acal$; and
\item[(c)] open subsets of $\Spec \Acal$.
\end{enumerate}
The bijection between (a) and (b) is given by the assignments $(\Scal, \Rcal) \mapsto \Scal\cap\fp\Acal$ and $\Lcal \mapsto ({}^{\perp_0}(\Lcal^{\perp_0}), \Lcal^{\perp_0})$. The assignment (b)$\to$(c) takes a Serre subcategory 
$\Lcal$ to ${\mathcal O}=\{E\in\Spec\Acal\mid E\not\in\Lcal^{\perp_0}\}$. The assignment (c)$\to$(a) maps an open set $\mathcal O$ to the hereditary torsion pair $(\Scal,\Rcal)$ cogenerated by the complement ${\mathcal O}^c$. 
\end{theorem}

Let us come back to Setup \ref{setup: restricted}. 
We are now in a position to show that, if the lifted torsion pairs $\overrightarrow{\mathfrak{t}}$ and $\overrightarrow{\mathfrak{u}}$ in $\overrightarrow{\Hcal}$ are related by mutation, then this mutation is controlled by objects of $\Hcal=\fp\overrightarrow{\Hcal}$.

\begin{theorem}\label{Thm: mutations are wide}
Suppose we are in Setup \ref{setup: restricted}.   The following statements are equivalent.
\begin{enumerate}
\item $\overrightarrow{\mathfrak{t}}$ is a right mutation of $\overrightarrow{\mathfrak{u}}$.
\item $\Scal$ is a wide subcategory of $\Hcal$.
\item If $g \colon T \to V$ is a map in $\Hcal$ with $T$ in $\Tcal$ and $V$ in $\Vcal$, then $\ker{g}$ lies in $\Tcal$ and $\coker{g}$ lies in $\Vcal$.
\end{enumerate}
\end{theorem}

\begin{proof}
(1)$\Rightarrow$(2): From Theorem~\ref{thm: mutation general} and Lemma~\ref{lem: s restricts}, the class $\overrightarrow{\Scal}$ is  a wide subcategory in $\overrightarrow{\Hcal}$. We prove that $\Scal=\overrightarrow{\Scal}\cap\Hcal$, thus showing that $\Scal$ is a wide subcategory of $\Hcal$. We have $\Scal=\Vcal\cap\Tcal=\overrightarrow{\Vcal}\cap\overrightarrow{\Tcal}\cap\Hcal$. The latter class coincides with $\overrightarrow{\Vcal}\cap\overrightarrow{\Tcal}\cap \Db{R}$ by Proposition \ref{MarksZvonareva} and, moreover, from Lemma \ref{lem: s restricts}, we have that it equals $\overrightarrow{S}\cap\Db{R}$. Using Proposition \ref{MarksZvonareva} again, we conclude our claim.

(2) $\Rightarrow$ (1): Both hearts $\Hcal$ and $\Hcal_{\mathfrak{u}^-}$ in $\Db R$ contain $\Wcal$. By Lemma~\ref{lem: wide}(3), our assumption implies that $\Wcal=\Scal\cap\Hcal_{\mathfrak{u}^-}$ is a wide subcategory of $\Hcal_{\mathfrak{u}^-}$. In fact, it is even a Serre subcategory:  it is closed under quotients in $\Hcal_{\mathfrak{u}^-}$ because $\Scal$ is a torsion class, and so it is also closed under subobjects. By Theorem \ref{thm: Krause Herzog} we then have that $\overrightarrow{\mathfrak{s}}$ is a hereditary torsion pair of finite type and, therefore, by Lemma \ref{lem: s restricts} and Theorem~\ref{thm: mutation general}, we conclude that $\overrightarrow{\mathfrak{t}}$ is a right mutation of $\overrightarrow{\mathfrak{u}}$.

(2) $\Leftrightarrow$ (3): This is an immediate consequence of Proposition~\ref{prop:wide intervals}.
\end{proof}

\section{Mutation and simple objects} \label{sec:mutation and simples}

In this section we specialise the results of Section \ref{Sec Mut in small} to the case where the heart $\Hcal$ (of an intermediate t-structure in $\Db R$, $R$ left coherent) is a length category. These hearts are known to occur frequently when $R$ is an artinian ring. In this setting we will show that mutation is controlled by simple objects. 

\subsection{Abelian length categories.} Recall that an abelian category $\Acal$ is called a \textbf{length category} if $\filt{\Scal}=\Acal$, where $\Scal$ is the set of simple objects in $\Acal$.  A Grothendieck category $\Gcal$ is called \textbf{locally finite} if it has a set of finite length generators. Recall that an object is of finite length if and only if it is both noetherian and artinian. By \cite[Proposition~8.2]{Popescu} we have that $\Gcal$ is locally finite if and only if $\fp\Gcal$ is a length category if and only if $\Gcal$ is locally noetherian and $\Filt{\Omega} = \Gcal$ where $\Omega$ is the set of simple objects in $\Gcal$.  In particular, if in Setup \ref{setup: restricted} $\Hcal$ is a length category, Proposition \ref{thm: CB bijections} tells us that every torsion pair of finite type in $\overrightarrow{\Hcal}$ is of the form $\overrightarrow{\mathfrak{v}}$, for a torsion pair $\mathfrak{v}$ in $\Hcal$. We will replace Setup \ref{setup: restricted} with the following.

\begin{setup}\label{setup: length} (= Setup \ref{setup: restricted} + $\Hcal$ length category) Let $R$ be a left coherent ring and $\mathbb{T}$ an intermediate t-structure in $\Db R$ whose heart $\Hcal$ is a length category. Consider two torsion pairs $\mathfrak{u}=(\Ucal, \Vcal)$ and $\mathfrak{t}=(\Tcal, \Fcal)$ in $\Hcal$ with $\Ucal \subseteq \Tcal$ and let us fix the notation:
\begin{itemize}
\item $\sigma$ denotes a pure-injective cosilting object in $\D R$ such that $\overrightarrow{\mathbb{T}}=\mathbb{T}_\sigma$;
\item $\mathfrak{s}=(\Scal, \Rcal)= (\Tcal\cap \Vcal, \: \Fcal\star \Ucal[-1])$ denotes the torsion pair in $\Hcal_{\mathfrak{u}^-}$ determined by $\mathfrak{u}$ and $\mathfrak{t}$ by Proposition \ref{prop: HRS-tilt triangle};
\item $\mathfrak{r} = (\Rcal, \Scal[-1])$ denotes the tilted torsion pair of $\mathfrak{s}$ in $(\mathbb{T}_{\mathfrak{u}^-})_{\mathfrak{s}^-}=\mathbb{T}_{\mathfrak{t}^-}$. 
\end{itemize}
\end{setup}  

An important example of Setup \ref{setup: length} is given by taking $\mathbb{T}$ to be the standard t-structure in $\Db R$, with $R$ being artinian, and any pair of torsion pairs $\mathfrak{u}=(\Ucal, \Vcal)$ and $\mathfrak{t}=(\Tcal, \Fcal)$ in $\mod R$ with $\Ucal \subseteq \Tcal$.  

\begin{remark} In view of Proposition \ref{thm: CB bijections}, mutation of torsion pairs in $\CosiltP{\overrightarrow{\Hcal}}$ admits an interpretation inside the lattice $\tors {\Hcal}$ of torsion classes in $\Hcal$ with partial order given by inclusion. We refer to \cite{DIRRT,Asai,BCZ,AP}  for details about the lattice structure of $\tors {\Hcal}$. 

 Following \cite[Section 6]{AP}, we will say that \textbf{$\mathfrak{t}$ is a right mutation of $\mathfrak{u}$} (and $\mathfrak{u}$ a left mutation of $\mathfrak{t}$) when we are in the situation of Theorem~\ref{Thm: mutations are wide}. Note also  that condition (2)  in that theorem, in the terminology of \cite{AP}, 
states that the interval ${[\Ucal, \Tcal]}$ is a wide interval of  $\tors{\Hcal}$.
\end{remark}

The following well-known theorem due to Ringel tells us that every object in a wide subcategory $\Wcal$ of $\Hcal$ admits a finite filtration by simple objects of $\Wcal$.  This is the point of view from which wide intervals are studied in \cite{AP}.    An object $X$ in $\Hcal$ is called a \textbf{brick} if $\End_{\Hcal}(X)$ is a skew-field.  A collection of bricks $\Omega$ in $\Hcal$ is called a \textbf{semibrick} if $\Hom_\Hcal(S, S') =0$ whenever $S$ and $S'$ are in $\Omega$ and $S\ne S'$.

\begin{theorem}[\cite{R}]\label{thm: ringel}
Let $\Acal$ be a length category.  If we assign to a wide subcategory $\Wcal$ of $\Acal$ the set $\Mcal$ of its simple objects, we obtain a semibrick $\Mcal$ such that $\Wcal = \filt{\Mcal}$. This yields a one-one correspondence between wide subcategories and semibricks  in $\Acal$. 
\end{theorem}

When a wide subcategory arises as in Theorem~\ref{Thm: mutations are wide}(2), it is possible to characterise its simple objects more precisely, and we do so in Lemma \ref{lem: simple bricks = ME}.  First we need the following definition.

\begin{definition}[\cite{AHL}]\label{def:almost torsion}
Let  $\mathfrak{u}=(\Ucal, \Vcal)$ be  torsion pair in an abelian category $\Hcal$. We say that a non-zero object $M$ in $\Vcal$ is  \textbf{almost torsion (for $\mathfrak{u}$)}
 if the following conditions are satisfied. 
\begin{enumerate}
\item[(i)]  All proper quotients of $M$ are in $\Ucal$.
\item[(ii)] For all short exact sequences $0 \to M \to Y \to Z \to 0$, with $Y$ in  $\Vcal$, we have that $Z$ lies in $\Vcal$.
\end{enumerate} 
\textbf{Almost torsion-free objects for $\mathfrak{u}$} are defined dually.
\end{definition}

\begin{remark}\label{rem: tf/t are simple}
(1) {\cite{AHL}, \cite[Theorem 2.3.6]{Ra}} If $\Hcal$ is the heart of a t-structure in an arbitrary triangulated category, then an object $M$ in $\Vcal$ is almost torsion if and only if $M$ becomes a (torsion) simple object in the tilted heart $\Hcal_{\mathfrak{u}^-}=\Vcal\star\Ucal[-1]$. The almost torsion-free objects for $\mathfrak{u}$ are precisely those $N$ in $\Ucal$ for which $N[-1]$ becomes a simple object in $\Hcal_{\mathfrak{u}^-}$.

(2) \cite[Section 2]{S} Let  $A$ be a finite-dimensional algebra and $\mathfrak{u}=(\Ucal, \Vcal)$  a torsion pair in $\mod A$ with lifted torsion pair $\overrightarrow{\mathfrak{u}}$ in $\Mod A$. The  finite-dimensional torsion-free, almost torsion modules for $\overrightarrow{\mathfrak{u}}$   coincide with the torsion-free, almost torsion modules for $\mathfrak{u}$ and are precisely the minimal extending modules defined in \cite{BCZ}. Moreover, all  torsion, almost torsion-free modules for  $\overrightarrow{\mathfrak{u}}$ are finite-dimensional and coincide with the  torsion, almost torsion-free modules for $\mathfrak{u}$, that is, with the minimal co-extending modules  from \cite{BCZ}. 

(3) It is  easy to check that the arguments in \cite{S} yield the same results for locally finite categories. In particular, in the situation of Setup \ref{setup: length}, every object $M$ which is torsion-free, almost torsion  for $\mathfrak{u}$ becomes a simple object in $\overrightarrow{\Hcal_{\mathfrak{u}^-}}$, and every object $N$ which is torsion, almost torsion-free  for $\mathfrak{t}$ gives rise to a simple object $N[-1]$ in $\overrightarrow{\Hcal_{\mathfrak{t}^-}}$.
\end{remark}

\begin{lemma}\label{lem: simple bricks = ME}
Suppose we are in Setup \ref{setup: length}.  If $\mathfrak{t}$ is a right mutation of $\mathfrak{u}$, then the following statements are equivalent for an object $B$ in $\Hcal$. \begin{enumerate}
\item $B$ is contained in the semibrick associated to the wide subcategory $\Scal$ of $\Hcal$.
\item $B$ is a torsion-free, almost torsion object for $\mathfrak{u}$ that belongs to $\Tcal$.
\item $B$ is a torsion, almost torsion-free object for $\mathfrak{t}$ that belongs to $\Vcal$.
\end{enumerate}
\end{lemma}
\begin{proof}
Let $\Mcal$ denote the semibrick associated to $\Scal$, that is, the set of simple objects of $\Scal$.  We show the equivalence of (1) and (2).  The equivalence of (1) and (3) uses a dual argument.

We begin by showing that every $B$ in $\Mcal$ is torsion-free, almost torsion for $\mathfrak{u}$. Let $N \cong B/K$ be a proper factor of $B$; we wish to show that $N$ lies in $\Ucal$ and hence condition (i) from Definition~\ref{def:almost torsion} holds.  Let $U_N$ and $V_N$ be objects in $\Ucal$ and $\Vcal$ such that there is a short exact sequence $0\to U_N\to N \to V_N \to 0$. First suppose that $U_N\neq0$ and consider the pullback diagram:
\[ \xymatrix{
& & 0 \ar[d] & 0 \ar[d] & \\
0 \ar[r] & K \ar@{=}[d] \ar[r] & X \ar[d]^g \ar[r] & U_N \ar[d] \ar[r] & 0 \\
0 \ar[r] & K \ar[r]^h & B \ar[d]^f \ar[r]^l & N \ar[d] \ar[r] & 0 \\
& & V_N \ar[d] \ar@{=}[r] & V_N \ar[d] & \\
& & 0 & 0 &
}\]  By condition (3) of Theorem~\ref{Thm: mutations are wide}  applied to $f$, we have that $X$ lies in $\Scal$. Hence the morphism $g$ is an isomorphism and so $V_N = 0$.  That is, we have $N =U_N$, which lies in $\Ucal$. Now suppose that $U_N=0$ and so $N$ lies in $\Vcal$. We may apply condition (3) of Theorem~\ref{Thm: mutations are wide} to $l$ and so we have that $K$ lies in $\Scal$. Then $h$ is an isomorphism, that is, we have $N = 0$. Condition (ii) in Definition~\ref{def:almost torsion} follows immediately from statement (3) of Theorem~\ref{Thm: mutations are wide}.

Conversely, if $B$ in $\Hcal$ is torsion-free, almost torsion for $\mathfrak{u}$ and belongs to $\Tcal$, then certainly $B$ lies in $\Scal$, and every non-zero subobject $K$ in $\Scal$ of $B$ satisfies that $B/K$ lies in $\Scal\subseteq\Vcal$, but also in $\Ucal$ by condition (i) in Definition~\ref{def:almost torsion}, hence $K=B$. This shows that $B$ is a simple object of $\Scal$, hence it belongs to $\Mcal$.
\end{proof}

Let $\Acal$ be  a Grothendieck category, and let $\Omega$ be the set of isoclasses of simple objects in $\Acal$. Given a subset $\Omega'$ of $\Omega$, we consider the torsion pair generated by $\Omega'$.  It has the shape $(\Filt{\Omega'}, (\Omega')^{\perp_0})$, cf.~\cite[Proposition~VIII.3.2]{Stenstrom}. Moreover, it is hereditary because, for every simple $S$ in $\Omega$ and every object $M$ in $\Gcal$, there exists a non-zero map $S \to E(M)$ if and only if  $S$ embeds in $M$, and so $(\Omega')^{\perp_0}$ is closed under injective envelopes.  By \cite[Lemma VIII.2.4]{Stenstrom}, the torsion pairs generated by subsets of $\Omega$ are precisely the hereditary torsion pairs of the form $(\Scal, \Rcal)$ with $\Scal \subseteq \Filt{\Omega}$; we call such a pair a \textbf{simple torsion pair}.  Observe that an object $M$ is contained in $\Filt\Omega$ if and only if every non-zero quotient of $M$ has a non-zero socle, cf.~~\cite[Proposition~VII.2.5]{Stenstrom}.  For the sake of the next result, we will say that a TTF class $\Fcal$ is a \textbf{simple TTF class} if $\Fcal = \Filt{\Omega'}$ for some set $\Omega' \subseteq \Omega$.

\begin{theorem}\label{Thm: f.d. mutations are simple}
Suppose we are in Setup \ref{setup: length}.  Let $\Mcal$ be the set of all torsion-free, almost torsion objects for $\mathfrak{u}$ which belong to $\Tcal$.  Let $\Ncal$ be the set of all torsion, almost torsion-free objects for $\mathfrak{t}$ which belong to $\Vcal$.  The following statements are equivalent.\begin{enumerate}
\item $\mathfrak{t}$ is a right mutation of $\mathfrak{u}$.
\item $\Scal=\filt{\Mcal}$ in $\Hcal$.
\item $\Scal=\filt{\Ncal}$ in $\Hcal$.
\item The pair $\overrightarrow{\mathfrak{s}}$ is a simple (hereditary) torsion pair in $\overrightarrow{\Hcal_{\mathfrak{u}^-}}$. 
\item The class $\overrightarrow{\Scal}[-1]$ is a simple TTF class in $\overrightarrow{\Hcal_{\mathfrak{t}^-}}$.
\end{enumerate}
\end{theorem}
\begin{proof}
The equivalence of the first three statements follows immediately from Theorem~\ref{Thm: mutations are wide}, Theorem~\ref{thm: ringel} and Lemma~\ref{lem: simple bricks = ME}.  

(2)$\Rightarrow$(4): First observe that, since $\Scal$ is extension-closed, condition (2) implies that $\Scal$ also coincides with $\filt{\Mcal}$ in $\Hcal_{\mathfrak{u}^-}$. Hence, we have that $\overrightarrow{\Rcal} = \Scal^{\perp_0} = \Mcal^{\perp_0}$ in $\overrightarrow{\Hcal_{\mathfrak{u}^-}}$, and the latter is the torsion-free class in a simple hereditary torsion pair in $\overrightarrow{\Hcal_{\mathfrak{u}^-}}$ by Remark~\ref{rem: tf/t are simple}. Thus, $\overrightarrow{\mathfrak{s}}$ is indeed a simple torsion pair in $\overrightarrow{\Hcal_{\mathfrak{u}^-}}$.

(4)$\Rightarrow$(5): We know from Theorem~\ref{thm: mutation general} that $\overrightarrow{\Scal}[-1]$ is a TTF class in $\overrightarrow{\Hcal_{\mathfrak{t}^-}}$. By assumption, there exists a set of simple objects $\Omega'$ in $\overrightarrow{\Hcal_{\mathfrak{u}^-}}$ such that $\overrightarrow{\Scal} = \Filt{\Omega'}$.  It follows easily from the definitions that the torsion, almost torsion-free objects for a hereditary torsion pair coincide with the torsion simple objects.  Thus, the objects in $\overrightarrow{\Scal}$ that are almost torsion-free coincide with $\Mcal$. By Remark \ref{rem: tf/t are simple}, we have that the objects $\Omega'[-1]$ are simple in $\overrightarrow{\Hcal_{\mathfrak{t}^-}} = (\overrightarrow{\Hcal_{\mathfrak{u}^-}})_{\overrightarrow{\mathfrak{s}}^-}$.  Since, considering the subcategory $\Filt{\Omega'}$ of $\overrightarrow{\Hcal_{\mathfrak{u}^-}}$ and the subcategory $\Filt{\Omega'[-1]}$ of $\overrightarrow{\Hcal_{\mathfrak{t}^-}}$, we have
$$\overrightarrow{\Scal}[-1] = (\Filt{\Omega'})[-1] = \Filt{\Omega'[-1]}$$ 
and, thus, $\overrightarrow{\Scal}[-1]$ is a simple TTF class in $\overrightarrow{\Hcal_{\mathfrak{t}^-}}$.

(5)$\Rightarrow$(1): This implication is immediate by Theorem~\ref{thm: mutation general}.  \end{proof}

Generalising results from~\cite{InTh} again, we obtain certain ``distinguished'' mutations of a torsion pair in a length category.

\begin{lemma}\label{lem:max wide intervals}
Suppose $\Hcal$ is a length category with a torsion pair $\mathfrak{v}=(\Xcal,\Ycal)$.
\begin{enumerate}
\item There is a torsion pair $\widecheck{\mathfrak{v}}=(\widecheck{\Xcal},\widecheck{\Ycal})$ in $\Hcal$ such that 
\[\widecheck{\Xcal}=\{X \in \Hcal \mid \text{ every } f \in \Hom_{\Hcal}( X, Y) \text{ with }  Y\in\Ycal \text{ has } \coker{f} \in\Ycal\}.\]
\item There is a torsion pair $\widehat{\mathfrak{v}}=(\widehat{\Xcal}, \widehat{\Ycal})$ in $\Hcal$ such that
\[\widehat{\Ycal}=\{Y \in \Hcal \mid \text{ every } f \in \Hom_{\Hcal}(X, Y) \text{ with }  X\in\Xcal \text{ has } \ker{f} \in\Xcal\}.\]
\end{enumerate}
Moreover, we have $\widehat{\Xcal}\subseteq\Xcal\subseteq\widecheck{\Xcal}$ and both $\widecheck{\Xcal}\cap\Ycal$ and $\Xcal\cap\widehat{\Ycal}$ are wide subcategories of $\Hcal$.
\end{lemma}

\begin{proof}
We prove only (1) and one half of the final statements. The others follow by a dual analogous argument.

It is easy to check that $\Xcal\subseteq\widecheck{\Xcal}$ and $\widecheck{\Xcal}$ is closed under quotients. To see that $\widecheck{\Xcal}$ is closed under extensions, consider an exact sequence $0\to X_1\to X\to X_2\to 0$ in $\Hcal$ such that $X_1,X_2\in\widecheck{\Xcal}$ and consider further a homomorphism $f\colon X\to Y$ with $Y\in\Ycal$. If we denote by $Y_1$ the image of the composition $X_1\to X\to Y$ and by $Y_2$ the cokernel of the same map, we obtain a commutative diagram with exact rows and columns
\[
\xymatrix{
0\ar[r] & X_1 \ar[r] \ar[d] & X\ar[r] \ar[d] & X_2\ar[r] \ar[d] & 0 \\
0\ar[r] & Y_1 \ar[r] \ar@{->>}[d] & Y\ar[r] \ar@{->>}[d] & Y_2\ar[r] \ar@{->>}[d] & 0 \\
0\ar[r] & C_1 \ar[r] & \coker{f} \ar[r] & C_2 \ar[r] & 0.
}
\]
Now $Y_1\in\Ycal$ since it is a subobject of $Y$ and $Y_2\in\Ycal$ since $X_1\in\widecheck{\Xcal}$. By inspecting the left and right columns, it follows that $C_1,C_2\in\Ycal$. As $\Ycal$ is closed under extensions, we also have $\coker{f}\in\Ycal$ and, hence, $X\in\widecheck{\Xcal}$. Since $\Hcal$ is a length category and $\widecheck{\Xcal}\subseteq\Hcal$ is closed under quotients and extensions, it is a torsion class in $\Hcal$.
Finally, the fact that $\widecheck{X}\cap\Ycal$ is a wide subcategory of $\Hcal$ follows by the argument for~\cite[Proposition 2.12]{InTh}.
\end{proof}
 

Combining the last two results with those from~\cite{AP} allows us to describe all right or left mutations of a torsion pair in terms of almost torsion objects or almost torsion-free objects, respectively. Moreover, we identify the ``distinguished'' mutations from Lemma~\ref{lem:max wide intervals} as ``extremal'' mutations of the given torsion pair.

\begin{corollary}
Suppose we are in Setup~\ref{setup: length}.
\begin{enumerate}
\item Let $\Mcal$ be a representative set of isomorphism classes of torsion-free, almost torsion objects for $\mathfrak{u}$.
Then the right mutations of $\mathfrak{u}$ bijectively correspond to subsets of $\Mcal$.
In particular, $\mathfrak{u}$ admits a proper right mutation if and only if there are torsion-free, almost torsion objects for $\mathfrak{u}$.
Moreover, if $\mathfrak{t}$ is a right mutation of $\mathfrak{u}$, then $\Ucal \subseteq \Tcal \subseteq \widecheck{\Ucal}$.
\item Let $\Ncal$ be a representative set of isomorphism classes of torsion, almost torsion-free objects for $\mathfrak{t}$.
Then the left mutations of $\mathfrak{t}$ bijectively correspond to subsets of $\Ncal$.
In particular, $\mathfrak{t}$ admits a proper left mutation if and only if there are torsion, almost torsion-free objects for $\mathfrak{t}$.
Moreover, if $\mathfrak{u}$ is a left mutation of $\mathfrak{t}$, then $\widehat{\Tcal}\subseteq\Ucal\subseteq\Tcal$.
\end{enumerate} 
\end{corollary}
\begin{proof}
We prove (1); the argument for (2) is dual. 
To start with, note that $\Mcal$ is a semibrick by Remark~\ref{rem: tf/t are simple}(1) (as it is a set of pairwise non-isomorphic simple objects in $\Hcal_{\mathfrak{u}^-}$), and so is any subset of~$\Mcal$.

The assignment between right mutations and subsets of $\Mcal$ can be described as follows. Given a right mutation $\mathfrak{t}=(\Tcal,\Fcal)$ of $\mathfrak{u}=(\Ucal,\Vcal)$, the intersection $\Scal=\Tcal\cap\Vcal$ is a wide subcategory of $\Hcal$ by Theorem~\ref{Thm: mutations are wide}, and so is of the form $\Scal=\filt{\Mcal'}$ for a unique subset $\Mcal'\subseteq\Mcal$ by Theorem~\ref{Thm: f.d. mutations are simple} (recall that simply $\Mcal'=\Mcal\cap\Scal$ by Theorem~\ref{thm: ringel}). This assignment is injective since one can recover $\Tcal$ from $\Mcal'$ as $\Tcal=\Ucal\star\filt{\Mcal'}$, see Proposition~\ref{prop: tp vs triples}.

On the other hand, if $\Mcal'\subseteq\Mcal\subseteq\Vcal$ is any subset, then $\Mcal'\subseteq\widecheck{\Ucal}$ by the very definition of almost torsion objects for $\mathfrak{u}$. Following~\cite[\S6]{AP}, we denote by $\Wcal_r(\Vcal):=\widecheck{\Ucal}\cap\Vcal$ the wide subcategory of $\Hcal$ obtained by applying Lemma~\ref{lem:max wide intervals} to $\mathfrak{u}$. Note that $\Wcal_r(\Vcal)$ is also a wide subcategory of $\Hcal_{\mathfrak{u}^-}$ by Lemma~\ref{lem: wide},  hence $\Mcal'$ becomes a set of simple objects in $\Wcal_r(\Vcal)$. As $\Wcal_r(\Vcal)$ is necessarily an abelian length category, $\Scal:=\filt{\Mcal'}$ is torsion class of a hereditary torsion pair in $\Wcal_r(\Vcal)$. Now it follows from~\cite[Theorems 4.2 and 6.6]{AP} that $\Tcal:=\Ucal\star\Scal$ is a torsion class in $\Hcal$ such that $\Vcal\cap\Tcal=\Scal=\filt{\Mcal'}$. Hence, the assignment from the previous paragraph is also surjective.
\end{proof}

\subsection{Irreducible mutations}  In this final subsection we consider irreducible mutations of torsion pairs.  The notation $\Ind{\sigma}$ and $\Ind{\sigma'}$ is used for the isoclasses of indecomposable objects in $\Prod{\sigma}$ and $\Prod{\sigma'}$ respectively. 

\begin{definition}  Suppose $\sigma$ and $\sigma'$ are pure-injective cosilting objects in a compactly generated triangulated category $\Dcal$ and let $\Escr$ be the class $\Prod{\sigma}\cap \Prod{\sigma'}$. Suppose $\sigma'$ is a right mutation of $\sigma$ or, equivalently, that $\sigma$ is a left mutation of $\sigma'$ (see Corollary~\ref{cor:inverse mutations}). We will say that $\sigma'$ is an \textbf{irreducible right mutation} of $\sigma$ if $|\Ind{\sigma}\setminus \Ind\Escr| = 1$.  We will say that $\sigma$ is an \textbf{irreducible left mutation} of $\sigma'$ if $|\Ind{\sigma'}\setminus \Ind\Escr| = 1$.
\end{definition}  

Recall that, by Proposition \ref{prop: bij ind summands}, there is a bijection between $\Ind{\sigma}\setminus \Ind\Escr$ and $\Ind{\sigma'}\setminus\Ind\Escr$ and so $\sigma'$ is an irreducible right mutation of $\sigma$ if and only if $\sigma$ is an irreducible left mutation of $\sigma'$. \\

\noindent\textbf{Notation:} Within Setup \ref{setup: length}, we fix some further notation that we use in the remainder of the section.
\begin{itemize}
\item Denote by $\sigma_\mathfrak{u}$ and $\sigma_\mathfrak{t}$ the cosilting objects in $\D R$ such that $\Hcal_{\sigma_{\mathfrak{u}}}=\overrightarrow{\Hcal_{\mathfrak{u}^-}}$ and $\Hcal_{\sigma_{\mathfrak{t}}}=\overrightarrow{\Hcal_{\mathfrak{t}^-}}$. 

\item In the case where $\mathfrak{t}$ is a right mutation of $\mathfrak{u}$, let $\Mcal$ be the semibrick associated to $\Scal$. For each $M$ in $\Mcal$, we know from Remark \ref{rem: tf/t are simple} and Lemma~\ref{lem: simple bricks = ME} that 
\begin{itemize}
\item $M$ is a simple object in $\overrightarrow{\Hcal_{\mathfrak{u}^-}}$; we denote by $\sigma_M$ the object in $\Ind{\sigma_\mathfrak{u}}$ such that $H^0_{\sigma_\mathfrak{u}}(\sigma_M)$ is the injective envelope of $M$ in the locally coherent Grothendieck category ${\Hcal_{\sigma_\mathfrak{u}}}$. 
\item $M[-1]$ is a simple object in $\overrightarrow{\Hcal_{\mathfrak{t}^-}}$ and we similarly denote by $\sigma_{M[-1]}$ the object in  $\Ind{\sigma_\mathfrak{t}}$ such that $H^0_{\mathfrak{t}^-}(\sigma_{M[-1]})$ is the injective envelope of $M[-1]$ in the locally coherent Grothendieck category ${\Hcal_{\sigma_\mathfrak{t}}}$.
\end{itemize}
\end{itemize}

\begin{lemma}\label{lem: indecomp mutation}
Suppose we are in Setup \ref{setup: length} and that $\mathfrak{t}$ is a right mutation of $\mathfrak{u}$. Let ${\sigma_\mathfrak{u}}$ and $\sigma_\mathfrak{t}$ be a cosilting object associated to $\mathfrak{u}$ and $\mathfrak{t}$ respectively and consider $\Escr =\Prod{\sigma_\mathfrak{u}} \cap \Prod{\sigma_\mathfrak{t}}$.  Then: \begin{enumerate}
\item the set $\Ind{\sigma_\mathfrak{u}}\setminus\Ind\Escr$ coincides with $\{\sigma_M \mid M \in \Mcal\}$; and
\item the set $\Ind{\sigma_\mathfrak{t}}\setminus\Ind\Escr$ coincides with $\{\sigma_{M[-1]} \mid M \in \Mcal\}$.
\end{enumerate}
\end{lemma}
\begin{proof}
(1) By Theorems~\ref{thm:mutation as HRS-tilt}(2) and \ref{Thm: f.d. mutations are simple}, we have that 
$$(\overrightarrow{\Scal},\overrightarrow{\Rcal})=({}^{\perp_0}H_{\sigma_{\mathfrak{u}}}^0(\Escr),\Cogen{H_{\sigma_{\mathfrak{u}}}^0(\Escr)})=(\Filt \Mcal, \Mcal^{\perp_0})$$
is a hereditary torsion pair of finite type in $\Hcal_{\sigma_\mathfrak{u}}$. By Theorem \ref{thm: Krause Herzog}
\[\mathcal O:=\{E\in\Spec{\Hcal_{\sigma_\mathfrak{u}}}\mid E\notin\Mcal^{\perp_0}\}\] 
is the associated open set in $\Spec{\Hcal_{\sigma_\mathfrak{u}}}$, and it clearly consists of the injective envelopes of the simples from $\Mcal$. In other words, we have that
$\mathcal O=\{H^0_{\sigma_\mathfrak{u}}(\sigma_M)\mid M\in \Mcal\}$. Since $H^0_{\sigma_\mathfrak{u}}(\Escr)$ is the class of torsion-free injective objects of $\Hcal_{\sigma_\mathfrak{u}}$, it follows that $H^0_{\sigma_\mathfrak{u}}$ induces a bijection between $\Ind{\sigma_\mathfrak{u}}\setminus\Ind\Escr$ and $\mathcal O$, and the claim is proven.

(2) It follows from the proof of Theorem \ref{Thm: f.d. mutations are simple}, Remark \ref{rem: tf/t are simple} and the fact that direct limits in both $\overrightarrow{\Hcal_{\mathfrak{t}^-}}$ and $\overrightarrow{\Hcal_{\mathfrak{u}^-}}$ are directed homotopy colimits (see \cite[Corollary 5.8]{SSV}) that $\overrightarrow{\Scal[-1]}=\overrightarrow{\Scal}[-1]=\Filt{\Ncal[-1]}$ in $\Hcal_{\mathfrak{t}^-}$. Recall that, by Lemma \ref{lem: s restricts}, $\overrightarrow{\mathbb{T}_{\mathfrak u^-}}$ is the left HRS-tilt of $\overrightarrow{\mathbb{T}_{\mathfrak t^-}}$ at the torsion pair $\overrightarrow{\mathfrak r}=(\overrightarrow{\Rcal},\overrightarrow{\Scal[-1])}$. Since $\mathfrak{u}$ is a left mutation of $\mathfrak{t}$, it follows from Theorem~\ref{thm:mutation as HRS-tilt}(1) that $\overrightarrow{\Scal[-1]}={}^{\perp_0}H_{\sigma_\mathfrak{t}}^0(\Escr)$. The same arguments as (1) yield that $\Ind{\sigma_\mathfrak{t}}\setminus\Ind\Escr$ coincides with $\{\sigma_{M[-1]} \mid M \in \Mcal\}$.
\end{proof}

\begin{remark}
It follows from Lemma \ref{lem: indecomp mutation} that right mutation of a torsion pair $\mathfrak{u}$ within Setup \ref{setup: length} consists of removing indecomposable summands of an associated cosilting object $\sigma_\mathfrak{u}$ in $\D R$ and replacing them with new ones.  Indeed, since $\sigma_M$ corresponds to the injective envelope of a simple object in $\Hcal_{\sigma_\mathfrak{u}}$ for every $M$ in $\Mcal$, it follows that $\sigma_M$ is a direct summand of every cosilting object equivalent to $\sigma_\mathfrak{u}$.  Similarly, for every $M$ in $\Mcal$, the indecomposable object $\sigma_{M[-1]}$ is a direct summand of every cosilting object corresponding to $\mathfrak{t}$. 
\end{remark}

As a corollary of the results above we are able to characterise minimal inclusions of torsion classes in length hearts $\Hcal$ (as in Setup \ref{setup: length}) in terms of irreducible mutations of the associated cosilting objects. Recall that if the inclusion $\Ucal\subseteq \Tcal$ is proper, it is said to be a \textbf{minimal inclusion of torsion classes} if for any other torsion class $\Xcal$ of $\Hcal$, if $\Ucal\subseteq \Xcal\subseteq \Tcal$ then either $\Ucal=\Xcal$ or $\Tcal=\Xcal$.

\begin{remark}\label{rmk: BCZ} If in Setup \ref{setup: length} we have $\Hcal=\mod A$ for a finite-dimensional algebra $A$, it is shown in \cite[Theorem 2.8]{BCZ} that if $\Ucal\subseteq \Tcal$ is a minimal inclusion of torsion classes, then $\Tcal$ can be built by adjoining to $\Ucal$ an indecomposable module satisfying certain properties. This module turns out to be precisely the unique torsion-free, almost torsion module for the torsion pair $\mathfrak{u}$, as shown in \cite{S}. It can easily be checked that these arguments hold also for an arbitrary length category $\Hcal$, by replacing the notion of dimension by length where necessary.
\end{remark}

\begin{corollary}\label{arrows} Suppose we are in Setup \ref{setup: length}.   The following statements are equivalent.
\begin{enumerate}
\item ${\sigma_{\mathfrak t}}$ is an {irreducible right mutation} of $\sigma_{\mathfrak{u}}$.
\item The class $\Scal$ coincides with $\filt{M}$ for a brick $M$ in $\Hcal$.
\item The inclusion $\Ucal\subseteq \Tcal$ is a minimal inclusion of torsion classes.
\end{enumerate}
\end{corollary}
\begin{proof}
(2)$\Rightarrow$(1):  By Theorem \ref{thm: ringel}, we have that $\Scal$ is a wide subcategory of $\Hcal$ and so, by Theorem \ref{Thm: mutations are wide}, we have that $\sigma_{\mathfrak t}$ is a right mutation of $\sigma_\mathfrak{u}$. It follows from Lemma \ref{lem: simple bricks = ME} that $\Mcal=\{M\}$. By Lemma \ref{lem: indecomp mutation}, we have that $\sigma_{\mathfrak t}$ is an irreducible right mutation of $\sigma_\mathfrak{u}$.

(3)$\Rightarrow$(2):  By Remark \ref{rmk: BCZ}, there is a torsion-free, almost torsion object for $\mathfrak{u}$ in $\Scal$. By Remark \ref{rem: tf/t are simple}, it follows that $M$ is a simple object in $\Hcal_{\mathfrak{u}^-}$.  Thus $\Scal':=\filt{M}$ is a non-trivial torsion class in $\Hcal_{\mathfrak{u}^-}$ that is contained in $\Scal$. By Proposition \ref{prop: HRS-tilt triangle},  there exists a torsion class $\Ucal \subsetneq \Tcal' \subseteq \Tcal$ in $\Hcal$.  By assumption, we have that $\Tcal = \Tcal'$ and so, by another application of Proposition \ref{prop: HRS-tilt triangle}, we conclude that $\Scal = \Scal'$.  

(1)$\Rightarrow$(3):  Suppose (1) holds and consider a torsion class $\Xcal$ in $\Hcal$ such that $\Ucal \subseteq \Xcal \subseteq \Tcal$.  We must show that $\Ucal=\Xcal$ or $\Tcal=\Xcal$. By Proposition \ref{prop: HRS-tilt triangle}, this implies that there exists a torsion class $\Scal' \subseteq \Scal \subseteq \Vcal$ in $\mathcal H_{\mathfrak u^-}$ (given by $\Xcal\cap \Vcal$) and, by assumption, we have that $\Scal = \filt{S}$ for the unique (up to iso) simple object $S$ in $\Scal$. If $\Scal'$ is trivial, then by Proposition \ref{prop: HRS-tilt triangle}, $\Ucal=\Xcal$. Suppose that $\Scal'$ is non-trivial and let $X$ be a non-zero object of $\Scal'$. Since $\Scal'$ is contained in $\filt{S}$, $S$ is a quotient of $X$ and, thus, $S$ lies in $\Scal'$. This shows that $\Scal'=\Scal$ and, thus, that $\Tcal=\Xcal$, again by Proposition \ref{prop: HRS-tilt triangle}.
\end{proof}

\begin{example}\label{expl:krone3}
We revisit Example~\ref{expl:krone}. This time we consider the indecomposable  preprojective modules $P_n, n\in\mathbb N$ over the Kronecker algebra $A$ and the torsion pairs $\mathfrak{t}_n=({}^{\perp_0}P_n,\Cogen{P_n})$ cogenerated by them. It is well known that $\mathfrak{t}_n$ is an irreducible mutation of $\mathfrak{t}_{n+1}$, the corresponding wide subcategory is $^{\perp_0}P_n\cap\Cogen{P_{n+1}}\cap \mod A=\add {P_{n+1}}$. Notice that 
$^{\perp_0}P_n\cap\Cogen{P_{n+2}}\cap \mod A=\add {P_{n+1}\oplus P_{n+2}}$ is not wide, so
$\mathfrak{t}_n$ is not a mutation of $\mathfrak{t}_{n+2}$. This shows that a sequence of irreducible mutations is not a mutation in general.

 We can also rediscover the fact that  $\sigma_\mathbb{X}$ does {not} admit  right mutation. Indeed,     $\sigma_\mathbb{X}$ is associated with the torsion pair $\mathfrak{u}=(\Gen \tube, \Vcal)$ generated by all finite-dimensional indecomposable regular modules, so  any torsion pair $\mathfrak{t}=(\Tcal,\Fcal)$ in $\Cosilt A$ lying above   $\mathfrak{u}$  has the form $\mathfrak{t}=\mathfrak{t}_{n}$ for some $n$, and
 $\Tcal\cap\Vcal\cap\mod A=\add {P_{n+1}\oplus P_{n+2}\oplus\ldots}$ is clearly not wide.
 
 Finally, we remark that the set $\Escr$ in Lemma~\ref{lem: indecomp mutation} may differ from $\Prod{\{\sigma_M\mid M\in \Omega\setminus \Mcal\}}$ where $\Omega$ is the set of isoclasses of simple objects in $\Hcal_{\mathfrak{u}^-}$. To this end, we consider  $\sigma_P$ with $P=\mathbb{X}\setminus\{x\}$ for some $x\in\mathbb{X}$. Here $\Mcal$ only contains  the simple regular module $S$ corresponding to the tube $\tube_x$, and  the set $\Omega$ consists of the Pr\"ufer module $S_\infty$ and
 the adic modules corresponding to $\tube_P=\bigcup_{y\not=x}\tube_y$. Thus  the generic module $G$ belongs to the
 set $\Escr$, but not to $\Prod{\{\sigma_M\mid M\in \Omega\setminus \Mcal\}}$.
\end{example}


\end{document}